\documentclass[12pt]{article}

\usepackage[dvipsnames]{xcolor}
\usepackage{enumerate,enumitem}

\usepackage{amsmath,amssymb,amsfonts,mathrsfs, amsthm,mathtools, bm, bbm, dsfont, mathrsfs, amsthm}
\usepackage{graphicx, epstopdf}

\usepackage{fullpage}

\usepackage{ebgaramond}
\usepackage[OT1]{fontenc}
\usepackage[utf8]{inputenc}

\usepackage[colorlinks=true,breaklinks=true,bookmarks=true,urlcolor=MidnightBlue,citecolor=MidnightBlue,linkcolor=MidnightBlue,bookmarksopen=false,draft=false]{hyperref}

\usepackage{tikz}
\usepackage{amsfonts}
\usepackage{scrextend}
\usepackage{amsfonts}
\usepackage{amsmath}
\usepackage{amssymb}
\usepackage{enumitem}
\usepackage{mathtools}
\usepackage{amsthm}
\usepackage{breqn}
\usepackage{dsfont}
\usepackage{amsmath}
\usepackage{amsfonts}
\usepackage{amssymb}
\usepackage{xcolor}

\usetikzlibrary{fit}
\hypersetup{
colorlinks=true,       
linkcolor=blue,          
citecolor=blue,        
filecolor=blue,      
urlcolor=blue 
}

\newcommand{\indep}{\perp \!\!\! \perp}
\newcommand{\LWS}{L_{{{\sf WS}}}}
\newcommand{\LRM}{L_{{{\sf RM}}}}

\newcommand{\LCON}{L_{{{\sf CO}}}^N}
\newcommand{\LEXN}{L_{{{\sf EX}}}^N}
\newcommand{\LCOSN}{L_{{{\sf CO, SYM}}}^N}
\newcommand{\LCO}{L_{{{\sf CO}}}}
\newcommand{\LEX}{L_{{{\sf EX}}}}
\newcommand{\LCOS}{L_{{{\sf CO, SYM}}}}
\newcommand{\PN}{$\mathcal{P}_{N}$}
\newcommand{\PIN}{$\mathcal{P}_{\infty}$}

\newcommand{\LPRS}{L_{{{\sf PR, SYM}}}}

\newcommand{\N}{\mathcal{N}}

\newcommand{\sina}[1]{{\color{blue} #1}}

\allowdisplaybreaks

\newtheorem{theorem}{Theorem}[]
\theoremstyle{lemma}
\newtheorem{lemma}{Lemma}[]
\theoremstyle{proposition}
\newtheorem{proposition}{Proposition}[]
\theoremstyle{definition}
\newtheorem{definition}{Definition}[]

\theoremstyle{assumption}
\newtheorem{assumption}{Assumption}[]

\theoremstyle{remark}
\newtheorem*{remark}{Remark}
\begin{document}

\title{Decentralized Exchangeable Stochastic Dynamic Teams in Continuous-time, their Mean-Field Limits and Optimality of Symmetric Policies}

\author{Sina Sanjari \quad Naci Saldi \quad Serdar Y\"uksel\thanks{Sina Sanjari is with the department of Mathematics and Computer Science at the Royal Military College of Canada. Naci Saldi is with the department of Mathematics at Bilkent University, Turkey. Serdar Y\"uksel is with the department of Mathematics and Statistics at Queen's University, Canada. Emails: \{sanjari@rmc.ca, naci.saldi@bilkent.edu.tr, yuksel@queensu.ca\}.}}

\maketitle

\begin{abstract}
We study a class of stochastic exchangeable teams comprising a finite number of decision makers (DMs) as well as their mean-field limits involving infinite numbers of DMs. In the finite population regime, we study exchangeable teams under the centralized information structure. For the infinite population setting, we study exchangeable teams under the decentralized mean-field information sharing. The paper makes the following main contributions: i) For finite population exchangeable teams, we establish the existence of a randomized optimal policy that is exchangeable (permutation invariant) and Markovian; ii) As our main result in the paper, we show that a sequence of exchangeable optimal policies for finite population settings converges to a conditionally symmetric (identical), independent, and decentralized randomized policy for the infinite population problem, which is globally optimal for the infinite population problem. This result establishes the existence of a symmetric, independent, decentralized optimal randomized policy for the infinite population problem. Additionally, this proves the optimality of the limiting measure-valued MDP for the representative DM; iii) Finally, we show that symmetric, independent, decentralized optimal randomized policies are approximately optimal for the corresponding finite-population team with a large number of DMs under the centralized information structure. Our paper thus establishes the relation between the controlled McKean-Vlasov dynamics and the optimal infinite population decentralized stochastic control problem (without an apriori restriction of symmetry in policies of individual agents), for the first time, to our knowledge. 
\end{abstract}

\section{Introduction}

Decentralized stochastic control or stochastic teams study a collection of decision makers (DMs) (or agents or controllers) acting together to optimize a common cost function but not necessarily having access to the same information. Each DM over time has partial access to the global information which is determined by the \emph{information structure} of the team \cite{wit75}. This is in contrast with stochastic games, where the DMs may have conflicting cost criteria or probabilistic models or priors on the system. Stochastic teams \cite{ho1980team, CDCTutorial} generalize classical single DM stochastic control problems with applications in many fields such as networked control \cite{ho1980team, hespanha2007survey, YukselBasarBook}, communication networks \cite{hespanha2007survey}, cooperative systems \cite{mar55, Radner}, large sensor networks \cite{tsitsiklis1988decentralized}, and electricity markets and smart grid design \cite{davison1973optimal}. 

In the context of stochastic games, an important set of results involves those on mean-field (MF) games: MF games are limiting models of weakly-interacting symmetric finite DM games (see e.g., \cite{CainesMeanField2,LyonsMeanField,carmona2018probabilistic}). The existence of a symmetric Nash equilibrium for MF games has been established in several papers; see e.g., \cite{LyonsMeanField,bardi2019non, carmona2016mean,lacker2015mean,saldi2018markov}. In one direction, the mean-field approach designs policies such that Nash equilibria for MF games are shown to be approximately Nash equilibria for the corresponding (pre-limit) games with large numbers of DMs; see e.g., \cite{CainesMeanField3, saldi2018markov, carmona2018probabilistic, cecchin2017probabilistic}. In the opposite direction, the limits of symmetric Nash equilibrium policies of finite DM games are shown to converge to a Nash equilibrium of the corresponding MF game as the number of DM goes to infinity; see e.g., \cite{fischer2017connection, lacker2018convergence, LyonsMeanField}. 


In contrast to the setting of MF games, we study continuous-time stochastic MF teams under local information. Our main goal is to establish the existence and structural results (such as symmetry) for a globally optimal solution for continuous-time exchangeable stochastic with a large but finite number of DMs as well as their infinite population MF teams. Such existence and structural results are important for developing computational, approximation, and learning methods (see e.g., \cite{carmona2023model,yongacoglu2022independent,pham2016discrete, Ali-2023, SSYMFsharing}). Our approach differs from the methods used in MF games. The distinction arises primarily because fixing policies of certain DMs to symmetric policies and utilizing fixed-point theorems only results in the existence of a symmetric Nash equilibrium or person-by-person optimal solution, not a globally optimal solution. Unlike games, where Nash equilibrium is the primary focus, in teams a globally optimal solution is the main objective. This is because Nash equilibrium or person-by-person optimal solutions, in the context of teams, often correspond to only locally optimal solutions. Hence, the existence and structural results for games may be inconclusive regarding global optimality for teams without uniqueness (e.g., see \cite{LyonsMeanField, hajek2019non, bayraktar2020non,bardi2019non,delarue2020selection} for non-uniqueness of Nash equilibrium). This gap between global optimality and person-by-person optimality is especially significant for stochastic teams with countably infinite number of DMs (or MF teams), as the deviation of a single DM typically has no impact on the team's overall performance. Consequently, one can anticipate that large weakly-interacting teams may have multiple local person-by-person optimal solutions. 

Additionally, the absence of centralization in the information structure for teams adds an extra layer of mathematical challenge for establishing the results for globally optimal solutions. In particular, the methods from the classical stochastic control theory do not directly apply to the decentralized setting (see e.g., Wistsenhausen's counterexample \cite{WitsenhausenCounter} for an intriguing counterexample). On the other hand, the results on convex discrete-time teams with a decentralized information structure in \cite{Radner, KraMar82, HoChu} apply only to static teams with finitely many DMs. These challenges carry over to the continuous-time setup as well; in \cite{pradhan2023controlled}, for continuous-time teams with a finite number of DMs both in centralized and decentralized settings, the existence and approximation of optimal solutions by discrete-time approximations have been established. 

A notable set of results related to MF teams are those on social optima control problems, MF optimal control problems, and their limiting problems of McKean-Vlasov optimal control (see e.g., \cite{motte2022mean,carmona2023model,cecchin2021finite,carmonaForwardbackward,pham2017dynamic,bensoussan2015master,bayraktar2018randomized,djete2022mckean,lacker2017limit,motte2023quantitative,fornasier2019mean,albi2017mean,albi2022moment,lauriere2016dynamic,achdou2020mean,Mahajan-CDC14}), with particular emphasis on the linear quadratic Gaussian model (see e.g., \cite{huang2012social,carmona2018probabilistic, caines2018peter, pham2016discrete,elliott2013discrete,ni2015discrete,toumi2024mean,arabneydi2015team}).  Within this framework, characterizing the bound for the discrepancy between centralized social performance and decentralized controllers has been a central focus under open-loop, relaxed, and Markovian feedback policies. Results on the McKean-Vlasov optimal control type or mean-field type problem assume symmetry in policies among DMs to make a connection between the finite population centralized team problem, the infinite population problem, and the optimal control problem of the representative agent. For centralized discrete-time MF teams with a finite population of agents, an equivalent Markov decision problem (MDP) formulation is characterized in \cite{bauerle2023mean}. Additionally, it has been shown that the value functions for the finite population MDP converge to the value function of a limiting MDP as the number of DMs goes to infinity. building on the MDP formulation in \cite{bauerle2023mean}, the existence of a symmetric optimal solution with MF information sharing information structure is established in \cite{SSYMFsharing} for discrete-time MF teams. Related to \cite{bauerle2023mean}, in \cite{carmona2023model} dynamic programming has been established for the discrete-time McKean-Vlasov problem with a representative agent using a lifted measure-valued MDP. It is assumed apriori that policies are symmetric in the infinite population limit (see \cite[Section 2]{carmona2023model}). Additionally in \cite{motte2022mean}, the connection between the finite population problem and the limiting MDP for the discrete-time McKean-Vlasov MDP has been established under the assumption that the policies are open-loop, symmetric, and decentralized. In this paper, we consider decentralized information structure for finite population teams and their infinite population MF teams. Furthermore, in contrast to the results in the literature, we do not assume symmetry in policies apriori. Instead, we rigorously show that there exists a symmetric optimal solution for the infinite population MF problem, certifying the optimality of a commonly assumed McKean-Vlasov formulation.

Discrete-time MF teams under decentralized information structure have been studied in \cite{sanjari2018optimal, sanjari2019optimal, SSYdefinetti2020}. For such models, the existence and convergence of a globally optimal solution have been established in  \cite{sanjari2018optimal, sanjari2019optimal, SSYdefinetti2020}. Additionally, a class of discrete-time exchangeable games among teams with infinite DMs has been studied in \cite{sanjari2024nash}. 
 Our results extend the aforementioned results for discrete-time teams to the continuous-time ones. Continuous time setting entails additional analysis and technical arguments compared to discrete-time prior works \cite{sanjari2018optimal, sanjari2019optimal, SSYdefinetti2020}.
 
 In a recent paper \cite{jackson2023approximately}, non-asymptotic bounds have been obtained between the optimal performance under centralized and strictly decentralized information structures for a class of continuous-time convex teams where optimal policies can be characterized by stochastic maximum principle and a Hamilton-Jacobi-type equation.  
 Our approach is complementary in that we do not approach the problem from a centralized agent's perspective (and thus cannot consider the maximum principle or the HJB optimality since we do not have a natural {\it state} to carry out such optimality equations), and throughout consider decentralized optimal policies for both finite and infinitely many DM setups. We establish optimality properties under strict decentralized information.  

We classify our analysis into problems for which the dynamics of agents are decoupled or coupled. For the former case, we also consider the setups where the problem is convex or only exchangeable; for the latter convexity is not assumed since convexity is typically incompatible with decentralized information when there is coupling through dynamics, since the agents may have an incentive to signal through their actions \cite{YukselBasarBook}. The decentralized information structure under the decoupled dynamics in \eqref{eq:st-dy-N} is a specific example of a class of \emph{partially-nested decentralized information structures} (see \cite{HoChu,YukselBasarBook}), where convexity properties are less restrictive \cite{saldi2022geometry,YukselSaldiSICON17} (as there is no signaling via actions between agents). We note that this is not the case for the decentralized information structure with the coupled dynamics in \eqref{eq:st-dy-C-intro}, which belongs to the class of \emph{non-classical decentralized information structures} for which convexity in policies typically never occurs; even if the running cost function is convex, the presence of the non-classical information structure results in non-convex expected cost in policies (as demonstrated in \cite{YukselSaldiSICON17} and the celebrated counter-example by Witsenhausen \cite{WitsenhausenCounter}). 

In our paper, we will consider both the convex and the non-convex setup. The models we study will consist of the following two types: For the decoupled setup, we consider finite horizon models of the form
\begin{align}\label{eq:st-dy-NDecoupled}
    dX_{t}^i = b_{t}(X_{t}^i,U_{t}^i)dt + \sigma_t(X_{t}^{i})dW_{t}^{i}, \quad t\in \mathbb{T}:=[0,T], \quad i \in {\cal N}
\end{align}
where ${\cal N}$ is either finite or countably infinite. 
For the coupled setup, we will consider systems of the form
\begin{align}\label{eq:st-dy-C-intro}
    dX_{t}^i = b_{t}\left(X_{t}^i,U_{t}^i, \frac{1}{N}\sum_{p=1}^{N} \delta_{X^{p}_t},\frac{1}{N}\sum_{p=1}^{N} \delta_{U^{p}_t}\right)dt + \sigma_t(X_{t}^{i})dW_{t}^{i}, \quad t\in \mathbb{T},
\end{align}
for all $i\in \mathcal{N}$, where $\delta_{\{\cdot\}}$ is the Dirac delta measure. Our analysis below will start with the first model above, and we will start studying the second model with coupled dynamics in Section \ref{sec:coup}. 

The cost criteria will be presented in detail in the corresponding sections. However, to give a high level flavor, the cost will be of the following form for the finite population problems
\begin{align}\label{eq:sec1-2-intro}
    \int_{0}^{T} \frac{1}{N}\sum_{i=1}^{N}\hat{c}\left(X^{i}_t,U^{i}_t,\frac{1}{N}\sum_{p=1}^{N} \delta_{X^{p}_t},\frac{1}{N}\sum_{p=1}^{N} \delta_{U^{p}_t} \right) dt
\end{align}
for some Borel measurable function $\hat{c}$. We consider the expectation of the above cost for the finite population team, and the limsup of its expectation as $N\to \infty$ for the infinite population regime. The policies considered will be decentralized, with appropriate topologies defined in their corresponding spaces in the following sections.

We emphasize that the information structure is decentralized; that is, each DM only uses locally measurable policies and their relaxed generalizations, unlike much of the prior literature on such problems. The paper makes the following main contributions: 

\begin{enumerate}
\item Under a decentralized information structure, for convex exchangeable teams with a finite number of DMs under the model \eqref{eq:st-dy-NDecoupled}, in Theorem \ref{the:1}--Theorem \ref{the:3}, we establish the existence of a globally optimal solution, and show that it is Markovian and symmetric (identical among DMs). For convex exchangeable teams, in  Theorem \ref{the:4}, we prove that a sequence of globally optimal solutions for teams with $N$ DMs converge (in a suitable sense) to a randomized optimal solution of the mean-field limit as $N$ goes to infinity. This establishes the existence of a (possibly randomized) globally optimal solution for infinite population mean-field teams and shows that it is Markovian and symmetric.
\item For non-convex exchangeable teams with a finite number of DMs, in Theorem \ref{the:5} for decoupled dynamics \eqref{eq:st-dy-NDecoupled} and in Theorem \ref{the:5-coup} for coupled dynamics \eqref{eq:st-dy-C-intro}, we establish the existence of a randomized globally optimal solution and show that it is exchangeable (the joint distribution is permutation invariant) which might not be necessarily symmetric. 
\item For the infinite population non-convex teams, in Theorem \ref{the:6} for decoupled dynamics and in Theorem \ref{the:6-coup} for coupled dynamics, we establish the existence of a globally optimal solution and show that it is privately randomized and symmetric.
\item In Theorem \ref{the:7} for non-convex teams with decoupled dynamics and Theorem \ref{the:7-coup} for non-convex teams with coupled dynamics, we establish that a symmetric globally optimal solution for the infinite population mean-field team is approximately optimal for the corresponding finite-population exchangeable team with a large number of DMs.
\end{enumerate}

The main contributions of the paper are summarized in the following figure.

 
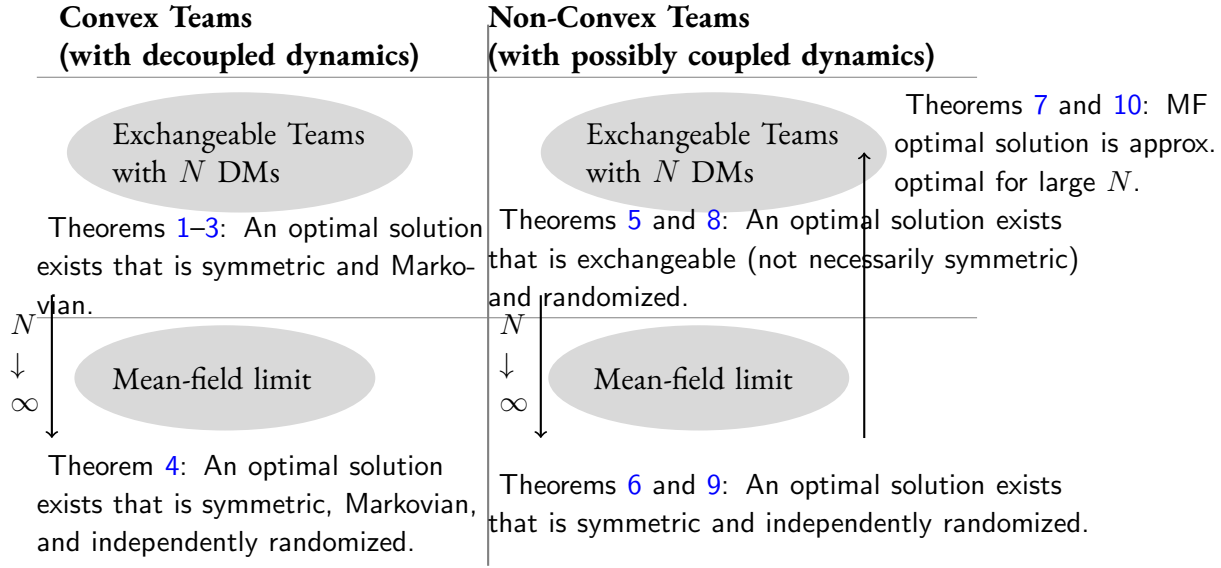
\begin{center}
    \begin{tikzpicture}[
     very thin]
    \fill[black!15] (0.2,0) ellipse (2.3 and 0.8);
    \fill[black!15] (0,-3) ellipse (2 and 0.7);
    \draw[gray, thick] (3.5,1.7) -- (3.5,-5.5);
     \draw[gray] (-2.5,-2.2) -- (10,-2.2);
    \draw[gray] (-2.5,1) -- (10,1);
    \node[text width=6cm] at (0.8,1.5) 
    {\bf Convex Teams\\ (with decoupled dynamics)};
    \node[text width=4cm] at (0.5,0) 
    {Exchangeable Teams with $N$ DMs};
    \node[text width=4cm] at (0.5,-3) 
    {Mean-field limit};
    \node[black,text width=6cm] at (0.5,-1.5) 
    {\small{ \sf Theorems \ref{the:1}--\ref{the:3}: An optimal solution exists that is symmetric and Markovian.}};
    \node[black,text width=6cm] at (0.5,-4.7) 
    {\small{ \sf Theorem \ref{the:4}: An optimal solution exists that is symmetric, Markovian, and independently randomized.}};
    \node[black,text width=0.1cm] at (-2.8,-2.8) 
    {\small{$N$ $\downarrow\infty$}};
    \draw[black, thick, ->] (-2.3,-1.9) -- (-2.3,-3.8);
\fill[black!15] (6.3,-3) ellipse (2 and 0.7);
    \fill[black!15] (6.5,0) ellipse (2.3 and 0.8);
        \draw[black, thick, ->] (4.2,-1.9) -- (4.2,-3.8);
\draw[black, thick, ->] (8.5,-3.8) -- (8.5,0);
\node[black,text width=0.1cm] at (3.7,-2.8) 
    {\small{$N$ $\downarrow\infty$}};
\node[black,text width=4.6cm] at (11.2,0.1) 
    {\small{ \sf Theorems \ref{the:7} and \ref{the:7-coup}:  MF optimal solution is approx. optimal for large $N$.}};
     \node[text width=7cm] at (7,1.5) 
    {\bf Non-Convex Teams\\ (with possibly coupled dynamics)};
    \node[text width=4cm] at (6.8,0) 
    {Exchangeable Teams with $N$ DMs};
    \node[text width=4cm] at (6.9,-3) 
    {Mean-field limit};
    \node[black,text width=8cm] at (7.5,-1.4) 
    {\small{ \sf Theorems \ref{the:5} and \ref{the:5-coup}: An optimal solution exists that is exchangeable (not necessarily symmetric) and randomized.}};
    \node[black,text width=8cm] at (7.5,-4.7) 
    {\small{ \sf Theorems \ref{the:6} and \ref{the:6-coup}: An optimal solution exists that is symmetric and independently randomized.}};
\end{tikzpicture}
\end{center}
We finally note that the paper can be viewed as a continuous-time counterpart of \cite{sanjari2018optimal, sanjari2019optimal, SSYdefinetti2020}; however, the setup in continuous-time offers additional technical challenges which are addressed in the paper. 

\section{Convex Decoupled Teams}

\subsection{Teams \PN\ and \PIN}
In this section, we introduce a class of continuous-time exchangeable teams with $N$-DMs as well as their MF limit with a countably infinite number of DMs.

\subsubsection{$N$-DM Continuous-Time Teams \PN.}

In our first  model, each DM has access to only a private state process that evolves as a controlled diffusion process given by the solution of the following stochastic differential equation (SDE)
\begin{align}\label{eq:st-dy-N}
    dX_{t}^i = b_{t}(X_{t}^i,U_{t}^i)dt + \sigma_t(X_{t}^{i})dW_{t}^{i}, \quad t\in \mathbb{T}:=[0,T],
\end{align}
 for some $T\in \mathbb{R}_{+}$ and all $i\in \mathcal{N}:=\{1, \ldots, N\}$, where $\mathcal{N}$ denotes the number of DMs. The process $(W_{t}^{i})_{t}$ denotes i.i.d. (among DMs) $d$-dimensional Wiener process for $i\in \mathcal{N}$, defined on a complete probability space $(\Omega, \mathcal{F}, \mathbb{P})$. Since our focus is on exchangeable teams, we let for every $t\in \mathbb{T}$, $\mathbb{X}=\mathbb{R}^{d}$ and $\mathbb{U}\subseteq \mathbb{R}^{m}$ be state and action spaces for some $d,m\in \mathbb{N}$, which are identical among DMs. In the above, $X_{0}^{1:N}$ are i.i.d. random variables with $\mathbb{E}[|X_{0}^{i}|^2]<\infty$ for all $i\in \mathcal{N}$, where $\mathbb{E}$ is the expectation with respect to $\mathbb{P}$. 
 
 Later on in Section \ref{sec:coup}, we consider a more general formulation by allowing the dynamics to be weakly coupled through the empirical measures of states and actions, i.e., each DM has access to only a private state process that evolves as a controlled diffusion process given by the solution of the SDE in \eqref{eq:st-dy-C-intro}. 
 
 In the following, we consider the first model with dynamics described by \eqref{eq:st-dy-N}. 
 
 \begin{definition}[{Admissible policies for each DM}]\label{Def:Adm}
     For every $i\in \mathcal{N}$, the action process $\{U_{t}^{i}\}_{t}$ is an adapted process to the natural filtration generated by $X_{[0,t]}^{i}$, that is  $U_{t}^{i}$ is induced by a measurable function that maps the history of private states $X_{[0,t]}^{i}$ to $U_t^i \in \mathbb{U}$.
 \end{definition}
  Denote the admissible policies for each DM$^i$ by $\pmb{\gamma^i}:=(\gamma^i_{t})_{t\in \mathbb{T}}$ which belongs to $ \Gamma^{i}$, and their product by $\pmb{\gamma^{1:N}}:=(\pmb{\gamma^{1}}, \ldots, \pmb{\gamma^{N}})$ belonging to  $\Gamma_{N}:=\prod_{i=1}^{N} \Gamma^{i}$. Later on, we allow randomization in policies which is required for our analysis. Throughout the entire paper, we use the bold letter notation to denote the path of the corresponding variable over time. The notation $a:b$ is used for $a, a+1, \ldots, b$ for any $b\geq a$. We emphasize that the information structure considered in Definition \ref{Def:Adm} for admissible policies is decentralized because each DM has only access to the private history of the states. In other words, the private states of DMs are not shared with others.

To guarantee the existence and uniqueness of a solution of \eqref{eq:st-dy-N} for every $i\in \mathcal{N}$, we impose the following assumptions on the drift $b_t:\mathbb{R}^d \times \mathbb{U} \to \mathbb{R}^d$ and the diffusion $\sigma_t:\mathbb{R}^d\to \mathbb{R}^{d\times d}$.
\begin{assumption}\label{assump:existence}
The following three conditions hold:
\begin{itemize}[wide]
    \item [(i)-Local Lipschitz continuity:] For every $t\in \mathbb{T}$, $b_t$ and $\sigma_{t}$ are locally Lipschitz continuous in $x$ (uniformly with respect to the other variables for $b_t$), i.e., for some constant $C_R > 0$ depending on $R > 0$
    \begin{align}
        |b_t(x,u)-b_t(y,u)|+\|\sigma_{t}(x)-\sigma_t(y)\| \leq C_R|x-y|,
    \end{align}
    for all $u\in \mathbb{U}$ and $x,y\in \mathcal{B}_R$, where $\mathcal{B}_R$ denotes the open ball of radius $R$ in $\mathbb{R}^d$
 centered at the origin. In the above, $\|\sigma_{t}(x)\|:=\sqrt{{\sf Tr}[\sigma^T_t(x)\sigma_t(x)]}$. Also, let $b_t(\cdot,\cdot)$ be
jointly continuous.
\item [(ii)-Affine growth condition:] For every $t\in \mathbb{T}$, $b_t$ and $\sigma_t$ satisfy a global growth condition of the form
\begin{align*}
|b_{t}(x,u)|^2 + \|\sigma_{t}(x)\|^2 \leq C_{0}(1+|x|^2)\qquad \forall (x,u)\in \mathbb{R}^d \times \mathbb{U}
\end{align*}
for some constant $C_{0}>0$.
\item [(iii)-Nondegeneracy:] The least eigenvalue of $\sigma_t\sigma_t^T$ is bounded away
from $0$ on every open ball $\mathcal{B}_{R}$ for every $t\in \mathbb{T}$, i.e., for every $t\in \mathbb{T}$ and $R > 0$, 
\begin{align}
    \sum_{i,j=1}^{d}a^{i,j}_t(x)l_{i}l_{j}\geq C_{R}^{-1}|l|^2\qquad \forall x\in \mathcal{B}_{R},
\end{align}
and for all $l:=(l_1, \ldots, l_j)^T\in \mathbb{R}^d$, where $a_t=\frac{1}{2}\sigma_t\sigma_{t}^T$.
\end{itemize}
\end{assumption}

 {Following \cite[Theorem 2.2.4]{arapostathis2012ergodic}, under Assumption \ref{assump:existence}(i) and (ii),  \eqref{eq:st-dy-N} admits a unique strong solution under the more general {\em non-anticipative policies}; see \cite[p. 30]{arapostathis2012ergodic}.} Although  Assumption \ref{assump:existence}(iii) is not needed for \eqref{eq:st-dy-N} to admit a unique strong solution, it is required to establish continuity of path space in policies, which will be needed in Section \ref{sec:exi}. 
 
 The DMs collectively minimize the following expected cost function given their private information:
\begin{align}\label{eq:sec1-2}
    J_N(\pmb{\gamma^{1:N}}):= \mathbb{E}^{\pmb\gamma^{1:N}}\left[\int_{0}^{T} c\left(X_{t}^{1:N}, U_{t}^{1:N}\right) dt\right]
\end{align}
for a given policies $\pmb{\gamma^{1:N}}$,
where the common cost function $c:\prod_{i=1}^{N}(\mathbb{X}\times \mathbb{U}) \to \mathbb{R}_{+}$ is measurable. In \eqref{eq:sec1-2}, $\mathbb{E}^{\pmb\gamma^{1:N}}$ denotes the expectation with respect to $\mathbb{P}$ where actions are induced by policies $\pmb{\gamma^{1:N}}$ given the information available. 
The goal of DMs is to find a globally optimal solution which is introduced below. 

\begin{definition}[Global Optimality for \PN]\label{def:GOS}
An admissible policy $\pmb{\gamma^{1\star:N\star}}\in \Gamma_{N}$ is globally optimal if 
\begin{align}
    \inf_{\pmb{\gamma^{1:N}}\in \Gamma_{N}}J_N(\pmb{\gamma^{1:N}})=  J_N(\pmb{\gamma^{1\star:N\star}}).
\end{align}
\end{definition}

Our focus in this paper is on exchangeable teams, and hence, we require the following assumption on the cost function.

\begin{assumption}\label{assump1}
 $c$ is exchangeable, i.e., for any permutation $\tau$ of $\{1, \ldots, N\}$ ($\tau\in S_{N}$), 
    \begin{align}
        c\left(X^{1:N}, U^{1:N}\right)=c\left(X^{\tau(1):\tau(N)}, U^{\tau(1):\tau(N)}\right).
    \end{align}
\end{assumption}

A specific example of a cost function that satisfies  Assumption \ref{assump1} is the cost for the MF teams: 
\begin{align}\label{eq:cost-mf}
   c\left(X^{1:N}, U^{1:N}\right)= \frac{1}{N}\sum_{i=1}^{N}\hat{c}\left(X^{i},U^{i},\frac{1}{N}\sum_{p=1}^{N} \delta_{X^{p}},\frac{1}{N}\sum_{p=1}^{N} \delta_{U^{p}} \right)
\end{align}
for some measurable function $\hat{c}:\mathbb{X}\times \mathbb{U} \times \mathcal{P}(\mathbb{X})\times \mathcal{P}(\mathbb{U})\to \mathbb{R}_{+}$, where the sets of probability measures are endowed with the weak convergence topology. For such team problems, each DM's running cost is weakly connected to the entire population of DMs only through the empirical of states and actions.

\subsubsection{$N$-DM Teams and Their MF Limits \PIN .}

To study the MF limit of exchangeable teams, we focus on the cost function in \eqref{eq:cost-mf}, which satisfies Assumption \ref{assump1}. 
 For this problem, the expected cost of a policy $\pmb{\gamma^{1:N}}\in \Gamma_{N}$ is given by
\begin{align}\label{eq:JN}
J_{N}(\pmb{\gamma^{1:N}}):= \mathbb{E}^{\pmb{\gamma^{1:N}}}\left[\int_{0}^{T} \frac{1}{N}\sum_{i=1}^{N}\hat{c}\left(X_{t}^{i},U_{t}^{i},\frac{1}{N}\sum_{p=1}^{N} \delta_{X_{t}^{p}},\frac{1}{N}\sum_{p=1}^{N} \delta_{U_{t}^{p}} \right) dt\right],
\end{align}
 where the common cost function $\hat{c}:\mathbb{X}\times \mathbb{U} \times \mathcal{P}(\mathbb{X})\times \mathcal{P}(\mathbb{U})\to \mathbb{R}_{+}$ is a measurable function. 
The goal of DMs is to find globally optimal solutions introduced in Definition \ref{def:GOS}. The corresponding MF limiting problem \PIN\ is defined as the team problem with the expected cost given by the limsup of \eqref{eq:JN}, i.e., 
\begin{align}
   J_{\infty}(\underline{\pmb{\gamma}}):= \limsup_{N\to \infty} \mathbb{E}^{\underline{\pmb{\gamma}}}\left[\int_{0}^{T} \frac{1}{N}\sum_{i=1}^{N}\hat{c}\left(X_{t}^{i},U_{t}^{i},\frac{1}{N}\sum_{p=1}^{N} \delta_{X_{t}^{p}},\frac{1}{N}\sum_{p=1}^{N} \delta_{U_{t}^{p}} \right) dt\right],
\end{align}
for a given policy $\underline{\pmb{\gamma}}:=(\pmb{\gamma^1}, \pmb{\gamma^2}, \ldots)$ in $\Gamma:=\prod_{i=1}^{\infty} \Gamma^{i}$. Throughout the entire paper, we use the notation $\underline{a}$ for any variable $a$ to denote the infinite sequence $(a^{1}, a^{2}, \ldots)$. Next, we define global optimality for \PIN.

\begin{definition}[Global Optimality for \PIN]
An admissible policy $\underline{\pmb{\gamma^{\star}}}\in \Gamma$ is globally optimal if 
\begin{align}
    \inf_{\underline{\pmb{\gamma}}\in \Gamma}J_{\infty}(\underline{\pmb{\gamma}})=  J_{\infty}(\underline{\pmb{\gamma^{\star}}}).
\end{align}
\end{definition}

\subsection{Existence of a Symmetric Markovian Optimal Policy for \PN\ and \PIN}

This section establishes the existence and structural results for a globally optimal policy of \PN\ and \PIN\ under the following convexity condition.

\begin{assumption}\label{assump:convex}
$\mathbb{U}$ is convex and $J(\pmb{\gamma^{1:N}})$ is convex in $\pmb{\gamma^{1:N}}$.
\end{assumption}

We call teams satisfying  Assumption \ref{assump:convex} by \emph{convex teams}. A class of team problems is provided in Appendix \ref{sec:examples} that satisfies   Assumption \ref{assump:convex}.

\subsection{Existence of a Symmetric Markovian Globally Optimal Solutions for \PN.}\label{sec:exi}

In this section, we establish the existence of a globally optimal solution for convex teams \PN\ that is symmetric and Markovian under continuity conditions. We first characterize the structural properties of a globally optimal solution.

In the following definition, we formally introduce a set of Markovian policies as a subset of admissible policies.

\begin{definition}[(deterministic) Markov Policies]\label{def:det-Mark}
     For each DM$^{i}$, a (deterministic) policy $\pmb{\gamma}^{i}$ is Markovian if for any $t\in [0,T]$,  $\gamma^{i}_{t}:\mathbb{X} \to \mathbb{U}$ such that $U_{t}^{i}$ is measurable with respect to the $\sigma$-field generated by only $(X_{t}^{i}, t)$.
\end{definition}

We denote the space of deterministic Markovian policies by $\Gamma_{N}^{\sf M}$. In this section, we address the following two questions:
\begin{itemize}
    \item [{\bf Q1)}] {\it Can we, without any loss in the team's performance, restrict the search for a globally optimal policy for \PN\ to Markovian policies?}
    \item [{\bf Q2)}] {\it If the cost is exchangeable among DMs, then can we, without any loss, restrict the search for a globally optimal policy for \PN\ to symmetric (identical) Markovian policies?} 
\end{itemize}

We affirmatively address these two questions under some conditions. Specifically, we begin by addressing Q1. {We first note that for any Markov policy, following \cite[Theorem 2.2.12]{arapostathis2012ergodic}, under  Assumption \ref{assump:existence},  \eqref{eq:st-dy-N} admits a pathwise unique strong solution which is also a strong Feller process
(see \cite[Theorem 2.2.12]{arapostathis2012ergodic}).} In the following theorem, we show that the search for globally optimal policies for \PN\ can be confined to Markovian policies without any loss in the team's performance. The proof of the following theorem can be found in Appendix \ref{App:the1}.

\begin{theorem}\label{the:1}
Consider \PN\ under Assumption \ref{assump:existence}. Then:
\begin{align}\label{eq:sec2-4}
    \inf_{\pmb{\gamma^{1:N}}\in \Gamma_{N}}J_N(\pmb{\gamma^{1:N}}) = \inf_{\pmb{\gamma^{1:N}}\in \Gamma_{N}^{\sf M}}J_N(\pmb{\gamma^{1:N}}).
\end{align}
\end{theorem}


In Theorem \ref{the:1}, we establish that without any loss, the globally optimal policies can be assumed to be Markovian. The idea behind the proof of Theorem \ref{the:1} is as follows: We first fix the policies of DMs to arbitrary admissible (possibly path-dependent policies), and consider the optimization of a deviating DM. Since the dynamics in \eqref{eq:st-dy-N} are decoupled, the independence of random variables stated allows us to define a new cost function for a deviating DM such that this optimization problem follows the classical continuous-time Markov decision problem. As a result, an optimal solution exists that is Markovian for the deviating DM. Then, we can reason that any (approximate) globally optimal solution can be replaced by an (approximate) globally optimal Markovian solution.

We use Theorem \ref{the:1} to address Q2. In the following, we show that under  Assumption \ref{assump:existence}--Assumption \ref{assump:convex}, without loss, the search for globally optimal policies can be confined to symmetric Markovian policies, which is the set of identical policies $(\pmb{\gamma}, \ldots, \pmb{\gamma})\in \Gamma^{\sf M}_{N}$ among DMs. We denote the set of all symmetric Markovian policies by $\Gamma^{{\sf SYM, M}}_{N}\subset \Gamma^{{\sf M}}_{N}$. We now present the following theorem for convex exchangeable teams, where the cost function and action spaces satisfy Assumption \ref{assump1} and Assumption \ref{assump:convex}. The proof of the following theorem is provided in Appendix \ref{App:the2}.

\begin{theorem}\label{the:2}
    Consider \PN\ under Assumption \ref{assump:existence} and Assumption \ref{assump1}. Then, 
    \begin{itemize}
        \item [(i)] For any $\tau\in S_{N}$, 
        $J_N(\pmb{\gamma^{1:N}}) = J_N(\pmb{\gamma}^{\tau(1):\tau(N)}).$
        \item [(ii)] If in addition Assumption \ref{assump:convex} holds, then
\begin{align}\label{eq:sec2-12}
    \inf_{\pmb{\gamma^{1:N}}\in \Gamma_{N}}J_N(\pmb{\gamma^{1:N}}) = \inf_{\pmb{\gamma^{1:N}}\in \Gamma_{N}^{\sf SYM,M}}J_N(\pmb{\gamma^{1:N}}).
\end{align}
    \end{itemize}
\end{theorem}

Theorem \ref{the:2}(ii) yields that we can assume that a globally optimal solution belongs to symmetric Markovian policies provided that the additional convexity conditions in Assumption \ref{assump:convex} hold. We note that Theorem \ref{the:1} and Theorem \ref{the:2} do not guarantee the existence of a globally optimal solution. Theorem \ref{the:2}(i) follows from exchangeability of the cost function in Assumption \ref{assump1}, the fact that dynamics are symmetric. Part (ii) of Theorem \ref{the:2} follows from Jensen's inequality and part (i), using the convexity conditions in Assumption \ref{assump:convex}. In particular, we can show that given any admissible policy, its convex combination by averaging over all permutations of policies of DMs performs at least as well as the given policy. We can replace the convexity condition in Assumption \ref{assump:convex} to a weaker assumption that $J(\pmb{\gamma^{1:N}})$ is convex only for Markovian policies $\pmb{\gamma^{1:N}}\in \Gamma^{M}_{N}$ since we can restrict the policies to be Markovian without any loss thanks to Theorem \ref{the:1}.

In the following, we establish the existence of a globally optimal policy for \PN. By Theorem \ref{the:1}, without any loss, we can restrict the search for globally optimal policies to Markov policies. In the following, we introduce the set of randomized (relaxed) Markov policies for each DM, which will be used in our analysis and results.

\begin{definition}[Set of randomized Markov  policies]\label{def:Rand-Mark}
For each DM$^i$, the set of randomized Markov policies is the set of all Borel measurable functions $\pmb{\nu^i}:[0,T]\times \mathbb{X}\to \mathcal{P}(\mathbb{U})$ such that the $\mathcal{P}(\mathbb{U})$-valued action process is given by $U_t^i=\nu_{t}^{i}(X_{t}^{i})$ which is adapted to $\sigma(X_t^i,t)$.
\end{definition}

 Denote the set of randomized Markov policies by $\mathcal{M}^i$. The corresponding subset of deterministic Markovian policies includes policies of the form $\nu_{t}^{i}(X_{t}^i)(\cdot)=\delta_{\{\gamma_{t}(X_t^i)\}}(\cdot)$ for a (deterministic) policy $\pmb{\gamma}\in \Gamma^i$ in Definition \ref{def:det-Mark}. We endow the set of randomized Markov policies with Borkar's topology in \cite{borkar1989topology} under which $\pmb{\nu^i_n}$ converges to $\pmb{\nu^i}$ as $n\to \infty$ if 
\begin{align}
    &\lim_{n\to \infty}\int_{0}^{T}\int_{\mathbb{X}} f(t,x) \int_{\mathbb{U}} g(t, x, u) \nu^i_{n,t}(x)(du)\ dx dt\nonumber\\
    &=\int_{0}^{T}\int_{\mathbb{X}} f(t,x) \int_{\mathbb{U}} g(t, x, u) \nu^i_{t}(x)(du)\ dxdt \label{eq:topology}
\end{align}
for all $f\in L_{1}([0,T]\times \mathbb{X})\cap L_{2}([0,T]\times\mathbb{X})$ and $g\in \mathcal{C}_b([0,T]\times\mathbb{X}\times \mathbb{U})$, where $\mathcal{C}_b$ denotes the set of all bounded continuous function and $L_1$ and $L_2$ denote the set of integrable and square-integrable functions, respectively.

We note that the aforementioned private randomization of DM's policy does not contribute to an improvement in team performance \cite{saldi2022geometry,YukselSaldiSICON17}; thus, the expansion of the policy space remains justified.  We note that Borkar’s topology is equivalent to Young's topology if the reference measure used for the change of measure is finite and has an everywhere positive density \cite{yuksel2023borkar}. 

We now present the following lemmas from \cite[Lemma 4.1 and Theorem 4.1]{borkar1989topology} (see also some technical relaxations in \cite{pradhan2024continuity}) that will be used for our main theorems for \PN\ and \PIN. Denote the law of $\pmb{X^{i}}$ on the path space under a randomized Markov policy $\pmb{\nu^i}$ by $\pmb{\mu^{i}}(\cdot ;\pmb{\nu^{i}}):=\mathcal{L}(\pmb{X^{i}};\pmb{\nu^{i}})(\cdot)\in \mathcal{P}(\mathcal{C}([0,T];\mathbb{X}))$. 

\begin{lemma}\label{lem:Borkar}
Suppose that $\mathbb{U}$ is compact. Then, $\mathcal{M}^i$ is compact for all $i\in \mathcal{N}$.
\end{lemma}
\begin{lemma}\label{lem:Borkar-cont}
For all $i\in \mathcal{N}$, if $\pmb{\nu^i_n}$ converges to $\pmb{\nu^i}$ in Borkar's topology as $n\to \infty$, then $\pmb{\mu^{i}}(\cdot;\pmb{\nu^{i}_n})$ converges weakly to $\pmb{\mu^{i}}(\cdot;\pmb{\nu^{i}})$ as $n\to \infty$, i.e., $G^i:\pmb{\nu^{i}}\mapsto \pmb{\mu^{i}}(\cdot;\pmb{\nu^{i}})$ is continuous for all $\pmb{\nu^{i}}$ for all $i\in \mathcal{N}$.
\end{lemma}

To establish existence results, our analysis relies on the compactness of the underlying policy space and continuity of the induced probability measure on the path space in policies. Lemma \ref{lem:Borkar} and Lemma \ref{lem:Borkar-cont} guarantee that $\mathcal{M}^i$ is a compact space under Borkar's topology in \eqref{eq:topology} and the probability measure on the path space is continuous in policies, rendering it suitable for our analysis. In the subsequent theorem, we establish the existence of a globally optimal solution for convex teams of \PN\ that is both Markovian and symmetric. However, to achieve this, we require the following assumption on the continuity of the cost function.

\begin{assumption}\label{assum:5-4}
$c$ in \eqref{eq:sec1-2} ($\hat{c}$ in \eqref{eq:cost-mf}) is continuous in all its arguments.
\end{assumption}

The proof of the following theorem is provided in Appendix \ref{App:the3}.

\begin{theorem}\label{the:3}
    Consider \PN\ under Assumption \ref{assump:existence} and Assumption \ref{assum:5-4}. Let $\mathbb{U}$ be compact. Then:
\begin{itemize}
        \item [(i)]  There exists a randomized globally optimal policy $(\pmb{{\nu}^{1}}, \ldots, \pmb{{\nu}^{N}})$ that is Markovian, i.e., $\pmb{{\nu}^{i}}\in \mathcal{M}^i$ for all $i\in \N$.
        \item [(ii)] If additionally  Assumption \ref{assump1} and Assumption \ref{assump:convex} hold, then there exists a randomized globally optimal policy $(\pmb{{\nu}}, \ldots, \pmb{{\nu}})$ that is Markovian and symmetric.
    \end{itemize}
\end{theorem}


Theorem \ref{the:3} initially establishes the existence of a randomized globally optimal policy that adheres to Markovian properties. Furthermore, under the presence of additional exchangeability and convexity conditions, as specified in Assumption \ref{assump1} and Assumption \ref{assump:convex}, the theorem ensures the existence of a globally optimal policy that is Markovian and symmetric. We note that the applicability of Theorem \ref{the:3} extends to the problem involving the MF teams with a finite number of DMs, as the cost function in this scenario is also exchangeable.

According to Theorem \ref{the:1}, we can assume, without loss of optimality, that a globally optimal solution is Markovian. This allows us to focus solely on demonstrating the existence of a globally optimal solution within the set of randomized Markov policies $\mathcal{M}^i$, as randomization does not enhance the team's performance. For part (i) of the proof, we establish that the expected cost function is lower-semi continuous for every policy in $\mathcal{M}^i$ using Lemma \ref{lem:Borkar-cont}. This leads to the existence of a globally optimal solution within $\mathcal{M}^i$, thanks to the compactness of $\mathcal{M}^i$ in Lemma \ref{lem:Borkar}. Part (ii) follows from a similar line of reasoning using the fact that under 
exchangeability and convexity in Assumption \ref{assump1} and Assumption \ref{assump:convex}, we can assume, without loss of optimality, that a globally optimal solution is Markovian and symmetric utilizing Theorem \ref{the:2}.

\begin{remark} 
Following Theorem \ref{the:3} and its proof, we can further establish the existence of a symmetric Markovian globally optimal policy that is also deterministic under sufficient regularity conditions on the cost function needed for the corresponding HJB equations: By fixing the policies of DMs to a randomized globally optimal solution that adheres to Markovian and symmetry properties (following part (ii) of Theorem \ref{the:3}), any deviating DM faces a standard stochastic control problem since the dynamics are decoupled, noise processes and initial states are independent. Hence, considering that \PN\ comprises only a finite number of DMs, we can sequentially show that HJB equations provide the necessary and sufficient conditions for the optimal solution, which is deterministic and Markovian. The assumed exchangeability and convexity of the team problem, as presented in Theorem \ref{the:2} then can be used to establish the existence of a globally optimal solution that is deterministic, Markovian, and symmetric.
\end{remark} 



\subsubsection{A MF Optimal Solution for \PIN\ as a Limit of a Globally Optimal Policy for \PN.}

This section establishes the existence and convergence of a globally optimal policy for \PIN. In particular, we address the following two questions:
\begin{itemize}
    \item [{\bf Q3)}] {\it Consider any sequence of symmetric (possibly randomized) Markovian globally optimal solutions for \PN\ that converges (through a subsequence) as $N\to \infty$ to a limiting policy. Is the limiting policy a (randomized) globally optimal solution for \PIN?}
    \item [{\bf Q4)}] {\it Does \PIN\ admit (possibly randomized) globally optimal solution that is symmetric and Markovian?} 
\end{itemize}

In the following, we address the above two questions by first establishing that the limit of any converging sequence of symmetric Markov globally optimal policies for \PN\ is a (possibly randomized) global optimal policy for \PIN. This, in particular, leads to the existence of a globally optimal policy for \PIN\ that is symmetric and Markovian. The proof of the following theorem is provided in Appendix \ref{App:the4}.

\begin{theorem}\label{the:4}
     Consider \PN\ and \PIN\ under Assumption \ref{assump:existence}, Assumption \ref{assump1},  Assumption \ref{assump:convex}, and Assumption \ref{assum:5-4}. Let $\mathbb{U}$ be compact. Then:
     \begin{itemize}
         \item [(i)] Any sequence of symmetric Markov globally optimal policies (with $\{\pmb{{\nu}_N}\}_N\subseteq \mathcal{M}^i$ for each DM) for \PN\ admits a subsequence that converges to a globally optimal policy for \PIN\ that is symmetric and Markovian.
         \item [(ii)] There exists a randomized globally optimal policy for \PIN\ that is Markovian and symmetric.
     \end{itemize}
\end{theorem}

Theorem \ref{the:4} establishes both convergence and existence results for \PN\ and \PIN. We accomplish this by utilizing a symmetric Markovian globally optimal solution of \PN\ and demonstrating its convergence to an optimal solution of the corresponding MF team, \PIN. The resulting limit policy is Markovian and symmetric, thus ensuring the existence of a globally optimal solution for \PIN\ that is also Markovian and symmetric. As a result of symmetry and independence in DMs' optimal policies and the induced states, we can represent the infinite population decentralized team problem as a single representative DM control problem, and hence,  the corresponding HJB equations characterize a symmetric Markovian optimal policy. 

We shall now explain the proof of Theorem \ref{the:4}. First, we demonstrate that the empirical measure of joint actions and states under a symmetric Markovian globally optimal solution for \PN\ converges weakly. This convergence is facilitated by the symmetry of a globally optimal solution for \PN, as established in Theorem \ref{the:3}, along with the independence of random variables. Next, by utilizing the generalized dominated convergence theorem in \cite{serfozo1982convergence} and leveraging the convergence of the empirical measure of joint actions and states, we demonstrate the continuity of the expected cost function. This continuity property arises since a sequence of a globally optimal solution for \PN\ converges to a limiting policy. With the continuity of the expected cost function established in part (i), proof of part (ii) follows immediately since $\mathcal{M}^i$ is a compact space, as demonstrated in Lemma \ref{lem:Borkar}, which ensures the existence of a globally optimal solution for \PIN.

{Theorem \ref{the:4} establishes the existence of a globally optimal solution for \PIN\ that is Markovian and symmetric. This structural result leads to a computationally tractable approach thanks to symmetry: following symmetry and independence, we can represent the infinite population MF problem as a single representative DM control problem by replacing the empirical measure of i.i.d. states and actions as their limits $\mu_{X_t^{\sf R}}=\mathcal{L}(X_t^{\sf R})$ and $\mu_{U_t^{\sf R}}=\mathcal{L}(U_t^{\sf R})$. As a result, by fixing  $\mu_{X_t^{\sf R}}\in \mathcal{P}(\mathbb{X})$ and $\mu_{U_t^{\sf R}}\in \mathcal{P}(\mathbb{U})$ for all $t\in [0,T]$, we can compute a symmetric Markovian optimal policy for \PIN\ provided that the corresponding HJB equations or master equations for the representative agent with consistency conditions $\mu_{X_t^{\sf R}}=\mathcal{L}(X_t^{\sf R})$ and $\mu_{U_t^{\sf R}}=\mathcal{L}(U_t^{\sf R})$ admit a solution. We do not provide an explicit analysis of this as further technical regularity conditions are required to guarantee the existence of a solution to the associated infinite dimensional optimality equation.}


\section{Exchangeable Stochastic Teams with Decoupled Dynamics: Non-Convex Case and Optimality and Near Optimality of Symmetric Policies}

In this section, we consider the model \eqref{eq:st-dy-N} and relax the convexity condition in Assumption \ref{assump:convex}. To this end, we allow for randomized policies; however, in contrast to the convex case where we considered Markovian randomized policies, we introduce different sets of randomized policies, known as \emph{wide-sense admissible policies} \cite{arapostathis2012ergodic,kushner2012weak} and their decentralized variation \cite{pradhan2023controlled}: To this end, we allow each DM$^i$ for $i\in \mathcal{N}$ to have access to the information generated by $X_{0}^i,W_{[0,t]}^i$ at time $t$ for $t\in \mathbb{T}$.  

\subsection{Sets of Randomized Policies}\label{sec:set-Ran}
{For each DM$^{i}$ ($i\in \mathcal{N}$), we re-define $\Gamma^i$ as the set of all wide-sense admissible control policies so that the random measure-valued variables $\pmb{\gamma^{i}}(\omega)\in  \mathcal{P}([0,T]\times \mathbb{U})$ satisfy the following two conditions:
\begin{itemize}
    \item [(i)] Marginals on $[0,T]$ are fixed to be the Lebesgue measure.
    \item [(ii)] $\gamma^{i}_{t}:\omega\mapsto \pmb{\gamma^{i}}(\cdot|t)(\omega)$ is independent of $W^{i}_{s}-W^{i}_{t}$ for $s>t$ and for any $t\in \mathbb{T}$, and $W^{j}_{[0,T]}$ for any $j\in \mathcal{N}$, $j\not=i$. 
    \end{itemize}

We equip the random measures $\Gamma^i\subseteq\mathcal{P}([0,T]\times \mathbb{U})$ with the Young topology; that is, $\pmb{\gamma^{i}_n}(\omega)$ converges to $\pmb{\gamma^{i}}(\omega)$ if the random measure converges weakly}. 

We consider the state process $\pmb{X^i}$ as the solution of diffusion equation \eqref{eq:st-dy-N} which is a $\mathcal{C}([0, T];\mathbb{X})$-valued process
 (under sup-norm) induced by $\pmb{\gamma^i}(d{U}^i,dt)(\omega)$ and define the space of probability measures on these random variables.  Following (ii), for any $x_{0}\in \mathbb{X}$, $t\in \mathbb{T}$, and twice continuously differentiable functions $f:\mathbb{X}\to \mathbb{R}$, the following martingale equation holds: 
    \begin{align}\label{eq:martingale-1}
    \mathbb{E}\left[f(X_{t}^{i})-f(X_{0}^i)-\int_{0}^{t}\mathcal{A}^{{U^i}}f(X_{s}^{i})\pmb{\gamma^{i}}(d{U^{i}},ds)\bigg|X_{0}^{i}=x_0\right] = 0,
        \end{align}
    where $\mathcal{A}^{{U^i}}$ is the generator of the transition semigroup of $\pmb{X^{i}}$, i.e.,  
\begin{align*}
    \mathcal{A}^{{U^i}}f(X_s^i)&:= {\sf Tr}\left(a_{s}(X_{s}^i)\nabla^2 f(X_{s}^{i})\right)+b_{s}(X_{s}^{i},{U^i})\cdot \nabla f(X_{s}^{i}),
    \end{align*}
where $a_{s}:=\frac{1}{2}\sigma\sigma^{T}$. We denote the set of all randomized policies for DMs within the team problem on $\prod_{i=1}^{N}\Gamma^{i}$ by $L^{N}:=\mathcal{P}(\prod_{i=1}^{N}\Gamma^{i})$ which is the collection of all Borel probability measures on $\prod_{i=1}^{N}\Gamma^{i}$, where Borel $\sigma$-field $\mathcal{B}(\Gamma^i)$ is induced by the Young topology and we equip $\prod_{i=1}^{N}\Gamma^{i}$ with the product topology. 

In the following, we introduce different sets of randomized policies by allowing common and individual randomness among DMs. Define the set of randomized policies with individual and common randomness by 
\begin{flalign*}
\LCON:=\bigg\{P_{\pi} \in L^{N}&\bigg{|}\text{for all}~A^{i} \in \mathcal{B}(\Gamma^{i}):P_{\pi}(\pmb{\gamma^{1}} \in A^{1},\dots,\pmb{\gamma^{N}} \in A^{N}) \nonumber\\
&=\int_{z\in [0,1]}\prod_{i=1}^{N}P_{\pi}^i(\pmb{\gamma^{i}}\in A^{i}|z)\eta(dz), ~~~~\eta \in \mathcal{P}([0, 1])\bigg\},
\end{flalign*} 
where $\eta$ is the distribution of common randomness (independent from the intrinsic exogenous system random variables). In the above, for every fixed $z$, $P_{\pi}^i\in \mathcal{P}(\Gamma^{i})$ corresponds to an independent randomized policy for each DM$^{i}$ ($i\in \mathcal{N}$).  Following \cite[Theorem A.1]{SSYdefinetti2020}, we note that $\LCON$ and $L^{N}$ are identical. 
{We next recall the definition of {\it exchangeability} for random variables.
\begin{definition}
Random variables $y^{1:m}$ defined on a common probability space are $m$-\it{exchangeable} if for any permutation $\sigma$ of the set $\{1,\dots,m\}$, $\mathcal{L}\left(y^{\sigma(1):\sigma(m)}\right)=\mathcal{L}\left(y^{1:m}\right)$, where $\mathcal{L}(\cdot)$ denotes the (joint) law of random variables. Random variables $(y^{1},y^{2},\dots)$ are {\it infinitely-exchangeable} if finite distributions of $(y^{1},y^{2},\dots)$ and $(y^{\sigma(1)},y^{\sigma(2)},\dots)$ are identical for any finite permutation (affecting only finitely many elements) of $\mathbb{N}$. 
\end{definition}}

The set of exchangeable randomized policies $\LEXN\subseteq L^{N}$  is given by
\begin{flalign}
\LEXN:=\bigg\{&P_{\pi} \in L^{N}\bigg{|} \forall A^{i} \in \mathcal{B}(\Gamma^{i})~\text{and}~\forall\sigma \in S_{N}:\nonumber\\
&P_{\pi}(\pmb{\gamma^{1}} \in A^{1},\dots,\pmb{\gamma^{N}}\in A^{N})=P_{\pi}(\pmb{\gamma}^{\sigma(1)} \in A^{1},\ldots,\pmb{\gamma}^{\sigma(N)}\in A^{N})\bigg\}\label{eq:LEXN},
\end{flalign}
where $S_{N}$ is the set of all permutations of $\{1,\ldots,N\}$. We note that $\LEXN$ is a subset of $L^{N}$.  {Let  $\LCOSN$ be the set of all symmetric (identical) randomized policies induced by common and individual randomness:
\begin{flalign*}
\LCOSN:=\bigg\{P_{\pi} \in L^{N}&\bigg{|}\forall A^{i} \in \mathcal{B}(\Gamma^{i}):P_{\pi}(\pmb{\gamma^{1}} \in A^{1},\dots,\pmb{\gamma^{N}}\in A^{N}) \nonumber\\
&=\int_{z\in [0,1]}\prod_{i=1}^{N}\widetilde{P}_{\pi}(\pmb{\gamma^{i}} \in A^{i}|z)\eta(dz), ~~~~\eta \in \mathcal{P}([0, 1])\bigg\},
\end{flalign*}
where for all $i \in \mathcal{N}$, conditioned on $z$, $\widetilde{P}_{\pi} \in \mathcal{P}(\Gamma^{i})$ corresponds to an independent randomized policy for each DM$^{i}$, which is symmetric among DMs. We additionally note that $\LCOSN \subset \LEXN$. 

We also define the following set of symmetric randomized policies with private independent randomness as
\begin{flalign*}
\LPRS^{N}:=\bigg\{&P_{\pi} \in L^N\bigg{|}\text{for all}~A^{i} \in \mathcal{B}(\Gamma^{i}):\\
&P_{\pi}(\pmb{\gamma^{1}} \in A^{1},\ldots, \pmb{\gamma^{N}} \in A^{N})=\prod_{i=1}^{N}\widetilde{P}_{\pi}(\pmb{\gamma^{i}}\in A^{i}),~\text{for}~\widetilde{P}_{\pi}\in \mathcal{P}(\Gamma^{i})\bigg\}\nonumber.
\end{flalign*}

We note that $\LPRS^{N}$ is not a convex set; however, $\LPRS^N$ contains the set of extreme points of the convex set $\LCOSN$; see e.g., \cite{saldi2022geometry,YukselSaldiSICON17}. This fact will be used in our analysis.



When the number of DMs is countably infinite, using the Ionescu-Tulcea extension theorem, we define the corresponding sets of randomized policies $L, \LCO, \LEX$, and $\LCOS$  by iteratively adding new coordinates for our probability measures (see e.g., \cite{InfiniteDimensionalAnalysis,HernandezLermaMCP}). The set of randomized policies $L$ on the infinite product Borel spaces $\Gamma:=\prod_{i\in \mathbb{N}}\Gamma^{i}$ is defined as
$L:=\mathcal{P}(\Gamma)$.
Let the set of all randomized policies with common and independent randomness be given by
\begin{flalign*}
\LCO:=\bigg\{P_{\pi} \in L&\bigg{|}\text{for all}~A^{i} \in \mathcal{B}(\Gamma^{i}): P_{\pi}(\pmb{\gamma^{1}} \in A^{1},\pmb{\gamma^{2}} \in A^{2},\dots)\nonumber\\
&=\int_{z\in [0,1]}\prod_{i\in \mathbb{N}}P_{\pi}^{i}(\pmb{\gamma^{i}}\in A^{i}|z)\eta(dz), ~~~~\eta \in \mathcal{P}([0, 1])\bigg\}.
\end{flalign*}
Let the set of all infinitely exchangeable randomized policies $\LEX$ be given by
\begin{flalign*}
\LEX:=\bigg\{&P_{\pi} \in L\bigg{|}\text{for all}~A^{i} \in \mathcal{B}(\Gamma^{i})~\text{and for all $N\in\mathbb{N}$, and for all $\sigma \in S_{N}$:} \nonumber\\
&P_{\pi}(\pmb{\gamma^{1}} \in A^{1},\dots,\pmb{\gamma^{N}}\in A^{N})=P_{\pi}(\pmb{\gamma}^{\sigma(1)} \in A^{1},\dots,\pmb{\gamma}^{\sigma(N)}\in A^{N})\bigg\}.
\end{flalign*}
Let the symmetric policies with common and independent randomness be given by
\begin{flalign*}
\LCOS:=\bigg\{P_{\pi} \in L&\bigg{|}\text{for all}~A^{i} \in \mathcal{B}(\Gamma^{i}): P_{\pi}(\pmb{\gamma^{1}} \in A^{1},\pmb{\gamma^{2}} \in A^{2},\dots)\nonumber\\
&=\int_{z\in [0,1]}\prod_{i\in \mathbb{N}}\widetilde{P}_{\pi}(\pmb{\gamma^{i}}\in A^{i}|z)\eta(dz), ~~~~\eta \in \mathcal{P}([0, 1])\bigg\}.
\end{flalign*}

Also, define the set of symmetric randomized policies with private independent randomness as
\begin{flalign*}
\LPRS:=\bigg\{&P_{\pi} \in L\bigg{|}\text{for all}~A^{i} \in \mathcal{B}(\Gamma^{i}):\\
&P_{\pi}(\pmb{\gamma^{1}} \in A^{1},\pmb{\gamma^{2}} \in A^{2},\dots)=\prod_{i\in \mathbb{N}}\widetilde{P}_{\pi}(\pmb{\gamma^{i}}\in A^{i}),~\text{for}~\widetilde{P}_{\pi}\in \mathcal{P}(\Gamma^{i})\bigg\}\nonumber.
\end{flalign*}

In the following, we first present two lemmas on convexity, compactness, and relationships between the preceding set of exchangeable randomized policies. We use the following lemmas for our main results in Theorem \ref{the:5} and Theorem \ref{the:6}. 

\begin{lemma}\label{lem:0}
Consider  $\LEXN, \LCOSN, \LEX$, and $\LCOS$. Then, the following statements hold:
\begin{itemize}
    \item [(i)] Any infinitely-exchangeable randomized policy $P_{\pi}$ in $\LEX$ also belongs to the set of randomized policies $\LCOS$, i.e., $\LCOS=\LEX$. 
    \item [(ii)] $\LCOSN \subseteq \LEXN$; however, in general $\LCOSN \not= \LEXN$.
\end{itemize}
\end{lemma}

 \begin{proof}
 Part (i) is from \cite[Theorem 1]{SSYdefinetti2020} utilizing a de Finetti representation theorem in \cite[Theorem 1.1]{kallenberg2006probabilistic}. Part (ii) holds since finitely exchangeable random variables are not necessarily conditionally i.i.d.; see e.g., \cite[Example 1.18]{schervish2012theory}, and \cite{diaconis1980finite,aldous2006ecole}.
\end{proof}

The proof of the following lemma is provided in Appendix \ref{app:lem:2}.

\begin{lemma}\label{lem:2}
Consider the sets of randomized policies $\LEXN$ and $\LEX$. Let Assumption \ref{assump:existence} hold and $b_t(\cdot,\cdot)$ and $\sigma_t(\cdot)$ in \eqref{eq:st-dy-N} be uniformly bounded for all $t\in \mathbb{T}$.
    \begin{itemize}
        \item [(i)] Let $\mathbb{U}$ be compact. Then, $\LEXN$ and $\LEX$ are convex and compact.
        \item [(ii)]  The mapping $\Theta: H \to \mathcal{P}(\mathcal{C}([0,T],\mathbb{X}))$ is continuous for $H=\LEXN$ or $\LEX$, where $\Theta$ maps any policy to the induced probability distribution on path space $\mathcal{C}([0,T],\mathbb{X})$ for the state process.
        
    \end{itemize}
\end{lemma}

The continuity of the probability measures of the path space of states follows from the martingale equation in \eqref{eq:martingale-1} and the generalized dominated convergence theorem since the dynamics in \eqref{eq:st-dy-N} are decoupled.

\subsection{Existence of Randomized Optimal Solutions for \PN\ and \PIN}

For any randomized policies $P_{\pi}^{N}\in L^{N}$ with the induced probability measure on states ${\mu^{1:N}}$, the expected cost for \PN\ is given by
    \begin{align}\label{eq:expcost}
    J_{N}^{\pi}(P_{\pi}^{N})=\int \left(c_{N}(\pmb{X}^{1:N},\pmb{\gamma^{1:N}})\prod_{i=1}^{N}\mu^i(d\pmb{X^{i}};\pmb{\gamma^{i}})\right) P_{\pi}^{N}(d\pmb{\gamma^{1:N}}),
\end{align}
where 
\begin{align}
   c_{N}(\pmb{X}^{1:N},\pmb{\gamma^{1:N}}):= \int \int_{0}^{T} c(X_{t}^{1:N},U_{t}^{1:N})\prod_{i=1}^{N}\pmb{\gamma^{i}}(d{U^{i}},dt),
\end{align}
and $c:\prod_{i=1}^{N}\mathbb{X}\times \mathbb{U}\to \mathbb{R}_{+}$ satisfies Assumption \ref{assump1}. We can further define the expected cost when the cost is of the form \eqref{eq:cost-mf} (in this case we denote $c_{N}$ in \eqref{eq:expcost} by $\widehat{c}_{N}$). Similarly, for any randomized policies $P_{\pi}\in L$, the expected cost for \PIN\ is given by 
\begin{align}
    J^{\pi}(P_{\pi})=\limsup_{N\to \infty}J_{N}^{\pi}(P_{\pi,N}),
\end{align} where $P_{\pi,N}$ is the marginal of the $P_{\pi}\in L$ to the first $N$ components. 

In the following theorem, we establish the existence of a randomized globally optimal policy for \PN\ that is exchangeable. The proof of the following theorem is provided in Appendix \ref{app:the:5}.

\begin{theorem}\label{the:5}
     Consider \PN\ under   Assumption \ref{assump:existence}, Assumption \ref{assump1}, and Assumption \ref{assum:5-4}. Let $\mathbb{U}$ be compact and $b_t(\cdot,\cdot)$ and $\sigma_t(\cdot)$ in \eqref{eq:st-dy-N} are uniformly bounded.  Then, there exists an exchangeable randomized globally optimal policy ${P_{\pi}^{N\star}}\in \LEXN$ for \PN\ among all randomized policies $L^{N}$.
\end{theorem}

The proof of Theorem \ref{the:5} proceeds in steps similar to those utilized in the proof of Theorem \ref{the:3}. We first show the continuity of expected cost in exchangeable randomized policies $\LEXN$ using Lemma \ref{lem:3}(ii). Then, using the compactness of $\LEXN$, we conclude that there exists an optimal policy among exchangeable policies as an application of the Weierstrass theorem. Since the cost is exchangeable and dynamics are symmetric among DMs, we show that any permutation of randomized policies belonging to $L^N$ attains the same team performance. This along with the convexity of $L^N$ and linearity of the expected cost in randomized policies, allows us to construct an exchangeable policy for any given randomized policy (by averaging over all its permutations) that performs at least as well as the given policy. This implies that the restriction to exchangeable policies is without loss of optimality and concludes the proof.

As a result of Theorem \ref{the:5}, for \PN, we can restrict our search for a globally optimal policy to exchangeable policies. In contrast with the convex teams setup, where in Theorem \ref{the:3} we established the existence of an optimal solution that is symmetric and Markovian, for non-convex teams, we can only show that a globally optimal solution is exchangeable (and not necessarily symmetric). This is because we have to allow for randomization to convexify the set of policies and the expected cost due to the lack of convexity of the cost and actions. To this end, we need to allow common and independent randomization to convexify the set of randomized policies. We note that the set of all randomized policies with only independent randomization is not a convex set; see \cite{saldi2022geometry,YukselSaldiSICON17}.  By allowing this large class of randomized policies, we can only construct an exchangeable policy by averaging over all the permutations of any randomized policies. While the existence of an exchangeable optimal solution significantly diminishes the search space, establishing symmetry and independence in randomization is more desirable. This preference arises from constraints in the computational/implementation aspects of common randomness and the computational burdens associated with solely having exchangeability without the guarantee of symmetry. Later on, for a class of mean-field teams, we establish the existence of an approximately optimal policy that is symmetric using a solution of the corresponding limiting mean-field team problem. 

We recall that for the convex team, a deterministic Markovian optimal policy can be obtained by fixing other DMs' policies to optimal symmetric Markovian randomized policies, considering a deviating DM's viewpoint as a stochastic control problem, and then taking a uniform convex combination to arrive at a resulting symmetric policy. For non-convex teams, we can start with an exchangeable optimal policy following Theorem \ref{the:5}, and then fix other policies except DM$^{i}$s policy to the marginal of an exchangeable optimal policy for DM$^{-i}$. Then, using the fact that dynamics are decoupled, the deviating DM$^{i}$s faces a classical stochastic control problem and we can derive HJB optimality equations to obtain a deterministic optimal solution. Sequentially repeating this procedure, we can obtain an optimal policy that is deterministic without showing that it is symmetric (as this requires convexity). From this policy, by considering its uniform permutations, we can create an exchangeable policy that is possibly randomized and asymmetric with correlation among DMs. Later on, as the number of DMs goes to infinite, we address this problem and we also provide an approximation result for large non-convex teams. 

Before presenting our main results for \PIN, we state the following lemma which will be used in the proof of Theorem \ref{the:6}. In the following key lemma, we show that to study the limit of the optimal expected cost for \PN\ within $N$-exchangeable policies, we can without any loss restrict the optimal policies to infinitely exchangeable policies. The proof of the following lemma is provided in Appendix \ref{app:lem:3}.

\begin{lemma}\label{lem:3}
    Consider \PN\ and \PIN\ under   Assumption \ref{assump:existence}, Assumption \ref{assump1}, and Assumption \ref{assum:5-4}.  Let $\mathbb{U}$ be compact and $b_t(\cdot,\cdot)$ and $\sigma_t(\cdot)$ in \eqref{eq:st-dy-N} are uniformly bounded. Suppose further that $\widehat{c}(\cdot, \ldots, \cdot)$ is uniformly bounded. Then,
\begin{align*}
    &\limsup_{N\to \infty}\inf_{P_{\pi}^{N}\in \LEXN}\int \left(\widehat{c}_{N}(\pmb{X}^{1:N},\pmb{\gamma^{1:N}})\prod_{i=1}^{N}\mu^i(d\pmb{X^{i}};\pmb{\gamma^{i}})\right) P_{\pi}^{N}(d\pmb{\gamma^{1:N}})\\
    &=\limsup_{N\to \infty}\inf_{P_{\pi}\in \LEX}\int \left(\widehat{c}_{N}(\pmb{X}^{1:N},\pmb{\gamma^{1:N}})\prod_{i=1}^{N}\mu^i(d\pmb{X^{i}};\pmb{\gamma^{i}})\right) P_{\pi,N}(d\pmb{\gamma^{1:N}}),
\end{align*}
where $P_{\pi,N}$ is the marginal of the $P_{\pi}\in \LEX$ to the first $N$ components.
\end{lemma}

Our approach for establishing the existence of an optimal policy that is symmetric for \PIN\ is to use a sequence of optimal policies for \PN\ (which can be assumed to be exchangeable by Theorem \ref{the:5}). Lemma \ref{lem:3} allows us to work with a sequence of infinitely exchangeable policies instead. This simplifies our convergence analysis thanks to a de Finetti representation theorem in Lemma \ref{lem:0}. We note that we cannot extend any exchangeable policy in $\LEXN$ to an infinitely exchangeable policy in $\LEX$ with marginals on the first $N$ components coinciding with the given exchangeable policy in $\LEXN$. The proof of Lemma \ref{lem:3}, utilizes two key results by Diaconis and Friedman \cite[Theorem 13]{diaconis1980finite} and Aldous \cite[Proposition 7.20]{aldous2006ecole}, regarding approximations of exchangeable random variables and convergence properties of exchangeable and infinitely exchangeable random variables. In particular, using \cite[Theorem 13]{diaconis1980finite} (see Theorem \ref{the:dia} in  Appendix), we show that a sequence of exchangeable optimal policies for \PN\ admits a subsequence that converges in distribution to an exchangeable policy. In addition, using \cite[Proposition 7.20]{aldous2006ecole}, we conclude that the sequence of empirical measures of converging subsequence of exchangeable optimal policies for \PN\ also converges to a directing measure of an exchangeable policy (see Theorem \ref{the:ald} in  Appendix). This then using continuity of the cost function concludes the proof.

Next using Theorem \ref{the:5},  Lemma \ref{lem:0}, and Lemma \ref{lem:3}, we establish the existence of a randomized globally optimal policy for \PIN\ that is symmetric and independently randomized. The proof of the following theorem is provided in Appendix \ref{app:the:6}.

\begin{theorem}\label{the:6}
    Consider \PIN\ under Assumption \ref{assump:existence}, Assumption \ref{assump1}, and Assumption \ref{assum:5-4}. Let $\mathbb{U}$ be compact and $b_t(\cdot,\cdot)$ and $\sigma_t(\cdot)$ in \eqref{eq:st-dy-N} be uniformly bounded. Then,  there exists a symmetric privately randomized globally optimal policy $P_{\pi}^{\star}$ for \PIN\ among all randomized policies $L$, i.e.,   $P_{\pi}^{\star}\in \LPRS$.

\end{theorem}

We now provide an insight for the proof of Theorem \ref{the:6}. Utilizing Theorem \ref{the:5},  Lemma \ref{lem:0}, and Lemma \ref{lem:3}, we can show that for studying the limiting behavior of the optimal expected cost of \PN, we can restrict the policies to infinitely exchangeable policies which are conditionally identical and independent, i.e., they belong to  $\LCOS$ restricted to $N$ first components $\LCOS|_N=\LCOSN$ (following Lemma \ref{lem:0}(i)). Since the set $\LCOSN$ is convex and the expected cost is linear in randomized policies, an optimal solution among all $\LCOSN$ lies in its set of extreme points which is the set of independently randomized symmetric policies $\LPRS^N$. This allows us to study the limiting behavior of the optimal expected cost of \PN\ only under the randomized symmetric policies $\LPRS^N$. Thanks to the symmetry and continuity of the cost, we then establish the convergence of the optimal expected cost of \PN\ to the optimal expected cost of its mean-field limit \PIN.

As a result of Theorem \ref{the:6}(i), for \PIN, we can restrict our search for a globally optimal policy to independently randomized symmetric policies. We emphasize that independently randomized symmetric policies are a strict subset of finite exchangeable policies, yielding that Theorem \ref{the:6} provides a stronger existence and structural result than Theorem \ref{the:5}.

Existence of a globally optimal solution for \PIN\ that is symmetric and independently randomized, allows us to represent the infinite population mean-field problem as a single representative DM control problem by replacing the empirical measure of i.i.d. states and actions as their limits $\mu_{X_t^{\sf R}}=\mathcal{L}(X_t^{\sf R})$ and $\mu_{U_t^{\sf R}}=\mathcal{L}(U_t^{\sf R})$. In contrast to the convex team, where we show that an optimal policy is symmetric and Markovian, for non-convex teams, we showed that an optimal solution is symmetric and independently randomized among wide-sense admissible policies. By fixing other DM's policies to these symmetric and independently randomized, the deviating DM still faces a decoupled stochastic control problem since policies are independently randomized and dynamics are decoupled. The fact that the optimization is over wide-sense admissible policies leads to additional complications for the corresponding HJB or master equations for the representative DM as the symmetric policies can be path-dependent. Nevertheless, the corresponding HJB or master equations with consistency conditions for the representative DM can be derived to compute a symmetric optimal policy provided that sufficient regularity conditions hold.

In the following theorem, we use an independently randomized symmetric optimal solution of \PIN\ to provide approximations for an optimal solution of \PN\ for large $N$. The proof of Theorem \ref{the:7} is included in the Appendix \ref{APP:pfthe7}.

\begin{theorem}\label{the:7}
    Consider \PN\ and \PIN\ under Assumption \ref{assump:existence}, Assumption \ref{assump1}, and Assumption \ref{assum:5-4}. Let $\mathbb{U}$ be compact and $b_t(\cdot,\cdot)$ and $\sigma_t(\cdot)$ in \eqref{eq:st-dy-N} be uniformly bounded. Then, there exists an $\epsilon_N$ that converges to $0$ as $N\to \infty$ such that a symmetric privately randomized globally optimal policy $P_{\pi}^{\star}\in \LPRS$ for \PIN\ is $\epsilon_N$-optimal for \PN, i.e.,
    \begin{align}\label{eq:approx-opt}
        \inf_{P_{\pi}^{N}\in L^N}J_{N}^{\pi}(P_{\pi}^{N})\geq J_{N}^{\pi}(P_{\pi,N}^{\star})-\epsilon_N,
    \end{align}
    where $P_{\pi,N}^{\star}$ is the restriction of $P_{\pi}^{\star}$ to its first $N$-components.
\end{theorem}

The proof of Theorem \ref{the:7} follows from our convergence analysis in the proof of Theorem \ref{the:6}. As an implication of Theorem \ref{the:7}, we can restrict our search for an approximate optimal solution of \PN\ for large $N$ to independently randomized symmetric policies. An optimal solution of the mean-field team provides an approximation for an independently randomized symmetric optimal solution that is scalable with $N$. The precision of this approximation depends on $N$. 

Throughout our analysis in this section, we used the fact that dynamics are decoupled. This allows us to establish continuity of the probability measure on the path space of states in randomized policies using martingale equations. In addition, the state processes are independent and identical among DMs under independently randomized symmetric policies since dynamics are decoupled. This is utilized in our convergence analysis in Theorem \ref{the:5}. In the next section, we allow coupling in dynamics which requires additional technical arguments.

\section{Exchangeable (Non-Convex) Stochastic Teams with Coupled Dynamics: Optimality and Near Optimality of Symmetric Policies}\label{sec:coup}

In the previous sections, we studied teams with decoupled dynamics. In this section, we relax this assumption by allowing the dynamics to be weakly coupled among DMs through the empirical measures of states and actions, i.e., each DM has access to only a private state process that evolves as a controlled diffusion process given by the solution of the following SDE, i.e.,
\begin{align}
    dX_{t}^i = b_{t}\left(X_{t}^i,U_{t}^i, \frac{1}{N}\sum_{p=1}^{N} \delta_{X^{p}_t},\frac{1}{N}\sum_{p=1}^{N} \delta_{U^{p}_t}\right)dt + \sigma_t(X_{t}^{i})dW_{t}^{i}, \quad t\in \mathbb{T}.\label{eq:st-dy-C}
\end{align}

{We assume that $b_t(\cdot,\cdot, \cdot, \cdot)$ and $\sigma_t(\cdot)$ in \eqref{eq:st-dy-C} satisfy conditions to guarantee the existence of a unique strong solution under non-wide-sense admissible policies.
\begin{assumption}\label{assump:existence2}
The following three conditions hold:
\begin{itemize}[wide]
    
    \item [(i)-Lipschitz continuity:] For every $t\in \mathbb{T}$, $b_t:\mathbb{R}^d \times \mathbb{U} \times \mathcal{P}(\mathbb{R}^d) \times \mathcal{P}(\mathbb{U})\to \mathbb{R}^d$ and $\sigma_{t}:\mathbb{R}^d \to \mathbb{R}^{d\times d}$ are  Lipschitz continuous as follows 
    \begin{align}
     \sup_{u\in \mathbb{U},q\in \mathcal{P}(\mathbb{U})}|b_{t}(x,u,\mu, q)- b_{t}(x^{\prime},u,\mu^{\prime},q)| +  \|\sigma_{t}(x)-\sigma_t(y)\| \leq C\left(|x-y|+\mathcal{W}_{2}(\mu,\mu^{\prime})\right),
    \end{align}
    for all $x,x^{\prime}\in \mathbb{R}^d$ and some constant $C>0$.  In the above, $\mathcal{W}_{2}(\cdot,\cdot)$ is the $2$-Wasserstein distance on the subspace of probability measures  order $2$ ($\mathcal{P}_{2}(\mathbb{R}^d)$), and  $\|\sigma_{t}(x)\|:=\sqrt{{\sf Tr}[\sigma^T_t(x)\sigma_t(x)]}$. 
\item [(ii)-Boundedness:] For every $t\in \mathbb{T}$, $b_t$ and $\sigma_t$ are uniformly bounded, i.e.,
\begin{align*}
\sup_{u\in \mathbb{U}, q\in \mathcal{P}(\mathbb{U})}|b_{t}(x,u,p,q)|^2 + \|\sigma_{t}(x)\|^2 \leq C_{0}\qquad \forall (x,p)\in \mathbb{R}^d \times \mathcal{P}(\mathbb{R}^d)
\end{align*}
for some constant $C_{0}>0$.
\item [(iii)] Assumption \ref{assump:existence}(iii) holds.
 \end{itemize}
 \end{assumption}
 
First, under Assumption \ref{assump:existence2}, following \cite[Theorem 2.2.4]{arapostathis2012ergodic},   \eqref{eq:st-dy-C} admits a unique strong solution under wide-sense admissible policies. In this section, we use Girsonov's theorem \cite{girsanov1960transforming, benevs1971existence} to (probabilistically) decouple the dynamics among DMs and transfer the states correlations to the cost, leading to a more complicated cost with an exponential Radon-Nikodym derivative term. By transferring coupling to the cost function, we can follow the approaches developed in the previous sections with some minor modifications.}

Denote the induced probability measure on $\pmb{X^{1:N}}$ by $P_N$ with the marginal on $\pmb{X^i}$ by $P^i_N$. Girsonov's theorem \cite{girsanov1960transforming, benevs1971existence} enables us to  transform the probability measure $P_N$ to another probability measure $P_{0}:=\prod_{i=1}^{N}P^{i}_{0}$ as the induced probability measure on the path space of states driven by the following SDE
 \begin{align}\label{eq:mf-dynamic-dec}
    dX_{t}^i = b_{t}\left(X_{t}^i,U_{t}^i, \mu^{X}_{t},\mu^{U}_{t}\right)dt + \sigma_t(X_{t}^{i})dW_{t}^{i}, \quad t\in \mathbb{T},
\end{align}
for every $i\in \mathcal{N}$, where $\mu^{X}_{t}$  and $\mu^{U}_{t}$  are the law of random variables $X_{t}^i$ and $U_t^i$ for all $i\in \mathcal{N}$ and $t\in \mathbb{T}$, i.e., $\mu^{X}_{t}=\mathcal{L}(X_{t}^i)$ and $\mu^{U}_{t}=\mathcal{L}(U_{t}^i)$. The above measure transformation applies provided that
\begin{align}\label{eq:novikov}
    \mathbb{E}^{P^{i}_{0}}\left[{\sf exp}\left(\frac{1}{2}\int_{0}^{T} \left|\sigma^{-1}_t(X_{t}^{i})\overline{b_{t}}\left(X_{t}^i,U_{t}^i, \frac{1}{N}\sum_{p=1}^{N} \delta_{X^{p}_t},\frac{1}{N}\sum_{p=1}^{N} \delta_{U^{p}_t}\right)\right|^2dt\right)\right]< \infty\quad \forall i\in \mathcal{N}
\end{align}
uniform over admissible policies, where 
\begin{align}
&\overline{b_{t}}\left(X_{t}^i,U_{t}^i, \frac{1}{N}\sum_{p=1}^{N} \delta_{X^{p}_t},\frac{1}{N}\sum_{p=1}^{N} \delta_{U^{p}_t}\right)\label{eq:overb}\\
&:=b_{t}\left(X_{t}^i,U_{t}^i, \frac{1}{N}\sum_{p=1}^{N} \delta_{X^{p}_t},\frac{1}{N}\sum_{p=1}^{N} \delta_{U^{p}_t}\right)-b_{t}\left(X_{t}^i,U_{t}^i, \mu^{X}_{t},\mu^{U}_{t}\right)\nonumber.
\end{align}
Under this measure transformation, $\pmb{X^{1:N}}$ are independent among DMs 
 with marginals $P^{i}_{0}$ for all $i\in \mathcal{N}$. Since $b_t$ is uniformly integrable over all admissible policies and $\sigma_t$ are uniformly bounded with bounded inverse, \eqref{eq:novikov} holds, and hence, the following Radon-Nikodym derivative term is well-defined
\begin{align*}
    \frac{dP_N}{dP_{0}}:=\prod_{i=1}^{N}\widehat{Z}_{N}^i&:=\prod_{i=1}^{N}\bigg[{\sf exp}\bigg(\int_{0}^{T} \sigma^{-1}_t(X_{t}^{i})\overline{b_{t}}\left(X_{t}^i,U_{t}^i, \frac{1}{N}\sum_{p=1}^{N} \delta_{X^{p}_t},\frac{1}{N}\sum_{p=1}^{N} \delta_{U^{p}_t}\right)dW_{t}^{i}\\
    &-\frac{1}{2}\int_{0}^{T} \left|\sigma^{-1}_t(X_{t}^{i})\overline{b_{t}}\left(X_{t}^i,U_{t}^i, \frac{1}{N}\sum_{p=1}^{N} \delta_{X^{p}_t},\frac{1}{N}\sum_{p=1}^{N} \delta_{U^{p}_t}\right)\right|^2dt\bigg)\bigg].
\end{align*}
Under this measure transformation, the new cost function for each DM$^i$ is given by
\begin{align}
\widetilde{c}\left(X_{t}^{i},U_{t}^{i},\frac{1}{N}\sum_{p=1}^{N} \delta_{X_{t}^{p}},\frac{1}{N}\sum_{p=1}^{N} \delta_{U_{t}^{p}} \right):= \widehat{c}\left(X_{t}^{i},U_{t}^{i},\frac{1}{N}\sum_{p=1}^{N} \delta_{X_{t}^{p}},\frac{1}{N}\sum_{p=1}^{N} \delta_{U_{t}^{p}} \right) \prod_{i=1}^{N}\widehat{Z}_{N}^i.\label{eq:MFT-cost-ind}
\end{align}
We note that under the above change of measure, \PN\ and \PIN\ can be redefined such that the cost is replaced with the new cost function and the state processes are independent among DMs and evolve according to SDE in \eqref{eq:mf-dynamic-dec}. 

Similar to Section \ref{sec:set-Ran}, for each DM$^{i}$ ($i\in \mathcal{N}$), we re-define $\Gamma^i$ as the set of all wide-sense admissible control policies so that the random measure-valued variables $\pmb{\gamma^{i}}(\omega)\in  \mathcal{P}([0,T]\times \mathbb{U})$ satisfy the following two conditions:
\begin{itemize}
    \item [(i)] Marginals on $[0,T]$ are fixed to be the Lebesgue measure.
    \item [(ii)]  $\gamma^{i}_{t}:\omega\mapsto \pmb{\gamma^{i}}(\cdot|t)(\omega)$ is independent of $W^{i}_{s}-W^{i}_{t}$ for $s>t$  and for any $t\in \mathbb{T}$, and  $W^{j}_{[0,T]}$ for any $j\in \mathcal{N}$, $j\not=i$. 
\end{itemize}

   We equip $\Gamma^i\subseteq\mathcal{P}([0,T]\times \mathbb{U})$ with the Young topology. We consider $\mathcal{C}([0, T];\mathbb{X})$-valued $\pmb{X}^{i}$ process
 (under sup-norm) induced by $\pmb{\gamma^i}(d{U},dt)(\omega)$ that solves SDE in \eqref{eq:mf-dynamic-dec}, and hence, for any $x_{0}\in \mathbb{X}$, $t\in \mathbb{T}$, and  twice continuously differentiable functions $f:\mathbb{X}\to \mathbb{R}$ 
    \begin{align}\label{eq:martin-2}
        \mathbb{E}\left[f(X_{t}^{i})-f(X_{0}^i)-\int_{0}^{t}\mathcal{A}^{{U^i}}f(X_{s}^{i})\pmb{\gamma^{i}}(d{U^{i}},ds)\bigg|X_{0}^{i}=x_0\right] = 0,
        \end{align}
    where $\mathcal{A}^{{U^i}}$ is the generator of the transition semigroup of $\pmb{X^{i}}$, i.e.,  
\begin{align}\label{eq:Gen-coup}
\mathcal{A}^{{U^i}}(X_{s}^{i})&:={\sf Tr}\left(a_{s}(X_{s}^i)\nabla^2 f(X_{s}^{i})\right)+b_{s}(X_{s}^{i},{U^i}, \mathcal{L}(X_{s}^{i}),\mathcal{L}({U^i}))\cdot \nabla f(X_{s}^{i}),
\end{align}
where $a_{s}:=\frac{1}{2}\sigma\sigma^{T}$. Using $\Gamma^i$, similar to Section \ref{sec:set-Ran}, we can define the sets of randomized policies $L^{N},\LEXN,\LCOSN,\LPRS^{N},L,\LCO,\LEX,\LCOS$, and $\LPRS$.

In the following, we closely follow our approach in the previous section but we utilize the change of measure argument. We first show a similar result as that established in Lemma \ref{lem:2}, where we showed that the induced probability measure on the path space is continuous in randomized policies. The proof of the following lemma is provided in Appendix \ref{app:lem5}.

\begin{lemma}\label{lem:2-coup}
 Suppose that Assumption \ref{assump:existence2} 
 holds. Let $\mathbb{U}$ be compact. If $\{P_{\pi, n}\}_{n}\subseteq\LEXN$ (or $\subseteq\LEX$) converges to $P_{\pi, \infty}\in\LEXN$ (or $\in\LEX$), then, the sequence of measures on the state process $\{P^{i}_{0, n}\}_{n}$ for each DM$^i$ induced by $\{P_{\pi, n}\}_{n}$ converges weakly to $P^{i}_{0, \infty}$, which is the measure on the state process under $P_{\pi, \infty}$. 
\end{lemma}

The main idea of the proof of Lemma \ref{lem:2-coup} is the same as Lemma \ref{lem:2} since under the change of measure argument, we show that the limiting state process satisfies the martingale equation in \eqref{eq:martin-2}. This is because, under the measure transformation, state processes are independent among DMs, i.e., $P_{0,n}:=\prod_{i=1}^{N}P_{0,n}^{i}$. This concludes the proof since the unique weak solution of the martingale equation in \eqref{eq:martin-2} characterizes the unique weak solution of \eqref{eq:mf-dynamic-dec} following the existence and uniqueness of a strong solution in \cite[Theorem 4.21]{carmona2018probabilistic} for McKean-Vlasov SDE of \eqref{eq:martin-2}.

In the following theorem, we establish the existence of a randomized globally optimal policy for \PN\ that is exchangeable. We first present a lemma on the $L^1$-continuity of the product of Radon-Nikodym derivative terms building on \cite[Lemma 2.5]{pradhan2023controlled} that will be used in our theorem below. We additionally note the result in \cite{charalambous2013dynamic}, where the existence of a globally optimal solution for a class of finite population teams is established using Girsanov’s change of measure argument. As an intermediate step, the authors showed that the Radon-Nikodym derivative is continuous as a function of policy under the weak*-topology (see \cite[Lemma 1]{charalambous2013dynamic}). Also, in \cite[Lemma 1]{charalambous2016decentralized}, the continuity of the Radon-Nikodym derivative terms in policies is established using a different topology from that introduced here, and $L^2$ convergence of the policies, where the authors assumed that information structure is defined such that external noise processes do not corrupt the observations. Here, our goal is to establish the existence of an optimal solution and show that it is exchangeable, where we use $L^1$-continuity of the product of Radon-Nikodym derivative terms. The proof of the following lemma is essentially from \cite[Lemma 2.5]{pradhan2023controlled} and is provided in Appendix \ref{app:lem-con} for completeness.

\begin{lemma}\label{lem:contZ}
     Consider \PN\ with coupled dynamics \eqref{eq:st-dy-C}. Let Assumption \ref{assump:existence2} and Assumption \ref{assum:5-4} hold.  Suppose that $\{\pmb{\gamma^{i}_n}(\omega)\}_{n}$ together with its induced distribution on states $\{\pmb{X^i_{n}}(\omega)\}_n$ converge weakly to $\pmb{\gamma^{i}}(\omega)$ and the distribution of $\pmb{X^i}(\omega)$ almost surely (in $\omega$) as $n\to \infty$. Let
$Z^{i,n}_{N}:={\sf exp}(A_n^i-B_n^i)$ and $Z^{i,\infty}_{N}:={\sf exp}(A^i-B^i)$ with
\begin{align}
    A_{n}^{i}&=\int_{0}^{T} \sigma^{-1}_t(X_{t,n}^{i})\int_{\mathbb{U}}\overline{b_{t}}\left(X_{t,n}^i,U_{t}^i, \frac{1}{N}\sum_{p=1}^{N} \delta_{X^{p}_{t,n}},\frac{1}{N}\sum_{p=1}^{N} \delta_{U^{p}_t}\right)\prod_{i=1}^{N}\pmb{\gamma^i_n}(d{U^i}|t)dW_{t}^{i}\label{eq:A-B1}\\
    B_{n}^{i}&=\frac{1}{2}\int_{0}^{T} \left|\sigma^{-1}_t(X_{t,n}^{i})\int_{\mathbb{U}}\overline{b_{t}}\left(X_{t,n}^i,U_{t}^i, \frac{1}{N}\sum_{p=1}^{N} \delta_{X^{p}_{t,n}},\frac{1}{N}\sum_{p=1}^{N} \delta_{U^{p}_t}\right)\prod_{i=1}^{N}\pmb{\gamma^i_n}(d{U^i}|t)\right|^2dt\label{eq:A-B2}\\
    A^{i}&=\int_{0}^{T} \sigma^{-1}_t(X_{t}^{i})\int_{\mathbb{U}}\overline{b_{t}}\left(X_{t}^i,U_{t}^i, \frac{1}{N}\sum_{p=1}^{N} \delta_{X^{p}_{t}},\frac{1}{N}\sum_{p=1}^{N} \delta_{U^{p}_t}\right)\prod_{i=1}^{N}\pmb{\gamma^i}(d{U^i}|t)dW_{t}^{i}\label{eq:A-B3}\\
    B^i&=\frac{1}{2}\int_{0}^{T} \left|\sigma^{-1}_t(X_{t}^{i})\int_{\mathbb{U}}\overline{b_{t}}\left(X_{t}^i,U_{t}^i, \frac{1}{N}\sum_{p=1}^{N} \delta_{X^{p}_{t}},\frac{1}{N}\sum_{p=1}^{N} \delta_{U^{p}_t}\right)\prod_{i=1}^{N}\pmb{\gamma^i}(d{U^i}|t)\right|^2dt\label{eq:A-B4}.
\end{align}
Then, $\prod_{i=1}^{N}Z^{i,n}_{N}$ converges in $L^1(\mathbb{P})$ to $\prod_{i=1}^{N}Z^{i,\infty}_{N}$.
\end{lemma}

The proof of the following theorem is provided in Appendix \ref{app:the8}.

\begin{theorem}\label{the:5-coup}
     Consider \PN\ with coupled dynamics \eqref{eq:st-dy-C}. Let Assumption \ref{assump:existence2}, Assumption \ref{assump1}, and Assumption \ref{assum:5-4} hold.  Suppose that $\mathbb{U}$ is compact. Then, there exists an exchangeable randomized globally optimal policy ${P_{\pi}^{N\star}}$ for \PN\ among all randomized policies $L^{N}$, i.e.,   $P_{\pi}^{N\star}\in \LEXN$.
\end{theorem}

The proof of Theorem \ref{the:5-coup} is essentially the same as the proof of Theorem \ref{the:5} except that we additionally need to show $L^1$-continuity of the product of Radon-Nikodym derivative terms in policies. We establish this in Lemma \ref{lem:contZ} using the generalized dominated convergence theorem and It\^o isometry following a similar technique as that used in \cite[Lemma 2.5]{pradhan2023controlled}. {We note that in contrast to the decoupled setting, for the coupled case, we cannot derive HJB equations for \PN\ since states and randomization in policies are correlated among DMs. Also, we cannot view the entire finite population as a single DM viewing a classical stochastic control problem since the information structure is decentralized, and hence, the solution of the corresponding HJB equations will not be admissible under the decentralized information structure.}

In the following key lemma, we show that to study the limit of the optimal expected cost for \PN\ within $N$-exchangeable policies, we can without any loss restrict our search to infinitely exchangeable policies. The proof of the following lemma is provided in Appendix \ref{app:lem6}. Our result uses the following assumption.

\begin{assumption}\label{eq:continuity-RD}
    For any sequence $\{\pmb{\gamma}^{1:N}_{N}\}$ of $N$-exchangeable random variables (policies) that converge in distribution, $\mathbb{E}\left[\left|\prod_{i=1}^{N}Z^{i,N}_{N}-1\right|\right]$ converges to $0$ as $N\to \infty$, where
\begin{align}
Z^{i,N}_{N}&:={\sf exp}\bigg(\int_{0}^{T} \sigma^{-1}_t(X_{t,N}^{i})\overline{b_{t}}\left(X_{t,N}^i,U_{t}^i, \frac{1}{N}\sum_{p=1}^{N} \delta_{X^{p}_{t,N}},\frac{1}{N}\sum_{p=1}^{N} \delta_{U^{p}_t}\right)\prod_{i-1}^{N}\gamma^{i}_{N}(dU^i|t)dW_{t}^{i}\nonumber\\
    &-\frac{1}{2}\int_{0}^{T} \left|\sigma^{-1}_t(X_{t,N}^{i})\overline{b_{t}}\left(X_{t,N}^i,U_{t}^i, \frac{1}{N}\sum_{p=1}^{N} \delta_{X^{p}_{t,N}},\frac{1}{N}\sum_{p=1}^{N} \delta_{U^{p}_t}\right)\prod_{i-1}^{N}\gamma^{i}_{N}(dU^i|t)\right|^2dt\bigg)\label{eq:ZNN},
\end{align}
with $\overline{b_{t}}$ as in \eqref{eq:overb}.
\end{assumption}

\begin{lemma}\label{lem:3-coup}
    Consider \PN\ and \PIN\ with coupled dynamics \eqref{eq:st-dy-C}. Let Assumption \ref{assump:existence2}, Assumption \ref{assump1}, Assumption \ref{assum:5-4}, Assumption \ref{eq:continuity-RD} hold.  Suppose that $\mathbb{U}$ is compact. Suppose further that $c(\cdot, \ldots, \cdot)$ is uniformly bounded. Then,
\begin{align*}
    &\limsup_{N\to \infty}\inf_{P_{\pi}^{N}\in \LEXN}\int \left(\widehat{c}_{N}(\pmb{X}^{1:N},\pmb{\gamma^{1:N}})\prod_{i=1}^{N}P^i_{0,N}(d\pmb{X^{i}};\pmb{\gamma^{i}})Z_{N}^{i,N}\right) P_{\pi}^{N}(d\pmb{\gamma^{1:N}})\\
    &=\limsup_{N\to \infty}\inf_{P_{\pi}\in \LEX}\int \left(\widehat{c}_{N}(\pmb{X}^{1:N},\pmb{\gamma^{1:N}})\prod_{i=1}^{N}P^i_{0,N}(d\pmb{X^{i}};\pmb{\gamma^{i}})Z_{N}^{i,N}\right)  P_{\pi,N}(d\pmb{\gamma^{1:N}}),
\end{align*}
where $P_{\pi,N}$ is the marginal of the $P_{\pi}\in \LEX$ to the first $N$ components.
\end{lemma}

Lemma \ref{lem:3-coup} similar to its counterpart Lemma \ref{lem:3} for teams with decoupled dynamics, allows us to focus our convergence analysis on a sequence of infinitely exchangeable policies. In Lemma \ref{lem:3}, the probability measures on the path of states are independent among DMs, and hence, we do not need to use the change of measure argument. In Lemma \ref{lem:3-coup}, the dynamics are coupled and hence, we used the change of measure argument to make probability measures on the path of states independent among DMs. As a result, our analysis of the optimal expected cost requires Assumption \ref{eq:continuity-RD} which implies $L^1$-continuity of the product of Radon-Nikodym derivative terms in \eqref{eq:ZNN} as $N\to \infty$. We note that in the proof of Theorem \ref{the:5-coup}, we established that the finite product of Radon-Nikodym derivative terms is $L^1$-continuous; however, establishing continuity as the number of DMs drives to infinity requires additional conditions. 

In the following theorem, we establish the existence of a randomized globally optimal policy for \PIN\ that is symmetric and privately randomized. The proof of the following theorem is provided in Appendix \ref{app:the9}.

\begin{theorem}\label{the:6-coup}
    Consider \PIN\ with coupled dynamics \eqref{eq:st-dy-C}. Let Assumption \ref{assump:existence2}, Assumption \ref{assump1}, Assumption \ref{assum:5-4}, and Assumption \ref{eq:continuity-RD} hold.  Suppose that $\mathbb{U}$ is compact. Then,  there exists a globally optimal policy $P_{\pi}^{\star}$ for \PIN\ among all randomized policies $L$ that is symmetric and privately randomized, i.e.,   $P_{\pi}^{\star}\in \LPRS$.
    
\end{theorem}

{Similar to the decoupled case, the existence of a globally optimal solution for \PIN\ that is symmetric and independently randomized, leads to a more efficient computation method for an optimal policy, by representing the infinite population MF problem as a single representative DM control problem with the empirical measure of i.i.d. states and actions are replaced with their limits. However, the corresponding HJB or master equations with consistency conditions for the representative DM would be more complicated because the dynamics are not decoupled.}

In the following theorem, we use the privately randomized symmetric optimal solution of \PIN\ to provide approximations for an optimal solution of \PN\ for large $N$. 

\begin{theorem}\label{the:7-coup}
    Consider \PN\ and \PIN. Let Assumption \ref{assump:existence2}, Assumption \ref{assump1}, Assumption \ref{assum:5-4}, and Assumption \ref{eq:continuity-RD} hold.  Suppose that $\mathbb{U}$ is compact. Then, there exists an $\epsilon_N$ that converges to $0$ as $N\to \infty$ such that a symmetric privately randomized globally optimal policy $P_{\pi}^{\star}\in \LPRS$ for \PIN\ is $\epsilon_N$-optimal for \PN, i.e., \eqref{eq:approx-opt} holds.
\end{theorem}

The proof of Theorem \ref{the:7-coup} follows from analogous reasoning as that used in the proof of Theorem \ref{the:7}. 

\section{Conclusion}
We have studied a class of continuous-time stochastic teams with a finite number of DMs and their MF limit with an infinite number of DMs. We have established the existence of a globally optimal solution for $N$-DM convex teams and have shown that this globally optimal solution is  Markovian and symmetric. For the convex MF
teams with an infinite number of DM, we have established the existence of a possibly randomized globally optimal solution, and
have shown that it is Markovian and symmetric. For the non-convex team with a finite population of DMs, we have established the existence of an exchangeable randomized optimal policy. As the number of DMs drives to infinity, we show the optimality of a privately randomized symmetric optimal policy for the MF team. Finally, we have shown that a privately randomized symmetric optimal policy for the MF team is approximately optimal for the corresponding exchangeable team with a finite but large number of DMs.


\appendix

\section{An Example for Assumption \ref{assump:convex}}\label{sec:examples}
In the following, we provide an example under which Assumption \ref{assump:convex} holds. Consider \PN\ with linear dynamics
  \begin{align}\label{eq:lin-lqg}
    dX_{t}^i = [AX_{t}^i + B U_{t}^i] dt + \sigma dW_{t}^{i}, \quad t\in \mathbb{T}, \quad i\in \mathcal{N},
\end{align}
where $A,B,\sigma$ are bounded, and cost is quadratic 
\begin{align*}
c\left(X_{t}^{1:N}, U_{t}^{1:N}\right):=&\begin{bmatrix}
 (X_{t}^{1:N})   
\end{bmatrix}^{'}
Q \begin{bmatrix}
 X_{t}^{1:N}    
\end{bmatrix} +\begin{bmatrix}
 (U_{t}^{1:N})    
\end{bmatrix}^{'}
R \begin{bmatrix}
 U_{t}^{1:N}   
\end{bmatrix} ,  
\end{align*}
where $Q\geq 0$ and $R>0$ are symmetric bounded matrices, and $V^{\prime}$ denotes the transpose of the vector $V$. For this class of linear quadratic teams, Assumption \ref{assump:convex} holds under the isomorphic open-loop information structure. 

Let $\mathcal{F}^{t}_{X^{i}}$ and $\mathcal{F}^{t}_{X_{0}^i,W^i}$ denote the natural filtration of $X_{[0,t]}^{i}$ and $\{X_{0}^{i},W_{s}^{i}, \forall s\leq t\}$, respectively. We first show that for any admissible policy with the induced actions process adapted to $\mathcal{F}^{t}_{X^{i}}$, there exists another policy with the induced action process adapted to $\mathcal{F}^{t}_{X_{0}^i,W^i}$ that attains identical team cost and vice versa. We have  $\mathcal{F}^{t}_{X^{i}}\subseteq \mathcal{F}^{t}_{X_{0}^i,W^i}$. Now, we show that $\mathcal{F}^{t}_{X_{0}^i,W^i}\subseteq\mathcal{F}^{t}_{X^{i}}$. For every $t\in \mathbb{T}$, we have
 \begin{align}
    {W}_{t}^i &= \sigma^{-1}\left({X}_{t}^i - {X}_{0}^i-\int_{0}^{t} \left[A{X}_{\tau}^i + B{U}_{\tau}^i\right] d\tau\right)\label{eq:ex2-cv-dyn1}. 
\end{align}
This implies that $\mathcal{F}^{t}_{X_{0}^i,W^i}\subseteq\mathcal{F}^{t}_{X^{i}}$, and hence, $\mathcal{F}^{t}_{X_{0}^i,W^i}$ and $\mathcal{F}^{t}_{X^{i}}$ are weakly isomorphic in a sense that there exist isomorphic admissible policies such that the induced action and state processes under two information structures are isomorphic in distribution and the expected performance remains the same under weakly isomorphic admissible policies for two information structures for every $t\in \mathbb{T}$ and $i\in \mathcal{N}$. Under this reduction, we can rewrite the actions under two isomorphic policies almost surely. This yields that the expected cost function remains identical under this reduction. Next, we show convexity of $J(\pmb{\gamma^{1:N}})$ in $\pmb{\gamma^{1:N}}$ for any open-loop adapted policies $\pmb{\gamma^{1:N}}$, where the induced actions $\{U_{t}^{i}\}_{t}$ are adapted to $\mathcal{F}^{t}_{X_{0}^i,W^i}$ and belong to $L^{2}([0,T];\mathbb{U})$, i.e., $U_{t}^{i}=\gamma^{i}_{t}(X_{0}^i,W^i_{[0,t]})$ with $\mathbb{E}[\int_{0}^{T}|U_{s}^{i}|^2ds]<\infty$. Denote the space of such actions by $\mathcal{U}^i$ equipped with $L^{2}$-inner products  $\langle \cdot, \cdot \rangle$, making it a Hilbert space. Also, let $\mathcal{X}^i$ be the set of $X_{t}^{i}$ that is $\mathcal{F}^{t}_{X_{0}^i,W^i}$ adapted, continuous, and $\mathbb{E}[\sup_{t\in \mathbb{T}}|X_{t}^{i}|^2]<\infty$. 

Following \cite[Section 2.2]{sun2020stochastic}, for any $\pmb{U^{i}}\in \mathcal{U}^i$, we  have
\begin{align}
    X^{i}_{t}(\pmb{U^{i}})=(\phi^i X_{0}^i)(t)+(\psi^{i}\pmb{U^i})(t)\quad \forall t\in \mathbb{T}\quad i\in \mathcal{N}\label{eq:Opt-dyn}
\end{align}
for some bounded linear operators $\phi^i:\mathbb{X}\to \mathcal{X}^i$ and $\psi^i:\mathcal{U}^i\to \mathcal{X}^i$. Denote the adjoint of these operators by $\phi^{i*}$ and $\psi^{i*}$, respectively. Using \eqref{eq:Opt-dyn}, we rewrite the expected cost as
\begin{align}
    J_{N}(\underline{\pmb{U}})&=\langle M_1\underline{\pmb{U}},\underline{\pmb{U}}\rangle+2\langle M_2\underline{\pmb{X_{0}}},\underline{\pmb{U}}\rangle+\langle M_3\underline{\pmb{X_0}},\underline{\pmb{X_0}}\rangle\label{cost-func},
\end{align}
where $\underline{\pmb{U}}:=\pmb{U^{1:N}},\underline{\pmb{X_0}}:=\pmb{X_{0}^{1:N}}$ and $M_{1}:=R+ \psi^* Q \psi, M_{2}:=\phi^{*}Q\psi, M_3:=\phi^{*}Q\phi$, 
 $\psi = \begin{bmatrix}
 \psi^{1:N}   
\end{bmatrix}$, and
 $\phi = \begin{bmatrix}
 \phi^{1:N}   
\end{bmatrix}$.

Following \eqref{cost-func}, $J_{N}(\underline{\pmb{U}})$ is convex in $\underline{\pmb{U}}$ since $R>0$ and $Q\geq 0$. Hence, Assumption \ref{assump:convex} holds under the weakly isomorphic information structure.

\section{Proof of Theorem \ref{the:1}}\label{App:the1}
Consider any (possibly path-dependent) policy of $\pmb{\gamma^{-i}}\in \Gamma^{-i}$, we have
\begin{align}
&\inf_{\pmb{\gamma^i}\in \Gamma^i}\mathbb{E}^{\pmb{\gamma^{1:N}}}\left[\int_{0}^{T} c\left(X_{t}^{1:N}, U_{t}^{1:N}\right) dt\right]\nonumber\\
&=\inf_{\pmb{\gamma^i}\in \Gamma^i}\mathbb{E}^{\pmb{\gamma^{i}}}\left[\int_{0}^{T}\int \bar{c}_t\left(\omega_{0}^{i},X_{t}^{i},  U_{t}^{i}\right) P^i(d\omega_{0}^{i})dt\right]\label{eq:sec2-6}\\
&=\inf_{\pmb{\gamma^i}\in \Gamma^{i, {\sf M}}}\mathbb{E}^{\pmb{\gamma^{i}}}\left[\int_{0}^{T}\hat{c}_t^i\left(X_{t}^{i},  U_{t}^{i}\right) dt\right]\label{eq:sec2-8}\\
&=\inf_{\pmb{\gamma^i}\in \Gamma^{i, {\sf M}}}\mathbb{E}^{\pmb{\gamma^{1:N}}}\left[\int_{0}^{T} c\left(X_{t}^{1:N},U_{t}^{1:N}\right) dt\right]\label{eq:sec2-9},
\end{align}
where $\omega_{0}^{i}:=(X_{[0,T]}^{-i}, U_{[0,T]}^{-i})$, 
$\bar{c}_t\left(\omega_{0}^{i},X_{t}^{i},  U_{t}^{i}\right):=c\left(X_{t}^{1:N}, U_{t}^{1:N}\right)$, and
 \begin{align*}   
 \hat{c}_{t}^i\left(X_{t}^{i},U_{t}^{i}\right)&:=\int \bar{c}_t\left(\omega_{0}^{i},X_{t}^{i},  U_{t}^{i}\right) P^i(d\omega_{0}^{i}).
\end{align*}
The costs $\tilde{c}_{t}$ and $\hat{c}_{t}^{i}$ depend on $t$ as a result of the projection in their definitions. Equality \eqref{eq:sec2-6} is true since dynamics are decoupled and the information structure is fully decentralized, and hence, $\omega_{0}^{i}$ is independent of $\pmb{X^{i}}$ and $\pmb{U^{i}}$ for any policy $\pmb{\gamma^{-i}}$. Equality \eqref{eq:sec2-8} follows from the fact that the dynamic of DM$^i$ independent of policies of other DMs is Markovian, and hence, DM$^i$ faces a classical continuous-time Markov decision problem. Let $\epsilon\geq0$. Equalities \eqref{eq:sec2-6}--\eqref{eq:sec2-9} imply that if $(\pmb{\gamma^{i\star}}, \pmb{\gamma^{-i\star}})\in \Gamma_{N}$ is an $\epsilon$-optimal policy for \PN\, then we can replace every policy with a Markovian policy and obtain an $\epsilon$-optimal policy for \PN\ that is Markovian.
Since $\epsilon$ is arbitrary, 
 this completes the proof.

\subsection{Proof of Theorem \ref{the:2}}\label{App:the2}

\begin{itemize}[wide]
    \item [Part (i):]
For any $\pmb{\gamma^{1:N}}\in \Gamma_{N}^{\sf M}$ and $\tau\in S_{N}$, we have
\begin{align}
    J(\pmb{\gamma}^{\tau(1)}, \ldots, \pmb{\gamma}^{\tau(N)})\nonumber
    &=\int\int_{0}^{T}  c\left(X_{t}^{\tau(1), \pmb{\gamma}^{\tau(1)}:\tau(N), \pmb{\gamma}^{\tau(N)}}, U_{t}^{\tau(1), \pmb{\gamma}^{\tau(1)}:\tau(N), \pmb{\gamma}^{\tau(N)}}\right)dt\nonumber\\
    &\times P(d\pmb{X}^{\tau(1), \pmb{\gamma}^{\tau(1)}:\tau(N), \pmb{\gamma}^{\tau(N)}})\prod_{i=1}^{N} P^{\pmb{\gamma}^{\tau(i)}}(d\pmb{U}^{\tau(i), \pmb{\gamma}^{\tau(i)}}|\pmb{X}^{\tau(i), \pmb{\gamma}^{\tau(i)}})\label{eq:sec2-14}\\
    &=\int\int_{0}^{T}  c\left(X_{t}^{1, \pmb{\gamma}^{1}}, \ldots, X_{t}^{N, \pmb{\gamma}^{N}}, U_{t}^{1, \pmb{\gamma}^{1}}, \ldots, U_{t}^{N, \pmb{\gamma}^{N}}\right)dt\nonumber\\
    &\times P(d\pmb{X^{1, {\gamma}^{1}}}, \ldots, d\pmb{X^{N, {\gamma}^{N}}})\prod_{i=1}^{N} P^{\pmb{\gamma}^{i}}(d\pmb{U^{i, {\gamma}^{i}}}|\pmb{X^{i, {\gamma}^{i}}})\label{eq:sec2-15}\\
    &=J(\pmb{\gamma^{1}}, \ldots, \pmb{\gamma^{N}}),
\end{align}
where $X_{t}^{i, \pmb{\gamma}^{\tau(i)}}$ denotes the state of $X_{t}^{i}$ under the policy $\pmb{\gamma}^{\tau(i)}$ for  DM$^i$, and $U_{t}^{i, \pmb{\gamma}^{\tau(i)}}$ is the action of DM$^i$ under the policy $\pmb{\gamma}^{\tau(i)}$ and the state $X_{t}^{i, \pmb{\gamma}^{\tau(i)}}$. In the above, $P^{\pmb{a^i}}$ is the conditional distribution of $\pmb{U^{i, a^i}}$ given $\pmb{X^{i, a^i}}$.
Equality \eqref{eq:sec2-14} follows from relabeling $u^{i}$, $X_{0}^{i}$, $W_{t}^{i}$ by $u^{\tau(i)}$, $X_{0}^{\tau(i)}$, $W_{t}^{\tau(i)}$ for all $t\in \mathbb{T}$, respectively. Equality \eqref{eq:sec2-15} follows from Assumption \ref{assump1}, and the fact that $b_{t}$  and $\sigma_{t}$ in \eqref{eq:st-dy-N} are identical among DMs. 

\item [Part (ii):]
Given $\pmb{\gamma^{1\star:N\star}}\in\Gamma_{N}^{\sf M}$, construct the following symmetric policy $\bar{\pmb{\gamma}}:=(\pmb{\gamma^{\star}}, \ldots, \pmb{\gamma^{\star}})$ as follows 
\begin{align}
{\gamma}^{\star}_{t}(\cdot) := \frac{1}{|S_{N}|}\sum_{\tau\in S_{N}}{\gamma}^{\tau(i)\star}_{t}(\cdot) \qquad \forall t\in \mathbb{T},\quad i\in \mathcal{N}.
\end{align}
We note that $\bar{\pmb{\gamma}}\in \Gamma_{N}^{\sf SYM,M}$.
Under Assumption \ref{assump:convex}, using Jensen's inequality, we have
\begin{align*}
J\left(\bar{\pmb{\gamma}}\right)&=J\left(\frac{1}{|S_{N}|}\sum_{\tau\in S_{N}}\pmb{\gamma^{\tau(1)\star}}, \ldots, \frac{1}{|S_{N}|}\sum_{\tau\in S_{N}}\pmb{\gamma^{\tau(N)\star}}\right)\\
&\leq \frac{1}{|S_{N}|}\sum_{\tau\in S_{N}} J\left(\pmb{\gamma^{\tau(1)\star}}, \ldots, \pmb{\gamma^{\tau(N)\star}}\right)\\
&=J\left(\pmb{\gamma^{1\star:N\star}}\right),
\end{align*}
where the last equality follows from part (i). This implies that for any arbitrary Markovian policy $\pmb{\gamma^{1\star:N\star}}\in\Gamma_{N}^{\sf M}$, there exists a policy $\bar{\pmb{\gamma}}\in \Gamma_{N}^{\sf SYM,M}$ that performs at least as well as $\pmb{\gamma^{1\star:N\star}}\in\Gamma_{N}^{\sf M}$. This completes the proof of part (ii).

\end{itemize}

\section{Proof of Theorem \ref{the:3}}\label{App:the3}
\begin{itemize}[wide]
    \item [Part (i):]
 In the following, we first show that 
$\pmb{{\nu}^{1:N}}\mapsto J_{N}(\pmb{{\nu}^{1:N}})$ 
is lower-semi continuous. Let $\pmb{{\nu}_k^i}$ converge to $\pmb{{\nu}^i}$ as $k\to \infty$ for all $i=1, \ldots, N$ and let
\begin{align*}
    \Theta_{M}({X^{1:N}},{U^{1:N}}):= \min\left\{M, c(X^{1:N}, U^{1:N})\right\}.
\end{align*}
We have
\begin{align}
    &\liminf_{k\to \infty} \mathbb{E}^{\pmb{\nu_{k}^{1:N}}}\left[\int_{0}^{T} c(X_{t}^{1:N}, U_{t}^{1:N})dt \right]\label{eq:pfT3-0}\\
    &\geq \lim_{M\to \infty}\liminf_{k\to \infty}\int_{0}^{T}\int \Theta_{M}({X^{1:N}_t},{U^{1:N}_t}) \prod_{i=1}^{N} \pmb{\mu^{i}}(d\pmb{X^{i}};\pmb{\nu^{i}_{k}}){\nu^{i}_{k,t}}(X_{t}^{i})(dU^i_t)dt\label{eq:pfT3-1}\\
    &=\lim_{M\to \infty}\int_{0}^{T}\int \Theta_{M}({X^{1:N}_t},{U^{1:N}_t}) \prod_{i=1}^{N} \pmb{\mu^{i}}(d\pmb{X^{i}};\pmb{\nu^{i}}){\nu^{i}_{t}}(X_{t}^{i})(dU^i_t)dt\label{eq:pfT3-3}\\
&=\mathbb{E}^{\pmb{\nu^{1:N}}}\left[\int_{0}^{T} c(X_{t}^{1:N},U_{t}^{1:N}) dt \right],\label{eq:pfT3-4}
\end{align}
where 
\eqref{eq:pfT3-1} is true since we have $c(X_{t}^{1:N}, U_{t}^{1:N})\geq\Theta_{M}({X^{1:N}_t},{U^{1:N}_t})$ for every $M\in \mathbb{R}_{+}$. Equality \eqref{eq:pfT3-3} follows from the generalized convergence theorem \cite[Theorem 3.5]{serfozo1982convergence} and
the fact that $\pmb{\nu^i_{k}}$ converges weakly to $\pmb{\nu^i}$ as $k\to \infty$ for all $i\in \mathcal{N}$ since $c$ is continuous by Assumption \ref{assum:5-4} and $\Theta_{M}$ is bounded. Equality \eqref{eq:pfT3-4} follows from the monotone convergence theorem.

Following \eqref{eq:pfT3-0}--\eqref{eq:pfT3-4}, $J_{N}(\pmb{\nu^{1:N}})$ is lower-semi continuous in $\pmb{\nu^{1:N}}$. 
 This together with the fact that $\mathcal{M}^i$ is compact following Lemma \ref{lem:Borkar} implies that there exists an optimal solution $\pmb{\nu^{1\star:N\star}}$ for \PN\ among randomized Markov policies. Along the same reasoning as that in the proof of Theorem \ref{the:1}, we can show that restricting the search for global optimality to randomized Markov policies is without loss among all path-dependent independently randomized policies, and this completes the proof of part (i). 
 
\item [Part (ii):] Following Theorem \ref{the:2}, we can assume that a globally optimal solution is Markovian and symmetric. Hence, we can assume that a randomized globally optimal solution is Markovian and symmetric since the randomization does not improve the team's performance. Using Lemma \ref{lem:Borkar}, we can conclude that the set of randomized symmetric Markovian policies in $\mathcal{M}^i$ is compact (symmetry will be preserved in the limit as we study each $\mathcal{M}^i$ separately). Following the reasoning of part (i), we conclude that $(\pmb{\nu}, \ldots, \pmb{\nu})\mapsto J_{N}(\pmb{{\nu}}, \ldots, \pmb{{\nu}})$ is lower-semi continuous. This completes the proof using the Weierstrass theorem. 
\end{itemize}

\section{Proof of Theorem \ref{the:4}}\label{App:the4}
\begin{itemize}[wide]
    \item [Part (i):]
Consider a sequence of globally optimal policies $\{\pmb{\nu_{N}^{\star}}\}_{N}\subseteq \mathcal{M}^i$ that is Markovian and symmetric. Since $\mathcal{M}^i$ is compact, there exists a subsequence $\{\pmb{\nu_{n}^{\star}}\}_{n}\subseteq \mathcal{M}^i$ that converges to a limiting policy $\pmb{\nu^{\star}}\in \mathcal{M}^i$. For every $t\in \mathbb{T}$, define the following empirical measures induced by $\pmb{\nu_{n}^{\star}}$ and $\pmb{\nu^{\star}}$ as
\begin{align*}
    Q_{n}^t(\cdot) &= \frac{1}{n}\sum_{i=1}^{n} \delta_{\{X^{i,n}_t, U^{i, n}_t\}}(\cdot),\quad\overline{Q}_{n}^t(\cdot) = \frac{1}{n}\sum_{i=1}^{n} \delta_{\{X^{i,\infty}_t, U^{i, \infty}_t\}}(\cdot),
\end{align*}
where $U^{i, n}_t,X^{i, n}_t$ and $U^{i, \infty}_t,X^{i, \infty}_t$ are induced by $\pmb{\nu_{n}^{\star}}$ and $\pmb{\nu^{\star}}$, respectively. Due to the symmetry and Assumption \ref{assump1} and Assumption \ref{assump:convex}, we have that $X^{i,n}_t, U^{i, n}_t$ and $X^{i,\infty}_t, U^{i, \infty}_t$ are i.i.d. among DMs. By Skorohod’s representation theorem, we have $X^{i, n}_t$ and $U^{i, n}_t$ converge almost surely to $X^{i, \infty}_t$ and $U^{i, \infty}_t$, respectively for every $t\in \mathbb{T}$. Also, using the strong law of large numbers, Markov inequality, and dominated convergence theorem, we can show that for a continuous bounded function $g$ ($g\in \mathcal{C}_{b}(\mathbb{X} \times \mathbb{U})$)
\begin{align*}
 \lim_{n\to \infty}\left|\frac{1}{n}\sum_{i=1}^{n} g(X^{i,n}_t, U^{i, n}_t)- \mathbb{E}[g(X^{i,\infty}_t, U^{i, \infty}_t)]\right|=0\quad \mathbb{P}\text{-a.s.} 
\end{align*}
Since the above holds for any function of $\mathcal{C}_{b}(\mathbb{X} \times \mathbb{U})$, considering a countable family of measure-determining functions subset of $\mathcal{C}_{b}(\mathbb{X} \times \mathbb{U})$ conclude that $\{{Q}_{n}^t\}_{n}$ converges (a.s.) weakly to $Q_{\infty}^t:=\mathcal{L}(X^{i, \infty}_t, U^{i,\infty}_t)$ for every $t\in \mathbb{T}$ along a subsequence denoted again by $n$. Similarly, $\{\overline{Q}_{n}^t\}_{n}$ converges (a.s.) weakly to $Q_{\infty}^t:=\mathcal{L}(X^{i, \infty}_t, U^{i,\infty}_t)$ for every $t\in \mathbb{T}$. Since the path space is continuous and the action space is compact, we can conclude that the empirical measure of the path space also converges weakly. Hence,
\begin{align}
   &\limsup_{N\to \infty} \inf_{\pmb{{\nu}^{1:N}}\in \prod_{i=1}^{N}\mathcal{M}^i} J_{N}(\pmb{{\nu}^{1:N}}) \nonumber\\
   &= \limsup_{N\to \infty} J_{N}(\pmb{\nu_{N}^{\star}}, \ldots, \pmb{\nu_{N}^{\star}}) \label{eq:4-01}
   \\
   &\geq J_{\infty}(\underline{\pmb{\nu^{\star}}})\label{eq:4-1}
\end{align}
where $\underline{\pmb{{\nu}^{\star}}}:=(\pmb{{\nu}^{\star}}, \pmb{{\nu}^{\star}}, \ldots)$. Equality in \eqref{eq:4-01} follows from the fact that $(\pmb{\nu_{N}^{\star}}, \ldots, \pmb{\nu_{N}^{\star}})$ is globally optimal for \PN\ for every finite $N$. Now, we prove the inequality in \eqref{eq:4-1}. We have
\begin{align}
    &\limsup_{N\to \infty} J_{N}(\pmb{\nu_{N}^{\star}}, \ldots, \pmb{\nu_{N}^{\star}})\nonumber\\
&\geq\lim_{M\to \infty}\lim_{n\to \infty}\int_{0}^{T}\int 
g(t,Q_{n}^t,M) \prod_{i=1}^{\infty} \pmb{\mu}(d\pmb{X^{i}};\pmb{\nu_{n}^{\star}}){\nu_{n,t}^{\star}}(X_{t}^{i})(d{U^{i}_t})dt \nonumber\\
&=\lim_{M\to \infty} \int_{0}^{T} \int g(t,Q_{\infty}^t,M) \prod_{i=1}^{\infty} \pmb{\mu}(d\pmb{X^{i}};\pmb{\nu^{\star}}){\nu_{t}^{\star}}(X_{t}^{i})(d{U^{i}_t})dt\label{eq:pfT4-3}\\
&=J_{\infty}(\underline{\pmb{{\nu}^{\star}}})\label{eq:pfT4-4},
\end{align}
where 
\begin{align*}
    &g(t,Q_{l}^t,M):=Q_{l}^{t}(dX, dU) \min\left\{M, \hat{c}(X, U, \int Q_{l}^{t}(dX, \mathbb{U}), \int Q_{l}^{t}(\mathbb{X}, dU))\right\}
\end{align*}
for $l=n$ or $l=\infty$. Equality \eqref{eq:pfT4-3} follows from the generalized dominated convergence theorem  since ${Q}_{n}^t$ and $\overline{Q}_{n}^t$ converge weakly to $Q_{\infty}^t$ for every $t\in \mathbb{T}$, and $\pmb{\nu_{n}^{\star}}$  converges to  $\pmb{\nu^{\star}}$. Equality \eqref{eq:pfT4-4} follows from the monotone convergence theorem and the same line of reasoning as that used for the expected cost under $\pmb{{\nu}^{\star}}$.
This completes the proof of the inequality in \eqref{eq:4-1}.
On the other hand, we have 
\begin{align}
   J_{\infty}(\underline{\pmb{{\nu}^{\star}}}) &\leq \limsup_{N\to \infty} \inf_{\pmb{{\nu}^{1:N}}\in \prod_{i=1}^{N}\mathcal{M}^i} J_{N}(\pmb{{\nu}^{1:N}}) \\&\leq  \inf_{\pmb{{\nu}^{1:\infty}}\in \prod_{i=1}^{\infty}\mathcal{M}^i} \limsup_{N\to \infty}  J_{N}(\pmb{{\nu}^{1:N}})\label{eq:4-5}
   \\&= \inf_{\pmb{{\nu}^{1:\infty}}\in \prod_{i=1}^{\infty}\mathcal{M}^i} J_{\infty}(\pmb{{\nu}^{1:\infty}})\label{eq:4-6},
\end{align}
where the second inequality in \eqref{eq:4-5} follows from exchanging infimum and limsup. Hence, the policy $\underline{\pmb{{\nu}^{\star}}}$ is globally optimal for \PIN, and completes the proof of part (i).

\item [Part (ii):] 
This part follows from Theorem \ref{the:3} and part (i). By Theorem \ref{the:3}, for every finite $N$, there exists a globally optimal policy that is symmetric and Markovian. Hence, there exists a sequence of globally optimal policies $\{\pmb{\nu_{N}^{\star}}\}_{N}\subseteq \mathcal{M}^i$ that is symmetric. By part (i),  $\{\pmb{\nu_{N}^{\star}}\}_{N}$ converges through a subsequence to a globally optimal policy $\pmb{{\nu}^{\star}}\in \mathcal{M}^i$ for \PIN. 

\end{itemize}

\section{Proof of  Lemma \ref{lem:2}}\label{app:lem:2}
 \begin{itemize}
        \item [Part(i):] Since marginals on $[0,T]$ is fixed and $\mathbb{U}$ is compact, $\Gamma^i$ is tight for all $i\in \mathcal{N}$. Using Kushner’s weak convergence approach \cite{kushner2012weak} or Borkar's approach in \cite{borkar1989topology}, we can show that $\Gamma^i$ is closed. Hence, $\Gamma^i$ is compact. Now, we show that $\LEXN$ and $\LEX$ are compact using the compactness of $\Gamma^i$. 

         Since $\Gamma^i$ is compact, $\LEXN$ and $\LEX$ are both tight under the product topology. We only need to show that $\LEXN$ and $\LEX$ are closed.  We only show that $\LEXN$ is closed under the weak-convergence topology. Closedness of $\LEX$ follows from a similar argument. 

For any permutation $\tau \in S_{N}$, we define a randomized policy $P_{\pi}^{\tau} \in {L}^{N}$ as a permutation, $\tau$, of arguments of a randomized policy $P_{\pi} \in {L}^{N}$, i.e., for $A^{i} \in \mathcal{B}(\Gamma^{i})$
\begin{flalign}\label{eq:s4.1}
&P_{\pi}^{\tau}(\pmb{\gamma^{1}}\in A^{1},\dots, \pmb{\gamma^{2}}\in A^{N}):=P_{\pi}(\pmb{\gamma}^{\tau(1)}\in A^{1},\dots, \pmb{\gamma}^{\tau(N)}\in A^{N}).
\end{flalign}

Suppose that a sequence of randomized policies $\{P_{\pi}^{ n}\}_{n}$ converges weakly to a randomized policies $P_{\pi}^{\infty}$ as $n \to \infty$. Also, suppose that $\{P_{\pi}^{\tau, n}\}_{n}$ converges weakly to $P_{\pi}^{\tau, \infty}$ as $n \to \infty$, where for $A^{i} \in \mathcal{B}(\Gamma^{i})$ 
\begin{flalign*}
&P_{\pi}^{\tau, n}(\pmb{\gamma^{1}}\in A^{1},\cdots, \pmb{\gamma^{N}}\in A^{N})=P_{\pi}^{n}(\pmb{\gamma}^{\tau(1)}\in A^{1},\cdots, \pmb{\gamma}^{\tau(N)}\in A^{N})\quad \forall \tau \in S_{N}.
\end{flalign*}
{We will show that the limit is also $N$-exchangeable:} Let $\mathcal{T}$ be a countable measure-determining subset of the set of all real-valued continuous functions on $\prod_{i=1}^{N}\Gamma^{i}$. For a function $f\in \mathcal{T}$, we have
\begin{align}
    \int f(\pmb{\gamma^{1:N}})P_{\pi}^{\tau, \infty}(d\pmb{\gamma^{1:N}})
    &= \lim_{n\to \infty}\int f(\pmb{\gamma^{1:N}})P_{\pi}^{n}(d\pmb{\gamma^{1:N}})
    =\int f(\pmb{\gamma^{1:N}})P_{\pi}^{ \infty}(d\pmb{\gamma^{1:N}})\label{eq:30-the1},
\end{align}
where \eqref{eq:30-the1} follows from the weak convergence of $\{P_{\pi}^{\tau, n}\}_{n}$ and $\{P_{\pi}^{n}\}_{n}$ and the fact that $P_{\pi}^{n}$ is $N$-exchangeability for every $n$. Since $\mathcal{T}$ is countable, for all $f\in \mathcal{T}$, we get \eqref{eq:30-the1}.
Since $\mathcal{T}$ is measure-determining, we get that $P_{\pi}^{\tau, \infty}$ and $P_{\pi}^{\infty}$ are the same, and hence, $\LEXN$ is compact as it is tight. An argument similar to the above establishes that $\LEX$ is compact. Convexity of $\LEXN$ and $\LEX$ follow from the fact that the convex combination of exchangeable probability measures remains an exchangeable probability measure.

\item [Part (ii):] Since $\{P_{\pi, k}\}_{k}$ as a subset of $\LEXN$ (or $\LEX$) converges weakly to $P_{\pi, \infty}$ as $k\to \infty$, the induced $\{\pmb{\gamma^{i}_k}(\omega)\}_{k}$ converges weakly to $\pmb{\gamma^{i}_k}(\omega)$ for all $i\in \mathcal{N}$  as $k\to \infty$. 

Consider a sequence of $\{\pmb{\gamma^{i}_k}(\omega)\}_{k}\subseteq \Gamma^{i}$ that converges to $\pmb{\gamma^{i}_k}(\omega)$ weakly for all $i\in \mathcal{N}$.  Let $\pmb{X^i_{k}}$ be the solution of the diffusion equation in \eqref{eq:st-dy-N} under $\pmb{\gamma^{i}_k}(\omega)$, and hence, for any twice continuously differentiable function $f:\mathbb{X}\to \mathbb{R}$, it satisfies the following Martingale solution
    \begin{align}\label{mart-1}
        \mathbb{E}_{x_0}\left[f(X_{t,k}^{i})-f(X_{0,k}^i)-\int_{0}^{t}\int_{\mathbb{U}}\mathcal{A}^{U^i}(X_{s,k}^{i})\pmb{\gamma^{i}_k}(d{U^{i}},ds)\right] = 0,
    \end{align}
    where $\mathbb{E}_{x_0}$ denotes the conditional expectation on $X_{0}^{i}=x_0$.
    By a tightness argument (via applying \cite[Theorem 1.4.2]{kushner2012weak} to \cite[Theorems 2.3.1 and 2.3.3]{kushner2012weak}), $\pmb{X^i_{k}}$ admits a subsequence (denoted by $n$) that converges weakly $\pmb{X^i}$. By Skorohod’s representation theorem, there exists a probability space in which $\{\pmb{\gamma^{i}_n}(\omega)\}_{n}$ and $\{\pmb{X^i_{n}}\}$ converge weakly to $\pmb{\gamma^{i}}(\omega)$ and $\pmb{X^i}$ almost surely (in $\omega$), respectively. Following \eqref{mart-1}, we have
\begin{align}
\mathbb{E}_{x_0}\bigg[&f(X_{t,n}^{i})-f(X_{0,n}^i)\nonumber\\
&-\int_{0}^{t}  \left({\sf Tr} \left(a_{s}(X_{s,n}^i)\nabla^2 f(X_{s,n}^{i})\right)+\int_{\mathbb{U}}b_{s}(X_{s,n}^{i},u^i)\cdot \nabla f(X_{s,n}^{i})\pmb{\gamma^{i}_n}(d{U^{i}}|s)\right)\bigg] = 0,\label{mart-2}
\end{align}
where $a_{s}(X_{s,n}^i)=\frac{1}{2}\sigma_{s}\sigma^{T}_{s}$
By the
generalized dominated convergence \cite[Theorem 3.5]{serfozo1982convergence}, we have
\begin{align}
&\lim_{n\to \infty}  \left({\sf Tr} \left(a_{s}(X_{s,n}^i)\nabla^2 f(X_{s,n}^{i})\right)+\int_{\mathbb{U}}b_{s}(X_{s,n}^{i},u^i)\cdot \nabla f(X_{s,n}^{i})\pmb{\gamma^{i}_n}(d{U^{i}}|s)\right)\nonumber\\
&= \left({\sf Tr} \left(a_{s}(X_{s}^i)\nabla^2 f(X_{s}^{i})\right)+\int_{\mathbb{U}}b_{s}(X_{s}^{i},u^i)\cdot \nabla f(X_{s}^{i})\pmb{\gamma^{i}}(d{U^{i}}|s)\right).
\end{align}
Hence, the limit $\pmb{X^i}$ solves the Martingale equation, which is induced by $\pmb{\gamma^{i}}(\omega)$. This implies that $\pmb{X^i_{n}}$ induced by $\pmb{\gamma^{i}_n}(\omega)$ converges weakly to $\pmb{X^i}$ induced by $\pmb{\gamma^{i}}(\omega)$, and hence,  $\mu_{n}^i$, the distribution of $\pmb{X^i_{n}}$, converges to $\mu^i$, the distribution of $\pmb{X^i}$ for all $i\in \mathcal{N}$. This completes the proof of part (ii).  
\end{itemize}  

\section{Proof of Theorem \ref{the:5}}\label{app:the:5}
 The proof proceeds in three steps. In step 1, we show that there exists an exchangeable randomized globally optimal policy ${P_{\pi}}$ for \PN\ among all exchangeable randomized policies of $\LEXN$. In step 2, we show that the expected cost remains the same under any permutation of randomized policies in $L^{N}$. Finally, we show that without loss we can restrict the search among all randomized policies in $L^{N}$ to exchangeable ones $\LEXN$.
\begin{itemize}
    \item [{\it Step 1.}] We first show that the mapping $J_{N}^{\pi}:\LEXN \to \mathbb{R}$ is continuous. Suppose that $\{P_{\pi, n}\}_{n}$ as a subset of $\LEXN$ (or $\LEX$) converge to $P_{\pi, \infty}$. By Lemma \ref{lem:2}(ii), $\mu_{n}^i$, the distribution of $\pmb{X^i_{n}}$, converges to $\mu^i$, the distribution of $\pmb{X^i}$ for all $i\in \mathcal{N}$. By Skorohod’s representation theorem, there exists a probability space in which $\{\pmb{\gamma^{i}_n}(\omega)\}_{n}$ and probability measure on $\{\pmb{X^i_{n}}(\omega)\}$ converge weakly to $\pmb{\gamma^{i}}(\omega)$ and the probability measure on $\pmb{X^i}(\omega)$ almost surely (in $\omega$), respectively. The generalized convergence theorem implies that $J_{N}^{\pi}:\LEXN \to \mathbb{R}$ is continuous. Since $\LEXN$ is compact by Lemma Lemma \ref{lem:2}(i), this implies that there exists an exchangeable randomized globally optimal policy $P_{\pi}^{N}$ for \PN\ among all randomized policies $\LEXN$.
\item [{\it Step 2.}]  For any $P_{\pi}^{N}\in L^{N}$ and any permutation of $\tau\in S_{N}$, we have
\begin{flalign}
&\int \left(c_{N}(\pmb{X}^{1:N},\pmb{\gamma^{1:N}})\prod_{i=1}^{N}\mu^i(d\pmb{X^{i}};\pmb{\gamma^{i}})\right) P_{\pi}^{N, \tau}(d\pmb{\gamma^{1:N}})\nonumber\\
&=\int \left(\int \int_{0}^{T} c(X_{t}^{1:N},U_{t}^{1:N})\prod_{i=1}^{N}\pmb{\gamma^{i}}(d{U^{i}},dt)\mu^i(d\pmb{X^{i}};\pmb{\gamma^{i}})\right) P_{\pi}^{N}(d\pmb{\gamma}^{\tau(1)},\ldots, d\pmb{\gamma}^{\tau(N)})\label{eq:4.4.4.5}\\
&=\int \left(\int \int_{0}^{T} c(X_{t}^{\tau(1:N)},U_{t}^{\tau(1:N)})\prod_{i=1}^{N}\pmb{\gamma}^{\tau(i)}(dU^{\tau(i)},dt)\mu^{\tau(i)}(d\pmb{X}^{\tau(i)};\pmb{\gamma}^{\tau(i)})\right)\nonumber\\
&\qquad\times P_{\pi}^{N}(d\pmb{\gamma}^{1},\ldots, d\pmb{\gamma}^{N})\label{eq:4.4.4.4}\\
&=\int \left(\int \int_{0}^{T} c(X_{t}^{1:N},U_{t}^{1:N})\prod_{i=1}^{N}\pmb{\gamma^{i}}(d{U^{i}},dt)\mu^i(d\pmb{X^{i}};\pmb{\gamma^{i}})\right) P_{\pi}^{N}(d\pmb{\gamma}^{1},\ldots, d\pmb{\gamma}^{N})\label{eq:4.4.4.6}\\
&=\int \left(c_{N}(\pmb{X}^{1:N},\pmb{\gamma^{1:N}})\prod_{i=1}^{N}\mu^i(d\pmb{X^{i}};\pmb{\gamma^{i}})\right) P_{\pi}^{N}(d\pmb{\gamma^{1:N}})\nonumber,
\end{flalign}
 where \eqref{eq:4.4.4.5} follows from the definition of permutation of a randomized policy in \eqref{eq:s4.1}, and \eqref{eq:4.4.4.4} follows from relabeling $\pmb{U}^{\tau(i)}, \pmb{X}^{\sigma(i)},\pmb{\gamma}^{\tau(i)}$ with $\pmb{U^{i}}, \pmb{X^{i}},\pmb{\gamma^{i}}$ for all $i=1,\dots,N$. Equality \eqref{eq:4.4.4.6} follows from exchangeability of cost in Assumption \ref{assump1} and symmetry in dynamics and initial states of the model. 

 \item [{\it Step 3.}]
 Consider a randomized policy $P^{N\star}_{\pi} \in {L}^{N}$. 
We can construct a policy $\tilde{P}^{N\star}_{\pi}$ as a convex combination of all possible permutations of $P^{N\star}_{\pi}$ by averaging them. Since ${L}^{N}$ is convex, we have $\tilde{P}_{\pi} \in {L}^{N}$. Also, we have that $\tilde{P}^{N\star}_{\pi} \in \LEXN$ since for any permutation $\tau \in S_{N}$
\begin{flalign}
\tilde{P}^{N\star}_{\pi}(d\pmb{\gamma^{1}},\dots,d\pmb{\gamma^{N}})&:=\frac{1}{|S_{N}|}\sum_{\tau\in S_{N}}P^{N\star,\tau}_{\pi}(d\pmb{\gamma^{1}},\dots,d\pmb{\gamma^{N}})\nonumber\\
&\:=\tilde{P}^{N\star, \tau}_{\pi}(d\pmb{\gamma^{1}},\dots,d\pmb{\gamma^{N}})\nonumber,
\end{flalign}
where $|S_{N}|$ denotes the cardinality of the set $S_{N}$, and the second equality follows from the fact that the sum is over all permutations $\tau$ by taking the average of them. We have
 \begin{flalign*}
 &\int \left(c_{N}(\pmb{X}^{1:N},\pmb{\gamma^{1:N}})\prod_{i=1}^{N}\mu^i(d\pmb{X^{i}};\pmb{\gamma^{i}})\right) \tilde{P}_{\pi}^{N\star}(d\pmb{\gamma^{1:N}})\nonumber\\
&= \frac{1}{|S_{N}|}\sum_{\tau\in S_{N}}\int \left(c_{N}(\pmb{X}^{1:N},\pmb{\gamma^{1:N}})\prod_{i=1}^{N}\mu^i(d\pmb{X^{i}};\pmb{\gamma^{i}})\right) P^{N\star,\tau}_{\pi}(d\pmb{\gamma^{1:N}})\\
&=\frac{1}{|S_{N}|}\sum_{\tau\in S_{N}}\int \left(c_{N}(\pmb{X}^{1:N},\pmb{\gamma^{1:N}})\prod_{i=1}^{N}\mu^i(d\pmb{X^{i}};\pmb{\gamma^{i}})\right) P^{N\star}_{\pi}(d\pmb{\gamma^{1:N}})\\
&=\int \left(c_{N}(\pmb{X}^{1:N},\pmb{\gamma^{1:N}})\prod_{i=1}^{N}\mu^i(d\pmb{X^{i}};\pmb{\gamma^{i}})\right) P^{N\star}_{\pi}(d\pmb{\gamma^{1:N}}),
 \end{flalign*}
 where the second equality is true since  $P_{\pi}^{N\star} \mapsto J_{N}^{\pi}(P_{\pi}^{N\star})$ is linear and the third equality follows from step 2.  Hence, for any arbitrary policy $P_{\pi}^{N\star}\in L^{N}$, there exists a policy $\tilde{P}_{\pi}^{N\star}\in \LEXN$ that performs at least as well as $P_{\pi}^{N\star}$. This together with step 1 completes the proof. 
\end{itemize}

\section{Proof of Lemma \ref{lem:3}}\label{app:lem:3}
To prove Lemma \ref{lem:3}, we use two following results by Diaconis and Friedman \cite[Theorem 13]{diaconis1980finite} and Aldous \cite[Proposition 7.20]{aldous2006ecole}, which we recall in the following.
\begin{theorem}\cite[Theorem 13]{diaconis1980finite}\label{the:dia}
Let $Y=(Y_{1},\dots,Y_{n})$ be an $n$-exchangeable and $Z=(Z_{1},Z_{2},\dots)$ be an infinitely exchangeable sequence of random variables with $\mathcal{L}(Z_{1},\dots,Z_{k})=\mathcal{L}(Y_{I_1},\dots,Y_{I_k})$ for all $k\geq 1$, where the indices $(I_{1}, I_{2}, \dots)$ are i.i.d. random variables with the uniform distribution on the set $\{1,\dots,n\}$. Then,  for all $m=1,\dots, n$,
\begin{flalign}
\left|\left|\mathcal{L}(Y_{1},\dots,Y_{m})-\mathcal{L}(Z_{1},\dots,Z_{m})\right|\right|_{\sf TV}\leq \frac{m(m-1)}{2n},
\end{flalign}
where $\mathcal{L}(\cdot)$ denotes the law of random variables and $||\cdot||_{\sf TV}$ is the total variation norm.
\end{theorem}

\begin{theorem}\cite[Proposition 7.20]{aldous2006ecole}\label{the:ald}
Let $X:=(X_{1},X_{2},\dots)$ be an infinitely exchangeable sequence of random variables taking values in a Polish space $\mathbb{X}$. Suppose that $X$ is directed by a random measure $\alpha$, i.e., $\alpha$ is a $\mathcal{P}(\mathbb{X})$-valued random variable and ${\sf Pr}(X\in A)=\int_{\mathcal{P}(\mathbb{X})} \prod_{i=1}^{\infty}\xi(A^{i})\:{\sf Pr}(\alpha\in d\xi)$ where $A^{i}\in \mathcal{B}(\mathbb{X})$ and $(A=A^{1}\times A^{2}\times \dots)$ (see \cite[Definition 2.6]{aldous2006ecole}). Suppose further that either for each $n$
\begin{itemize}
\item[(1)] $X^{(n)}=(X_{1}^{(n)},X_{2}^{(n)},\dots)$  is infinitely exchangeable directed by $\alpha_{n}$, or
\item[(2)]  $X^{(n)}=(X_{1}^{(n)},\dots, X_{n}^{(n)})$  is $n$-exchangeable with empirical measure $\alpha_{n}$.
\end{itemize}
Then, $X^{(n)}$ converges in distribution\footnote{By convergence in distribution to an infinitely exchangeable sequence, we mean the following: $X^{(n)} \xrightarrow[n \to \infty]{{\sf d}} X$ if and only if $(X_{1}^{(n)},\dots, X_{m}^{(n)})\xrightarrow[n \to \infty]{{\sf d}}(X_{1},\dots, X_{m})$ for each $m\geq 1$; see \cite[page 55]{aldous2006ecole}.} to $X$ as $n\to \infty$ (i.e., $X^{(n)}\xrightarrow[n \to \infty]{{\sf d}}X$)  if and only if $\alpha_{n}\xrightarrow[n \to \infty]{{\sf d}}\alpha$.
\end{theorem}

We now prove Lemma \ref{lem:3} using Theorem \ref{the:dia} and Theorem \ref{the:ald}.

\begin{proof}[Proof of Lemma \ref{lem:3}]
By Theorem \ref{the:5}, for every finite $N\geq 1$, there exists an exchangeable randomized globally optimal policy ${P_{\pi}^{N\star}}\in \LEXN$ for \PN\ among all randomized policies $L^{N}$. The proof of  Lemma \ref{lem:3} proceeds in two steps. In the first step, for every $N\geq 1$ and  $P_{\pi}^{N\star} \in  \LEXN$, we construct an infinitely-exchangeable randomized policy $P^{N\star,\infty}_{\pi} \in \LEX$ using Theorem \ref{the:dia}, by assigning i.i.d. random indices with uniform distribution on the set $\{1,\dots,N\}$.  Using Theorem \ref{the:dia} and compactness of $\LEX$, we show that sequences $\{P^{N\star,\infty}_{\pi}\}_N$ and $\{P_{\pi}^{N\star}\}_N$ converge to the same limit $P^{\star,\infty}_{\pi}\in \LEX$ through a subsequence. Then, using Theorem \ref{the:ald}, we establish that the induced states and actions under $P_{\pi}^{N\star}$ as well as their empirical measures converge weakly to a limit. In the second step, we show that the expected cost induced by $P_{\pi}^{N\star}$  converges to a limit induced by the limiting policy $P^{\star,\infty}_{\pi}\in \LEX$.

\begin{itemize}
    \item [{\it Step 1.}] By Theorem \ref{the:5}, for every finite $N\geq 1$, there exists an exchangeable randomized globally optimal policy ${P_{\pi}^{N\star}}\in \LEXN$ for \PN. Consider random i.i.d.  variables $\{I_{i}\}_{i\geq 1}$ with uniform distribution on $\{1,\ldots,N\}$. For any $N\geq 1$ and  $P_{\pi}^{N\star} \in  \LEXN$, define an infinitely-exchangeable randomized policy $P^{N\star,\infty}_{\pi} \in \LEX$ such that for every $N\geq 1$ and for all $A^{i} \in \mathcal{B}(\Gamma^{i})$
\begin{flalign*}
&P^{N\star,\infty}_{\pi,N}(\pmb{\gamma^{1}}\in A^{1},\ldots, \pmb{\gamma^{m}}\in A^{m}):=P_{\pi}^{N\star}(\pmb{\gamma^{I_{1}}}\in A^{1},\ldots, \pmb{\gamma^{I_{m}}}\in A^{m})\quad m\leq N,
\end{flalign*}
where $P^{N\star,\infty}_{\pi,N}\in \LEX|_N$ is the restriction of $P^{N\star,\infty}_{\pi}\in  \LEX$ to its first $N$ components.
The fact that $\{I_{i}\}_{i\geq 1}$ are i.i.d, yields that for each $N\geq 1$, $P^{N\star,\infty}_{\pi}\in \LEX$, and hence, $\{P^{N\star,\infty}_{\pi}\} \subseteq \LEX$. 
Let $\pmb{\gamma^{1\star:N\star}_{N}}:=(\pmb{\gamma^{1\star}_{N}}, \ldots, \pmb{\gamma^{N\star}_{N}})$ and $\pmb{\gamma^{1\star:N\star}_{N,\infty}}:=(\pmb{\gamma^{1\star}_{N,\infty}}, \ldots, \pmb{\gamma^{N\star}_{N,\infty}})$ be the policies induced by  randomized policies $P^{N\star,\infty}_{\pi, N}$ 
 and $P^{N\star}_{\pi}$, respectively. Theorem \ref{the:dia} yields that for any finite $m\geq 1$
\begin{flalign}
&\lim_{N\to \infty}\left\|\mathcal{L}(\pmb{\gamma^{1\star}_{N}},\dots,\pmb{\gamma^{m\star}_{N}})-\mathcal{L}(\pmb{\gamma^{1\star}_{N,\infty}},\dots,\pmb{\gamma^{m\star}_{N,\infty}})\right\|_{\sf TV}=0.\label{eq:poA.7}
\end{flalign}

Since $\LEX$ is compact, $\{(\pmb{\gamma^{1\star}_{N,\infty}},\pmb{\gamma^{2\star}_{N,\infty}}, \ldots)\}_{N}$ admits a subsequence (denoted by the index $n$) such that $\{(\pmb{\gamma^{1\star}_{n,\infty}},\pmb{\gamma^{2\star}_{n,\infty}}, \ldots)\}_n$ converges to an infinitely-exchangeable sequence of random variables $(\pmb{\gamma^{1\star}_{\infty}},\pmb{\gamma^{1\star}_{\infty}}, \ldots)$ in distribution as $n\to \infty$. By \eqref{eq:poA.7}, this implies that $\{(\pmb{\gamma^{1\star}_{n}},\pmb{\gamma^{1\star}_{n}}, \ldots)\}_n$ also converges to $(\pmb{\gamma^{1\star}_{\infty}},\pmb{\gamma^{1\star}_{\infty}}, \ldots)$ in distribution as $n\to \infty$ (i.e., it converges in distribution for every finite marginal). By Lemma \ref{lem:2}, this implies that the induced probability measures on states converge weakly, and hence, the sequence of states $\{(\pmb{X}^{1\star}_{n},\pmb{X}^{2\star}_{n}, \ldots)\}_{n}$ induced under  $\{(\pmb{\gamma^{1\star}_{n}},\pmb{\gamma^{1\star}_{n}}, \ldots)\}_n$ converges in distribution to the states $(\pmb{X}^{1\star}_{\infty},\pmb{X}^{2\star}_{\infty}, \ldots)$ induced by $(\pmb{\gamma^{1\star}_{\infty}},\pmb{\gamma^{1\star}_{\infty}}, \ldots)$ as $n\to \infty$. We have also the sequence of the actions $\{(\pmb{U}^{1\star}_{n},\pmb{U}^{2\star}_{n}, \ldots)\}_{n}$ induced under $\{(\pmb{\gamma^{1\star}_{n}},\pmb{\gamma^{1\star}_{n}}, \ldots)\}_n$ are $n$-exchangeable and converges in distribution to the infinitely-exchangeable actions $(\pmb{U}^{1\star}_{\infty},\pmb{U}^{2\star}_{\infty}, \ldots)$ induced by $(\pmb{\gamma^{1\star}_{\infty}},\pmb{\gamma^{1\star}_{\infty}}, \ldots)$ as $n\to \infty$.

Thanks to Theorem \ref{the:ald}, we conclude that the sequence of empirical measures on the induced states and actions
\begin{align*}
     Q_{n}^t(\cdot) &= \frac{1}{n}\sum_{i=1}^{n} \delta_{\{{X}^{1\star}_{n,t},{U}^{1\star}_{n,t}\}}(\cdot)
\end{align*}
converges weakly to $Q_{\infty}^t:=\mathcal{L}({X}^{i\star}_{\infty,t},{U}^{i\star}_{\infty,t})$ for every $t\in \mathbb{T}$. Since the path space is continuous and the action space is compact, we can conclude that the empirical measure of the path space also converges weakly. By Skorohod’s representation theorem, there exists a probability space in which, all the convergence mentioned above is in distribution but almost surely in $\omega$. Hence, the marginals of empirical measures of on the induced states and actions  $${\sf E}_{n,t}^{\pmb{X^\star}}:=\frac{1}{n}\sum_{i=1}^{n} \delta_{{X}^{i\star}_{n,t}}, \quad {\sf E}_{n,t}^{\pmb{U^\star}}:=\frac{1}{n}\sum_{i=1}^{n} \delta_{{U}^{i\star}_{n,t}}$$  converge weakly to ${\sf E}_{\infty,t}^{\pmb{X^\star}}:=\mathcal{L}({X}^{i\star}_{\infty,t})$ and ${\sf E}_{\infty,t}^{\pmb{U^\star}}:=\mathcal{L}({U}^{i\star}_{\infty,t})$ almost surely in $\omega$, respectively. 

\item [{\it Step 2.}] Let ${\sf E}_{n}^{\pmb{X^\star}}:=\{{\sf E}_{n,t}^{\pmb{X^\star}}\}_{t\in \mathbb{T}}$, ${\sf E}_{n}^{\pmb{U^\star}}:=\{{\sf E}_{n,t}^{\pmb{U^\star}}\}_{t\in \mathbb{T}}$, ${\sf E}_{\infty}^{\pmb{X^\star}}:=\{{\sf E}_{\infty,t}^{\pmb{X^\star}}\}_{t\in \mathbb{T}}$, and ${\sf E}_{\infty}^{\pmb{U^\star}}:=\{{\sf E}_{\infty,t}^{\pmb{U^\star}}\}_{t\in \mathbb{T}}$. Following step 1, for every $t\in \mathbb{T}$, $\nu^{n\star}_t$, the joint distribution on $\pmb{X^{i\star}_{n}},\pmb{U^{i\star}_{n}},{\sf E}_{n}^{\pmb{X^\star}}, {\sf E}_{n}^{\pmb{U^\star}}$, converges in distribution (almost surely in $\omega$) converges to $\nu^{\star}_t$, the joint distribution on $\pmb{X^{i\star}_{\infty}},\pmb{U^{i\star}_{\infty}},{\sf E}_{\infty}^{\pmb{X^\star}}, {\sf E}_{\infty}^{\pmb{U^\star}}$. Hence, we can show that
\begin{align}
&\limsup_{N\to \infty}\inf_{P_{\pi}^{N}\in \LEXN}\int \left(\hat{c}_{N}(\pmb{X^{1:N}},\pmb{\gamma^{1:N}})\prod_{i=1}^{N}\mu^i(d\pmb{X^{i}};\pmb{\gamma^{i}})\right) P_{\pi}^{N}(d\pmb{\gamma^{1:N}})\nonumber\\
&= \lim_{n\to \infty} \int \left(\int \int_{0}^{T} \frac{1}{n}\sum_{i=1}^{n}\hat{c}(X_{t}^{i},U_{t}^{i},{\sf E}_{n,t}^{\pmb{X^\star}},{\sf E}_{n,t}^{\pmb{U^\star}} )\prod_{i=1}^{n}\pmb{\gamma^{i}}(d{U^{i}},dt)\mu^i(d\pmb{X^{i}};\pmb{\gamma^{i}})\right) P_{\pi}^{n\star}(d\pmb{\gamma}^{1:N})\nonumber\\
&=\lim_{n\to \infty} \int \int \int_{0}^{T} \hat{c}(X_{n,t}^{1\star},U_{t,n}^{1\star},{\sf E}_{n,t}^{\pmb{X^\star}},{\sf E}_{n,t}^{\pmb{U^\star}} )\nu^{n\star}_t\left(d\pmb{X^{1}},d\pmb{U^{1}},d{\sf E}_{n}^{\pmb{X^\star}},d{\sf E}_{n}^{\pmb{U^\star}}\right)dt\label{pf-s2-lem3-1}\\
&=\int \int \int_{0}^{T} \hat{c}(X_{t}^{1},U_{t}^{1},{\sf E}_{\infty,t}^{\pmb{X^\star}},{\sf E}_{\infty,t}^{\pmb{U^\star}} )\nu^{\star}_t\left(d\pmb{X^{1}},d\pmb{U^{1}},d{\sf E}_{\infty}^{\pmb{X^\star}},d{\sf E}_{\infty}^{\pmb{U^\star}}\right)dt\label{pf-s2-lem3-2}\\
&\geq \limsup_{N\to \infty}\inf_{P_{\pi}\in \LEX}\int \left(\hat{c}_{N}(\pmb{X}^{1:N},\pmb{\gamma^{1:N}})\prod_{i=1}^{N}\mu^i(d\pmb{X^{i}};\pmb{\gamma^{i}})\right) P_{\pi,N}(d\pmb{\gamma^{1:N}})\label{pf-s2-lem3-3},
\end{align}
where \eqref{pf-s2-lem3-1} follows from the fact that the induced $\pmb{X^{1\star:n\star}}$, $\pmb{U^{1\star:n\star}}$, and $\pmb{\gamma^{1\star:n\star}}$ are exchangeable. Equality \eqref{pf-s2-lem3-2} follows from step 1 using the generalized dominated convergence theorem, and \eqref{pf-s2-lem3-3} follows from the fact that the limiting policy is infinitely exchangeable. This completes the proof. 
\end{itemize}
\end{proof}

\section{Proof of Theorem \ref{the:6}}\label{app:the:6}


 We have
\begin{flalign}
&\inf\limits_{P_{\pi} \in L}\limsup\limits_{N \to \infty} \int \left(\widehat{c}_{N}(\pmb{X}^{1:N},\pmb{\gamma^{1:N}})\prod_{i=1}^{N}\mu^i(d\pmb{X^{i}};\pmb{\gamma^{i}})\right) P_{\pi,N}(d\pmb{\gamma^{1:N}})\nonumber\\
&\geq\limsup\limits_{N \to \infty}\inf\limits_{P_{\pi}^{N} \in L^{N}}\int \left(\widehat{c}_{N}(\pmb{X}^{1:N},\pmb{\gamma^{1:N}})\prod_{i=1}^{N}\mu^i(d\pmb{X^{i}};\pmb{\gamma^{i}})\right) P_{\pi}^{N}(d\pmb{\gamma^{1:N}})\label{eq:q4.1}\\
&=\limsup\limits_{N \to \infty}\inf\limits_{P_{\pi}^{N} \in \LEXN}\int \left(\widehat{c}_{N}(\pmb{X}^{1:N},\pmb{\gamma^{1:N}})\prod_{i=1}^{N}\mu^i(d\pmb{X^{i}};\pmb{\gamma^{i}})\right) P_{\pi}^{N}(d\pmb{\gamma^{1:N}})\label{eq:q4.2}\\
&\geq  \lim\limits_{M \to \infty}\limsup\limits_{N \to \infty}\inf\limits_{P_{\pi}^{N} \in \LEXN}\int \left(\min{\{M,\widehat{c}_{N}(\pmb{X}^{1:N},\pmb{\gamma^{1:N}})\}}\prod_{i=1}^{N}\mu^i(d\pmb{X^{i}};\pmb{\gamma^{i}})\right) P_{\pi}^{N}(d\pmb{\gamma^{1:N}})\label{eq:q4.s2}\\
&=  \lim\limits_{M \to \infty}\limsup\limits_{N \to \infty}\inf\limits_{P_{\pi} \in \LEX}\int \left(\min{\{M,\widehat{c}_{N}(\pmb{X}^{1:N},\pmb{\gamma^{1:N}})\}}\prod_{i=1}^{N}\mu^i(d\pmb{X^{i}};\pmb{\gamma^{i}})\right) P_{\pi,N}(d\pmb{\gamma^{1:N}})\label{eq:q4.3}\\
&= \lim\limits_{M \to \infty}\limsup\limits_{N \to \infty}\inf\limits_{P_{\pi}^{N} \in \LCOSN}\int \left(\min{\{M,\widehat{c}_{N}(\pmb{X}^{1:N},\pmb{\gamma^{1:N}})\}}\prod_{i=1}^{N}\mu^i(d\pmb{X^{i}};\pmb{\gamma^{i}})\right) P_{\pi}^{N}(d\pmb{\gamma^{1:N}})\label{eq:4.1}\\
&= \lim\limits_{M \to \infty}\limsup\limits_{N \to \infty}\inf\limits_{P_{\pi}^{N} \in \LPRS^{N}}\int \left(\min{\{M,\widehat{c}_{N}(\pmb{X}^{1:N},\pmb{\gamma^{1:N}})\}}\prod_{i=1}^{N}\mu^i(d\pmb{X^{i}};\pmb{\gamma^{i}})\right) P_{\pi}^{N}(d\pmb{\gamma^{1:N}})\label{eq:4.3}\\
&\geq \inf\limits_{P_{\pi} \in \LPRS}\limsup\limits_{N \to \infty}  \int \left(\widehat{c}_{N}(\pmb{X}^{1:N},\pmb{\gamma^{1:N}})\prod_{i=1}^{N}\mu^i(d\pmb{X^{i}};\pmb{\gamma^{i}})\right) P_{\pi,N}(d\pmb{\gamma^{1:N}})\label{eq:4.4}\\
&\geq\inf\limits_{P_{\pi} \in L} \limsup\limits_{N \to \infty}  \int \left(\widehat{c}_{N}(\pmb{X}^{1:N},\pmb{\gamma^{1:N}})\prod_{i=1}^{N}\mu^i(d\pmb{X^{i}};\pmb{\gamma^{i}})\right) P_{\pi,N}(d\pmb{\gamma^{1:N}})\label{eq:4.6},
\end{flalign}
where \eqref{eq:q4.1} follows from exchanging limsup with infimum. Equality \eqref{eq:q4.2} follows from Theorem \ref{the:4} and \eqref{eq:q4.s2} follows from the fact that $\min{\{M, \widehat{c}_{N}(\pmb{X}^{1:N},\pmb{\gamma^{1:N}})\}}\leq \widehat{c}_{N}(\pmb{X}^{1:N},\pmb{\gamma^{1:N}})$. Equality \eqref{eq:q4.3} follows from Lemma \ref{lem:3}, and \eqref{eq:4.1} follows from de Finetti Theorem in Lemma \ref{lem:0}(i) since $\LEX=\LCOS$ and hence their restrictions to $N$ first components are the same $\LEX|_N=\LCOSN$. The set of extreme points of the convex set $\LCOSN$ is in $\LPRS^N$ (that is because, $\LCOSN$ corresponds to the randomized policies with common and individual independent randomness, where each DM selects an identical randomized policy), hence, \eqref{eq:4.3} is true since $\LCOSN$ is convex, and the map $ J^{\pi}_{N}:\LCOSN \to \mathbb{R}$ is linear. Inequality \eqref{eq:4.4} follows from an analogous argument as that used in parts (i)--(ii) of the proof of Theorem \ref{the:4}. Inequality \eqref{eq:4.6} follows from the fact that $\LPRS \subset L$. Hence, the chain of inequalities \eqref{eq:q4.1}--\eqref{eq:4.6} is the chain of equalities. 

\section{Proof of Theorem \ref{the:7}}\label{APP:pfthe7}

Following Theorem \ref{the:5-coup}, we can restrict our search for an optimal solution to the set of exchangeable policies $\LEXN$ for \PN\ without any loss of optimality. Let $\epsilon\geq0$ and  $P_{\pi}^{N\star} \in  \LEXN$ an $\epsilon$-optimal policy for \PN, i.e.,
\begin{align}\label{eq:asy32.1}
   J^{\pi}_{N}(P_{\pi}^{ N\star}) \leq \inf\limits_{P_{\pi}^{N} \in \LEXN}J^{\pi}_{N}(P_{\pi}^{ N}) + \epsilon
\end{align}
Following from the proof of Lemma \ref{lem:3-coup}, using $\{P_{\pi}^{N\star}\}_{N} \subseteq  \LEXN$, by considering the indexes as a sequence of i.i.d. random variables with uniform distribution on the set $\{1,\dots,N\}$, we can construct a sequence of  infinitely exchangeable policies $\{P_{\pi}^{N,\infty\star}\}$ where the restriction of an infinitely exchangeable policy to $N$ first components, $P^{\infty\star}_{\pi, N} \in \LEX\big{|}_{N}$, satisfies 
\begin{flalign}
& J^{\pi}_{N}(P^{\infty\star}_{\pi, N}) \leq  J^{\pi}_{N}(P_{\pi}^{ N\star})+\widehat{\epsilon_{N}
}\label{eq:asymendif}.
\end{flalign}
for some $\widehat{\epsilon_{N}
}\geq 0$ such that $\widehat{\epsilon_{N}
}$ goes to $0$ as $N\to \infty$. Since $\epsilon>0$ is arbitrary, \eqref{eq:asy32.1} and \eqref{eq:asymendif} imply that
\begin{flalign}\label{eq:6.49}
\inf\limits_{P_{\pi}^{N} \in \LEX|_N}J^{\pi}_{N}(P_{\pi}^{ N})\leq \inf\limits_{P_{\pi}^{N} \in \LEXN}J^{\pi}_{N}(P_{\pi}^{ N})+\widehat{\epsilon_{N}
}.
\end{flalign}
 By Lemma \ref{lem:0}(i), $\LEX=\LCOS$ and hence their restrictions to $N$ first components are the same $\LEX|_N=\LCOSN$. Since $\LCOSN$ is convex with extreme points in  $\LPRS^N$, the linearity of the expected cost in randomized policies together with \eqref{eq:6.49} yields that
 \begin{flalign}\label{eq:approx-exc}
\inf\limits_{P_{\pi}^{N} \in \LPRS^N}J^{\pi}_{N}(P_{\pi}^{ N})\leq \inf\limits_{P_{\pi}^{N} \in \LEXN}J^{\pi}_{N}(P_{\pi}^{ N})+\widehat{\epsilon_{N}
}.
\end{flalign}
 Let $P_{\pi}^{\star}\in \LPRS$ be an optimal policy for \PIN\ and $P_{\pi,N}^{\star}$ is the restriction of $P_{\pi}^{
\star}$ to the first $N$ components. Following from the proof of Theorem \ref{the:6-coup}, since the cost function is bounded, we have
\begin{flalign}
J_{N}^{\pi}(P_{\pi,N}^{\star})\leq \inf\limits_{P_{\pi}^{N} \in \LPRS^N} J_{N}^{\pi}(P_{\pi}^{N})+\widetilde{\epsilon_{N}}\label{eq:L23}
\end{flalign}
for some $\widetilde{\epsilon_{N}}\geq 0$ such that $\widetilde{\epsilon_{N}
}$ goes to $0$ as $N\to \infty$. Letting $\epsilon_{N}=\widetilde{\epsilon_{N}}+\widehat{\epsilon_{N}}$,  \eqref{eq:L23} and \eqref{eq:approx-exc} imply that
\begin{flalign}
J_{N}^{\pi}(P_{\pi,N}^{\star})\leq \inf\limits_{P_{\pi}^{N} \in \LEX^N} J_{N}^{\pi}(P_{\pi}^{N})+{\epsilon_{N}},\label{eq:approxl}
\end{flalign}
which completes the proof thanks to Theorem \ref{the:5-coup}.

\section{Proof of Lemma \ref{lem:2-coup}}\label{app:lem5}
  The set $\Gamma^i$ is tight for all $i\in \mathcal{N}$ since marginals of $\mathcal{P}(\mathbb{U}\times [0,T])$ on $[0,T]$ is fixed and $\mathbb{U}$ is compact.  Consider a sequence of $\{\pmb{\gamma^{i}_k}(\omega)\}_{k}\subseteq \Gamma^{i}$ that converges to $\pmb{\gamma^{i}}(\omega)$ weakly as $k\to \infty$ for all $i\in \mathcal{N}$. {Following \cite[Theorem 4.21]{carmona2018probabilistic}, using Banach fixed point Theorem, \eqref{eq:mf-dynamic-dec} 
  admits a unique strong solution for any open-loop adapted policy. This implies that the strong unique solution $\pmb{X^i_{k}}$ of  \eqref{eq:mf-dynamic-dec} is also adapted. The unique weak solution  $\pmb{X^i_{k}}$ of  \eqref{eq:mf-dynamic-dec} exists under non-anticipative policy $\pmb{\gamma^{i}_k}$. 
  Hence, the unique weak solution of \eqref{eq:mf-dynamic-dec} is also the unique weak solution of the following martingale equation
    \begin{align}\label{mart-20}
        \mathbb{E}_{x_0}\left[f(X_{t,k}^{i})-f(X_{0,k}^i)-\int_{0}^{t}\int_{\mathbb{U}}\mathcal{A}^{{U^{i}}}(X_{s,k}^{i})\pmb{\gamma^{i}_k}(d{U^{i}},ds)\right] = 0,
    \end{align}
 for any twice continuously differentiable function $f:\mathbb{X}\to \mathbb{R}$,   where 
\begin{align}
\mathcal{A}^{{U^{i}}}(X_{s,k}^{i}):={\sf Tr}\left(a_{s}(X_{s,k}^i)\nabla^2 f(X_{s,k}^{i})\right)+b_{s}(X_{s,k}^{i},{U^{i}}, \mathcal{L}(X_{s,k}^{i}),\mathcal{L}({U^{i}}))\cdot \nabla f(X_{s,k}^{i}).
\end{align}
Using Aldous' criterion for tightness in \cite[Theorem 16.11]{kallenberg2021foundations},  a family of $\mathbb{R}^d$-valued continuous stochastic processes $\{\pmb{X^i_{k}}\}_k$ is tight since $b_t$ and $\sigma_t$ are bounded.} Hence, there exists a subsequence $\{\pmb{X^i_{n}}\}_n$  that converges in distribution to $\pmb{X^i}$ as $n\to \infty$. By Skorohod’s representation theorem, there exists a probability space in which $\{\pmb{\gamma^{i}_n}(\omega)\}_{n}$ and $\{\pmb{X^i_{n}}(\omega)\}$ converge to $\pmb{\gamma^{i}}(\omega)$ and $\pmb{X^i}(\omega)$ almost surely (in $\omega$),  respectively as $n\to \infty$. By the
generalized dominated convergence \cite[Theorem 3.5]{serfozo1982convergence}, using an analogous argument as that utilized in the proof of Lemma \ref{lem:2}, we get
\begin{align}\label{mart-2-dec}
\mathbb{E}_{x_0}\bigg[&f(X_{t}^{i})-f(X_{0}^i)\nonumber\\
&-\int_{0}^{t}  \left({\sf Tr} \left(a_{s}(X_{s}^i)\nabla^2 f(X_{s}^{i})\right)+\int_{\mathbb{U}}b_{s}(X_{s}^{i},{U^{i}},\mathcal{L}(X_{s}^{i}),\mathcal{L}({U^{i}}))\cdot \nabla f(X_{s}^{i})\pmb{\gamma^{i}}(d{U^{i}}|s)\right)ds\bigg] = 0
\end{align}
for any twice continuously differentiable function $f:\mathbb{X}\to \mathbb{R}$. This implies that $\pmb{X^i}$ solves the martingale equation, and hence, it is a solution of SDE \eqref{eq:mf-dynamic-dec}. This completes the proof.

\section{Proof of Lemma \ref{lem:contZ}}\label{app:lem-con}
Let $\{\pmb{\gamma^{i}_n}(\omega)\}_{n}$ and $\{\pmb{X^i_{n}}(\omega)\}_n$ converge weakly to $\pmb{\gamma^{i}}(\omega)$ and $\pmb{X^i}(\omega)$ almost surely (in $\omega$), respectively as $n\to \infty$. Let $Z^{i,n}_{N}:={\sf exp}(A_n^i-B_n^i)$ and $Z^{i, \infty}_{N}:={\sf exp}(A^i-B^i)$ with $A_{n}^{i},B_{n}^{i},A^i$ and $B^i$ as \eqref{eq:A-B1}--\eqref{eq:A-B4}.
In the following, We use a similar analysis as that in the proof of  \cite[Lemma 2.5]{pradhan2023controlled} to show that 
\begin{align}
    \lim_{n\to \infty} \mathbb{E} \left[\left|\prod_{i=1}^{N}Z_{N}^{i,n}-\prod_{i=1}^{N}Z_{N}^{i, \infty}\right|\right]=0.
\end{align}
 We have 
\begin{align}
    &\left|{\sf exp}(\sum_{i=1}^{N}(A_n^i-B_n^i))-{\sf exp}(\sum_{i=1}^{N}(A^i-B^i))\right|\nonumber\\
    &\leq \left|\sum_{i=1}^{N}((A_n^i-A^i)-(B_n^i-B^i))\right| \max\left({\sf exp}(\sum_{i=1}^{N}(A_n^i-B_n^i)),{\sf exp}(\sum_{i=1}^{N}(A^i-B^i))\right)\label{eq:expon}
\end{align}
By Cauchy-Schwarz inequality, we have
\begin{align}\label{eq:Z-term}
    \left(\mathbb{E} \left[\left|\prod_{i=1}^{N}Z_{N}^{i,n}-\prod_{i=1}^{N}Z_{N}^{i}\right|\right]\right)^2&\leq \mathbb{E} \left[\left|\sum_{i=1}^{N}((A_n^i-A^i)-(B_n^i-B^i))\right|^2\right]\\
    &\times \mathbb{E} \left[\max\left({\sf exp}(2\sum_{i=1}^{N}(A_n^i-B_n^i)),{\sf exp}(2\sum_{i=1}^{N}(A^i-B^i))\right)\right]
\end{align}
 We have
\begin{align}
    \mathbb{E} \left[\left|\sum_{i=1}^{N}((A_n^i-A^i)-(B_n^i-B^i))\right|^2\right]\leq 2N^2\sup_{1\leq i\leq N} \max\left(\mathbb{E}[|A_{n}^i-A^i|^2],\mathbb{E}[|B_{n}^i-B^i|^2]\right)\label{eq:AB-term}
\end{align}
Let 
\begin{align*}
    Y_{t,n}^{i}:=&\sigma^{-1}_t(X_{t,n}^{i})\int\overline{b_{t}}\left(X_{t,n}^i,U_{t}^i, \frac{1}{N}\sum_{p=1}^{N} \delta_{X^{p}_{t,n}},\frac{1}{N}\sum_{p=1}^{N} \delta_{U^{p}_t}\right)\prod_{i=1}^{N}\pmb{\gamma^i_n}(d{U^i}|t)\\
    &-\sigma^{-1}_t(X_{t}^{i})\int\overline{b_{t}}\left(X_{t}^i,U_{t}^i, \frac{1}{N}\sum_{p=1}^{N} \delta_{X^{p}_{t}},\frac{1}{N}\sum_{p=1}^{N} \delta_{U^{p}_t}\right)\prod_{i=1}^{N}\pmb{\gamma^i}(d{U^i}|t).
\end{align*}
For every $i=1, \ldots, N$, we have
\begin{align*}
  \lim_{n\to \infty}  \mathbb{E}[|A_{n}^i-A^i|^2]&=\lim_{n\to \infty}\mathbb{E}\left[\left(\int_{0}^{T} Y_{t,n}^{i}dW_{t}^{i}\right)^2\right]\\
&=\lim_{n\to \infty}\mathbb{E}\left[\int_{0}^{T} \left(Y_{t,n}^{i}\right)^2dt\right]\\
&=0,
\end{align*}
 where the second equality follows from It\^o isometry and the last equality follows from the dominated convergence theorem since $Y_{t,n}^{i}$ is continuous and also bounded (following from continuity and boundedness of $\sigma_t$ and $b_t$ using again the generalized dominated convergence theorem). Similarly, we can show that $\mathbb{E}[|B_{n}^i-B^i|^2]$ goes to $0$ as $n\to \infty$, and hence, the right hand side of \eqref{eq:AB-term} goes to $0$ as $n\to \infty$.

On the other hand, we have
\begin{align*}
    &\mathbb{E} \left[\max\left({\sf exp}(2\sum_{i=1}^{N}(A_n^i-B_n^i)),{\sf exp}(2\sum_{i=1}^{N}(A^i-B^i))\right)\right]\\&\leq \mathbb{E} \left[{\sf exp}(2\sum_{i=1}^{N}(A_n^i-B_n^i))+{\sf exp}(2\sum_{i=1}^{N}(A^i-B^i))\right]\\
    &\leq \mathbb{E} \left[{\sf exp}(2\sum_{i=1}^{N}(A_n^i-B_n^i))\right]+\mathbb{E} \left[{\sf exp}(2\sum_{i=1}^{N}(A^i-B^i))\right].
\end{align*}

Since $\sigma_t$ and $b_t$ are uniformly bounded, following proof of \cite[Lemma 2.5]{pradhan2023controlled}, we have  
 \begin{align}
     \mathbb{E} \left[{\sf exp}(2\sum_{i=1}^{N}(A_n^i-B_n^i))\right]&=\mathbb{E} \left[{\sf exp}(2\sum_{i=1}^{N}(A_n^i-2B_n^i)){\sf exp}(2\sum_{i=1}^{N}B_n^i))\right]\\
     &\leq {\sf exp}(2NK).
 \end{align}
 for some $K\in \mathbb{R}_{+}$. This together with \eqref{eq:Z-term} completes the proof.
 
\section{Proof of Theorem \ref{the:5-coup}}\label{app:the8}
 We first show that the mapping $J_{N}^{\pi}:\LEXN \to \mathbb{R}$ is continuous. Suppose that $\{P_{\pi, n}\}_{n}\subseteq\LEXN$ (or $\LEX$) converges to $P_{\pi, \infty}$ as $n\to \infty$. By Lemma \ref{lem:2-coup}, $\{P_{0,n}^i\}_n$, sequence of the distribution of $\pmb{X^i_{n}}$ under the measure transformation, converges weakly to $P_{0}^i$, the distribution of $\pmb{X^i}$ for all $i\in \mathcal{N}$. By Skorohod’s representation theorem, there exists a probability space in which $\{\pmb{\gamma^{i}_n}(\omega)\}_{n}$ and $\{\pmb{X^i_{n}}(\omega)\}$ converge weakly to $\pmb{\gamma^{i}}(\omega)$ and $\pmb{X^i}(\omega)$ almost surely (in $\omega$), respectively as $n\to \infty$. We have
 \begin{align}
    & \lim_{n\to \infty} \bigg|\mathbb{E} \left[\prod_{i=1}^{N}Z_{N}^{i,n} \left(\int_{0}^{T}\int_{\mathbb{U}}\frac{1}{N}\sum_{i=1}^{N}\hat{c}\left(X_{t,n}^{i},U_{t}^{i},\frac{1}{N}\sum_{p=1}^{N} \delta_{X_{t,n}^{p}},\frac{1}{N}\sum_{p=1}^{N} \delta_{U_{t}^{p}} \right) \prod_{i=1}^{N}\pmb{\gamma^i_n}(d{U^i},dt)\right)\right]\nonumber\\
     &-\mathbb{E} \left[\prod_{i=1}^{N}Z_{N}^{i,\infty} \left(\int_{0}^{T}\int_{\mathbb{U}}\frac{1}{N}\sum_{i=1}^{N}\hat{c}\left(X_{t}^{i},U_{t}^{i},\frac{1}{N}\sum_{p=1}^{N} \delta_{X_{t}^{p}},\frac{1}{N}\sum_{p=1}^{N} \delta_{U_{t}^{p}} \right) \prod_{i=1}^{N}\pmb{\gamma^i}(d{U^i},dt)\right)\right]\bigg|\nonumber\\
     &\leq \lim_{n\to \infty} \bigg|\mathbb{E} \left[\prod_{i=1}^{N}Z_{N}^{i,\infty} \left(\int_{0}^{T}\int_{\mathbb{U}}\frac{1}{N}\sum_{i=1}^{N}\hat{c}\left(X_{t,n}^{i},U_{t}^{i},\frac{1}{N}\sum_{p=1}^{N} \delta_{X_{t,n}^{p}},\frac{1}{N}\sum_{p=1}^{N} \delta_{U_{t}^{p}} \right) \prod_{i=1}^{N}\pmb{\gamma^i_n}(d{U^i},dt)\right)\right]\nonumber\\
     &-\mathbb{E} \left[\prod_{i=1}^{N}Z_{N}^{i,\infty} \left(\int_{0}^{T}\int_{\mathbb{U}}\frac{1}{N}\sum_{i=1}^{N}\hat{c}\left(X_{t}^{i},U_{t}^{i},\frac{1}{N}\sum_{p=1}^{N} \delta_{X_{t}^{p}},\frac{1}{N}\sum_{p=1}^{N} \delta_{U_{t}^{p}} \right) \prod_{i=1}^{N}\pmb{\gamma^i}(d{U^i},dt)\right)\right]\bigg|\label{eq:FE}\\
     &+ M \lim_{n\to \infty} \mathbb{E} \left[\left|\prod_{i=1}^{N}Z_{N}^{i,n}-\prod_{i=1}^{N}Z_{N}^{i,\infty}\right|\right]\label{eq:FE2},
 \end{align}
where $M$ is the uniform bound for the cost function $\hat{c}$, $Z^{i,n}_{N}:={\sf exp}(A_n^i-B_n^i)$ and $Z^{i,\infty}_{N}:={\sf exp}(A^i-B^i)$ with $A_{n}^{i},B_{n}^{i},A^i$ and $B^i$ as \eqref{eq:A-B1}--\eqref{eq:A-B4}. The generalized convergence theorem implies that the expression in \eqref{eq:FE} is zero and Lemma \ref{lem:contZ} implies that the expression in \eqref{eq:FE2} is zero. 

Hence, the mapping $J_{N}^{\pi}:\LEXN \to \mathbb{R}$ is continuous. This implies that there exists an exchangeable randomized globally optimal policy $P_{\pi}^{N}$ for \PN\ among all randomized policies $\LEXN$ since $\LEXN$ is compact. The rest of the proof proceeds by following steps (ii)-(iii) of the proof of Theorem \ref{the:5}.

\section{Proof of Lemma \ref{lem:3-coup}}\label{app:lem6}

Similar to the proof of Lemma \ref{lem:3}, we have for every $t\in \mathbb{T}$, $\nu^{n\star}_t$, the joint distribution on $\pmb{X^{i\star}_{n}},\pmb{U^{i\star}_{n}},{\sf E}_{n}^{\pmb{X^\star}}, {\sf E}_{n}^{\pmb{U^\star}}$, converges in distribution (almost surely in $\omega$) converges to $\nu^{\star}_t$, the joint distribution on $\pmb{X^{i\star}_{\infty}},\pmb{U^{i\star}_{\infty}},{\sf E}_{\infty}^{\pmb{X^\star}}, {\sf E}_{\infty}^{\pmb{U^\star}}$. Hence, we can show that
\begin{align}
&\limsup_{N\to \infty}\inf_{P_{\pi}^{N}\in \LEXN}\int \prod_{i=1}^{N}Z^{i}_{N} \:\hat{c}_{N}(\pmb{X^{1:N}},\pmb{\gamma^{1:N}})\prod_{i=1}^{N}P^i_{0,N}(d\pmb{X^{i}};\pmb{\gamma^{i}}) P_{\pi}^{N}(d\pmb{\gamma^{1:N}})\nonumber\\
&= \lim_{n\to \infty} \int \prod_{i=1}^{n}Z_{n}^{i}\int_{0}^{T} \int \frac{1}{n}\sum_{i=1}^{n}\hat{c}(X_{t}^{i},U_{t}^{i},{\sf E}_{n,t}^{\pmb{X^\star}},{\sf E}_{n,t}^{\pmb{U^\star}} )\prod_{i=1}^{n}\pmb{\gamma^{i}_n}(d{U^{i}},dt)P^i_{0,n}(d\pmb{X^{i}};\pmb{\gamma^{i}}) \nonumber\\
&= \lim_{n\to \infty} \int \left(\prod_{i=1}^{n} Z^{i}_{n} - 1\right)\int_{0}^{T} \int \frac{1}{n}\sum_{i=1}^{n}\hat{c}(X_{t}^{i},U_{t}^{i},{\sf E}_{n,t}^{\pmb{X^\star}},{\sf E}_{n,t}^{\pmb{U^\star}} )\prod_{i=1}^{n}\pmb{\gamma^{i}_n}(d{U^{i}},dt)P^i_{0,n}(d\pmb{X^{i}};\pmb{\gamma^{i}})\label{pf-s2-lem3-0-dec}\\
&+\lim_{n\to \infty} \int \int_{0}^{T} \hat{c}(X_{n,t}^{1\star},U_{t,n}^{1\star},{\sf E}_{n,t}^{\pmb{X^\star}},{\sf E}_{n,t}^{\pmb{U^\star}} )\nu^{n\star}_t\left(d\pmb{X^{1}},d\pmb{U^{1}},d{\sf E}_{n}^{\pmb{X^\star}},d{\sf E}_{n}^{\pmb{U^\star}}\right)dt\label{pf-s2-lem3-1-dec}\\
&=\int \int \int_{0}^{T} \hat{c}(X_{t}^{1},U_{t}^{1},{\sf E}_{\infty,t}^{\pmb{X^\star}},{\sf E}_{\infty,t}^{\pmb{U^\star}} )\nu^{\star}_t\left(d\pmb{X^{1}},d\pmb{U^{1}},d{\sf E}_{\infty}^{\pmb{X^\star}},d{\sf E}_{\infty}^{\pmb{U^\star}}\right)dt\label{pf-s2-lem3-2-dec}\\
&\geq \limsup_{N\to \infty}\inf_{P_{\pi}\in \LEX}\int \left(\hat{c}_{N}(\pmb{X}^{1:N},\pmb{\gamma^{1:N}})\prod_{i=1}^{N}\mu^i(d\pmb{X^{i}};\pmb{\gamma^{i}})\right) P_{\pi,N}(d\pmb{\gamma^{1:N}})\label{pf-s2-lem3-3-dec}.
\end{align}
Expression in \eqref{pf-s2-lem3-0-dec} is equal to $0$ as
\begin{align*}
&\lim_{n\to \infty} \int \left(\prod_{i=1}^{n} Z^{i}_{n} - 1\right)\int_{0}^{T} \int \frac{1}{n}\sum_{i=1}^{n}\hat{c}(X_{t}^{i},U_{t}^{i},{\sf E}_{n,t}^{\pmb{X^\star}},{\sf E}_{n,t}^{\pmb{U^\star}} )\prod_{i=1}^{n}\pmb{\gamma^{i}_n}(d{U^{i}},dt)P^i_{0,n}(d\pmb{X^{i}};\pmb{\gamma^{i}})\\
&\leq M \lim_{n\to \infty} \int \left|\prod_{i=1}^{n} Z^{i}_{n} - 1\right|P^i_{0,n}(d\pmb{X^{i}};\pmb{\gamma^{i}})\\
&=0,
\end{align*}
where the first inequality follows the fact that the cost is uniformly bounded and the last equality follows from Assumption \ref{eq:continuity-RD}. Equality \eqref{pf-s2-lem3-1-dec} follows from the fact that the induced $\pmb{X^{1\star:n\star}}$, $\pmb{U^{1\star:n\star}}$, and $\pmb{\gamma^{1\star:n\star}}$ are exchangeable. Equality \eqref{pf-s2-lem3-2-dec} follows from step 1 using the generalized dominated convergence theorem, and \eqref{pf-s2-lem3-3-dec} follows from the fact that the limiting policy is infinitely exchangeable. This completes the proof. 

\section{Proof of Theorem \ref{the:6-coup}}\label{app:the9}
 Along the same reasoning as that used in the chain of inequalities in \eqref{eq:q4.1}--\eqref{eq:4.6}, we only need to show that 
\begin{align}
& \lim_{M\to \infty}\limsup\limits_{N \to \infty}\inf\limits_{P_{\pi}^{N} \in \LPRS^{N}}\int \prod_{i=1}^{N}Z^{i}_{N} \:\min\{M,\hat{c}_{N}(\pmb{X^{1:N}},\pmb{\gamma^{1:N}})\}\prod_{i=1}^{N}P^i_{0,N}(d\pmb{X^{i}};\pmb{\gamma^{i}}) P_{\pi}^{N}(d\pmb{\gamma^{1:N}})\label{eq:4.3-dec}\\
&\geq \inf\limits_{P_{\pi} \in \LPRS}\limsup\limits_{N \to \infty}  \int \left(\hat{c}_{N}(\pmb{X}^{1:N},\pmb{\gamma^{1:N}})\prod_{i=1}^{N}P^i_{0}(d\pmb{X^{i}};\pmb{\gamma^{i}})\right) P_{\pi,N}(d\pmb{\gamma^{1:N}})\label{eq:4.4-dec}.
\end{align}

 Denote the index of a converging sequence of symmetric optimal policies $\pmb{\gamma_n^{\star}}$ of \PN\ by $n$. We have
\begin{align}
   &\lim_{M\to \infty}\lim_{n\to \infty} \mathbb{E} \bigg[\prod_{i=1}^{n}Z_{n}^{i,n}\bigg(\int_{0}^{T}\int_{\mathbb{U}}\frac{1}{n}\sum_{i=1}^{N}\min\left\{M, \hat{c}\left(X_{t,n}^{i},U_{t}^{i},\frac{1}{n}\sum_{p=1}^{n} \delta_{X_{t,n}^{p}},\frac{1}{N}\sum_{p=1}^{n} \delta_{U_{t}^{p}} \right)\right\}\nonumber\\
   &\qquad \qquad \qquad\times \prod_{i=1}^{n}\pmb{\gamma_n^{\star}}(d{U^i},dt)\bigg)\bigg]\nonumber\\
   &=\lim_{M\to \infty}\lim_{n\to \infty} \mathbb{E} \bigg[ \int_{0}^{T}\int_{\mathbb{U}}\frac{1}{n}\sum_{i=1}^{N}\min\left\{M,\hat{c}\left(X_{t,n}^{i},U_{t}^{i},\frac{1}{n}\sum_{p=1}^{n} \delta_{X_{t,n}^{p}},\frac{1}{N}\sum_{p=1}^{n} \delta_{U_{t}^{p}} \right)\right\} \nonumber\\
   &\qquad \qquad \qquad\times \prod_{i=1}^{n}\pmb{\gamma_n^{\star}}(d{U^i},dt)\bigg]\nonumber\\
   &+\lim_{M\to \infty}\lim_{n\to \infty} \mathbb{E} \bigg[(\prod_{i=1}^{n}Z_{n}^{i,n}-1)\int_{0}^{T}\int_{\mathbb{U}}\frac{1}{n}\sum_{i=1}^{N}\min\left\{M,\hat{c}\left(X_{t,n}^{i},U_{t}^{i},\frac{1}{n}\sum_{p=1}^{n} \delta_{X_{t,n}^{p}},\frac{1}{N}\sum_{p=1}^{n} \delta_{U_{t}^{p}} \right)\right\}\nonumber\\
   &\qquad \qquad \qquad\times \prod_{i=1}^{n}\pmb{\gamma_n^{\star}}(d{U^i},dt)\bigg]\nonumber\\
   &=\lim_{M\to \infty}\lim_{n\to \infty} \mathbb{E} \bigg[ \int_{0}^{T}\int_{\mathbb{U}}\frac{1}{n}\sum_{i=1}^{N}\min\left\{M,\hat{c}\left(X_{t,n}^{i},U_{t}^{i},\frac{1}{n}\sum_{p=1}^{n} \delta_{X_{t,n}^{p}},\frac{1}{N}\sum_{p=1}^{n} \delta_{U_{t}^{p}} \right)\right\} \nonumber\\
   &\qquad \qquad \qquad\times \prod_{i=1}^{n}\pmb{\gamma_n^{\star}}(d{U^i},dt)\bigg]\label{eq:32-2}\\
   &=J_{\infty}(\underline{\pmb{P_{\pi}^{\star}}})\label{eq:32-1},
\end{align}
where $Z^{i,n}_{n}:={\sf exp}(A_n^i-B_n^i)$ with
\begin{align*}
    A_{n}^{i}&=\int_{0}^{T} \sigma^{-1}_t(X_{t,n}^{i})\int\overline{b_{t}}\left(X_{t,n}^i,U_{t}^i, \frac{1}{n}\sum_{p=1}^{n} \delta_{X^{p}_{t,n}},\frac{1}{n}\sum_{p=1}^{n} \delta_{U^{p}_t}\right)\prod_{i=1}^{n}\pmb{\gamma_n^{\star}}(d{U^i}|t)dW_{t}^{i}\\
    B_{n}^{i}&=\frac{1}{2}\int_{0}^{T} \left|\sigma^{-1}_t(X_{t,n}^{i})\int_{\mathbb{U}}\overline{b_{t}}\left(X_{t,n}^i,U_{t}^i, \frac{1}{n}\sum_{p=1}^{n} \delta_{X^{p}_{t,n}},\frac{1}{n}\sum_{p=1}^{n} \delta_{U^{p}_t}\right)\prod_{i=1}^{n}\pmb{\gamma^{\star}_n}(d{U^i}|t)\right|^2dt.
    \end{align*}

The equality \eqref{eq:32-2} follows from Assumption \ref{eq:continuity-RD} since
\begin{align*}
    \lim_{n\to \infty} \mathbb{E} \left[\left|\prod_{i=1}^{n}Z_{n}^{i,n}-1\right|\right]=0
\end{align*}
for symmetrically independent policy of $\pmb{\gamma^{\star}_{n}}(\omega)$ as $n\to \infty$. Equality \eqref{eq:32-1} follows from an argument used in the proof of Theorem \ref{the:6} using the generalized dominated convergence and the monotone convergence theorems. This completes the proof. 


\begin{thebibliography}{alpha}

\bibitem[AB06]{InfiniteDimensionalAnalysis}
C.~D. Aliprantis and K.~C. Border.
\newblock {\em Infinite Dimensional Analysis: A Hitchhiker's Guide, third
  edition}.
\newblock Springer, Berlin, 2006.

\bibitem[ABG12]{arapostathis2012ergodic}
A.~Arapostathis, V.~S. Borkar, and M.~K. Ghosh.
\newblock {\em Ergodic control of diffusion processes}, volume 143.
\newblock Cambridge University Press, 2012.

\bibitem[ACFK17]{albi2017mean}
G.~Albi, Y.~Choi, M.~Fornasier, and D.~Kalise.
\newblock Mean field control hierarchy.
\newblock {\em Applied Mathematics \& Optimization}, 76:93--135, 2017.

\bibitem[AHKS22]{albi2022moment}
G.~Albi, M.~Herty, D.~Kalise, and C.~Segala.
\newblock Moment-driven predictive control of mean-field collective dynamics.
\newblock {\em SIAM Journal on Control and Optimization}, 60(2):814--841, 2022.

\bibitem[AIJ85]{aldous2006ecole}
D.~J. Aldous, I.~A. Ibragimov, and J.~Jacod.
\newblock {\em Ecole d'Ete de Probabilites de Saint-Flour XIII, 1983}, volume
  1117.
\newblock Springer, 1985.

\bibitem[AL20]{achdou2020mean}
Y.~Achdou and M.~Lauri{\`e}re.
\newblock Mean field games and applications: Numerical aspects.
\newblock {\em Mean Field Games: Cetraro, Italy 2019}, pages 249--307, 2020.

\bibitem[AM14]{Mahajan-CDC14}
J.~Arabneydi and A.~Mahajan.
\newblock Team optimal control of coupled subsystems with mean-field sharing.
\newblock In {\em 53rd IEEE Conference on Decision and Control}, pages
  1669--1674, 2014.

\bibitem[AM15]{arabneydi2015team}
J.~Arabneydi and A.~Mahajan.
\newblock Team-optimal solution of finite number of mean-field coupled {LQG}
  subsystems.
\newblock In {\em IEEE 54th Annual Conference on Decision and Control (CDC)},
  pages 5308--5313, 2015.

\bibitem[B{\"a}u23]{bauerle2023mean}
N.~B{\"a}uerle.
\newblock Mean field markov decision processes.
\newblock {\em Applied Mathematics \& Optimization}, 88(1):12, 2023.

\bibitem[BBK23]{Ali-2023}
E.~Bayraktar, N.~Bauerle, and A.~Kara.
\newblock Finite approximations for mean field type multi-agent control and
  their near optimality.
\newblock {\em arXiv preprint arXiv:2211.09633}, 2023.

\bibitem[BCP18]{bayraktar2018randomized}
E.~Bayraktar, A.~Cosso, and H.~Pham.
\newblock Randomized dynamic programming principle and {Feynman-Kac}
  representation for optimal control of {McKean-Vlasov} dynamics.
\newblock {\em Transactions of the American Mathematical Society},
  370(3):2115--2160, 2018.

\bibitem[Ben71]{benevs1971existence}
V.~E. Bene{\v{s}}.
\newblock Existence of optimal stochastic control laws.
\newblock {\em SIAM Journal on Control}, 9(3):446--472, 1971.

\bibitem[BF19]{bardi2019non}
M.~Bardi and M.~Fischer.
\newblock On non-uniqueness and uniqueness of solutions in finite-horizon mean
  field games.
\newblock {\em ESAIM: Control, Optimisation and Calculus of Variations}, 25:44,
  2019.

\bibitem[BFY15]{bensoussan2015master}
A.~Bensoussan, J.~Frehse, and S.~Yam.
\newblock The master equation in mean field theory.
\newblock {\em Journal de Math{\'e}matiques Pures et Appliqu{\'e}es},
  103(6):1441--1474, 2015.

\bibitem[Bor89]{borkar1989topology}
V.~Borkar.
\newblock A topology for {M}arkov controls.
\newblock {\em Applied Mathematics and Optimization}, 20(1):55--62, 1989.

\bibitem[BZ20]{bayraktar2020non}
E.~Bayraktar and X.~Zhang.
\newblock On non-uniqueness in mean field games.
\newblock {\em Proceedings of the American Mathematical Society},
  148(9):4091--4106, 2020.

\bibitem[CA13]{charalambous2013dynamic}
C.~D. Charalambous and N.~Ahmed.
\newblock Dynamic team theory of stochastic differential decision systems with
  decentralized noisy information structures via girsanov's measure
  transformation.
\newblock {\em arXiv preprint arXiv:1309.1913}, 2013.

\bibitem[CD15]{carmonaForwardbackward}
R.~Carmona and F.~Delarue.
\newblock {Forward–backward stochastic differential equations and controlled
  {McKean–Vlasov} dynamics}.
\newblock {\em The Annals of Probability}, 43(5):2647 -- 2700, 2015.

\bibitem[CD18]{carmona2018probabilistic}
R.~Carmona and F.~Delarue.
\newblock {\em Probabilistic Theory of Mean Field Games with Applications
  I-II}.
\newblock Springer, 2018.

\bibitem[CDL16]{carmona2016mean}
R.~Carmona, F.~Delarue, and D.~Lacker.
\newblock Mean field games with common noise.
\newblock {\em The Annals of Probability}, 44(6):3740--3803, 2016.

\bibitem[Cec21]{cecchin2021finite}
A.~Cecchin.
\newblock Finite state {N}-agent and mean field control problems.
\newblock {\em ESAIM: Control, Optimisation and Calculus of Variations}, 27:31,
  2021.

\bibitem[CF20]{cecchin2017probabilistic}
A.~Cecchin and M.~Fischer.
\newblock Probabilistic approach to finite state mean field games.
\newblock {\em Applied Mathematics \& Optimization}, 81(2):253--300, 2020.

\bibitem[Cha16]{charalambous2016decentralized}
C.~D. Charalambous.
\newblock Decentralized optimality conditions of stochastic differential
  decision problems via {G}irsanov's measure transformation.
\newblock {\em Mathematics of Control, Signals, and Systems}, 28(3):1--55,
  2016.

\bibitem[CHM17]{caines2018peter}
P.~Caines, M.~Huang, and R.~Malham{\'e}.
\newblock Mean field games.
\newblock {\em Handbook of Dynamic Game Theory}, pages 345--372, 2017.

\bibitem[CLT23]{carmona2023model}
R.~Carmona, M.~Lauri{\`e}re, and Z.~Tan.
\newblock Model-free mean-field reinforcement learning: mean-field {MDP} and
  mean-field {Q}-learning.
\newblock {\em The Annals of Applied Probability}, 33(6B):5334--5381, 2023.

\bibitem[DF80]{diaconis1980finite}
P.~Diaconis and D.~Freedman.
\newblock Finite exchangeable sequences.
\newblock {\em The Annals of Probability}, pages 745--764, 1980.

\bibitem[DPT22]{djete2022mckean}
M.~Djete, D.~Possama{\"\i}, and X.~Tan.
\newblock {McKean}--{Vlasov} optimal control: limit theory and equivalence
  between different formulations.
\newblock {\em Mathematics of Operations Research}, 47(4):2891--2930, 2022.

\bibitem[DRP73]{davison1973optimal}
E.~Davison, N.~Rau, and F.~Palmay.
\newblock The optimal decentralized control of a power system consisting of a
  number of interconnected synchronous machines.
\newblock {\em International Journal of Control}, 18(6):1313--1328, 1973.

\bibitem[DT20]{delarue2020selection}
F.~Delarue and R.~Tchuendom.
\newblock Selection of equilibria in a linear quadratic mean-field game.
\newblock {\em Stochastic Processes and their Applications}, 130(2):1000--1040,
  2020.

\bibitem[ELN13]{elliott2013discrete}
R.~Elliott, X.~Li, and Y.~Ni.
\newblock Discrete time mean-field stochastic linear-quadratic optimal control
  problems.
\newblock {\em Automatica}, 49(11):3222--3233, 2013.

\bibitem[Fis17]{fischer2017connection}
M.~Fischer.
\newblock On the connection between symmetric {N}-player games and mean field
  games.
\newblock {\em The Annals of Applied Probability}, 27(2):757--810, 2017.

\bibitem[FLOS19]{fornasier2019mean}
M.~Fornasier, S.~Lisini, C.~Orrieri, and G.~Savar{\'e}.
\newblock Mean-field optimal control as gamma-limit of finite agent controls.
\newblock {\em European Journal of Applied Mathematics}, 30(6):1153--1186,
  2019.

\bibitem[Gir60]{girsanov1960transforming}
I.~V. Girsanov.
\newblock On transforming a certain class of stochastic processes by absolutely
  continuous substitution of measures.
\newblock {\em Theory of Probability \& Its Applications}, 5(3):285--301, 1960.

\bibitem[HC72]{HoChu}
Y.~Ho and K.~C. Chu.
\newblock Team decision theory and information structures in optimal control
  problems - part {I}.
\newblock {\em IEEE Transactions on Automatic Control}, 17:15--22, February
  1972.

\bibitem[HCM06]{CainesMeanField3}
M.~Huang, P.~E. Caines, and R.~P. Malham\'e.
\newblock Large population stochastic dynamic games: closed-loop
  {M}ckean-{V}lasov systems and the {N}ash certainty equivalence principle.
\newblock {\em Communications in Information and Systems}, 6:221--251, 2006.

\bibitem[HCM07]{CainesMeanField2}
M.~Huang, P.~E. Caines, and R.~P. Malham\'e.
\newblock Large-population cost-coupled {L}{Q}{G} problems with nonuniform
  agents: Individual-mass behavior and decentralized $\epsilon$-{N}ash
  equilibria.
\newblock {\em IEEE Transactions on Automatic Control}, 52:1560--1571, 2007.

\bibitem[HCM12]{huang2012social}
M.~Huang, P.~Caines, and R.~Malham{\'e}.
\newblock Social optima in mean field {LQG} control: centralized and
  decentralized strategies.
\newblock {\em IEEE Transactions on Automatic Control}, 57(7):1736--1751, 2012.

\bibitem[HL19]{hajek2019non}
B.~Hajek and M.~Livesay.
\newblock On non-unique solutions in mean field games.
\newblock In {\em 2019 IEEE 58th Conference on Decision and Control (CDC)},
  pages 1219--1224, 2019.

\bibitem[HLL96]{HernandezLermaMCP}
O.~Hern\'andez-Lerma and J.~B. Lasserre.
\newblock {\em Discrete-Time {M}arkov Control Processes: Basic Optimality
  Criteria}.
\newblock Springer, 1996.

\bibitem[HNX07]{hespanha2007survey}
J.~Hespanha, P.~Naghshtabrizi, and Y.~Xu.
\newblock A survey of recent results in networked control systems.
\newblock {\em Proceedings of the IEEE}, 95(1):138--162, 2007.

\bibitem[Ho80]{ho1980team}
Y.~Ho.
\newblock Team decision theory and information structures.
\newblock {\em Proceedings of the IEEE}, 68(6):644--654, 1980.

\bibitem[JL23]{jackson2023approximately}
J.~Jackson and D.~Lacker.
\newblock Approximately optimal distributed stochastic controls beyond the mean
  field setting.
\newblock {\em arXiv preprint arXiv:2301.02901}, 2023.

\bibitem[Kal06]{kallenberg2006probabilistic}
O.~Kallenberg.
\newblock {\em Probabilistic symmetries and invariance principles}.
\newblock Springer Science \& Business Media, 2006.

\bibitem[Kal21]{kallenberg2021foundations}
O.~Kallenberg.
\newblock {\em Foundations of modern probability}, volume~99.
\newblock Springer, 2021.

\bibitem[KSM82]{KraMar82}
J.~C. Krainak, J.~L. Speyer, and S.~I. Marcus.
\newblock Static team problems -- {P}art {I}: {S}ufficient conditions and the
  exponential cost criterion.
\newblock {\em IEEE Transactions on Automatic Control}, 27:839--848, April
  1982.

\bibitem[Kus12]{kushner2012weak}
H.~Kushner.
\newblock {\em Weak convergence methods and singularly perturbed stochastic
  control and filtering problems}.
\newblock Springer Science \& Business Media, 2012.

\bibitem[Lac15]{lacker2015mean}
D.~Lacker.
\newblock Mean field games via controlled martingale problems: existence of
  markovian equilibria.
\newblock {\em Stochastic Processes and their Applications}, 125(7):2856--2894,
  2015.

\bibitem[Lac17]{lacker2017limit}
D.~Lacker.
\newblock Limit theory for controlled {M}ckean--{V}lasov dynamics.
\newblock {\em SIAM Journal on Control and Optimization}, 55(3):1641--1672,
  2017.

\bibitem[Lac20]{lacker2018convergence}
D.~Lacker.
\newblock On the convergence of closed-loop nash equilibria to the mean field
  game limit.
\newblock {\em The Annals of Applied Probability}, 30(4):1693--1761, 2020.

\bibitem[LL07]{LyonsMeanField}
J.~M. Lasry and P.~L. Lions.
\newblock Mean field games.
\newblock {\em Japanese J. of Mathematics}, 2:229--260, 2007.

\bibitem[LP16]{lauriere2016dynamic}
M.~Lauri{\`e}re and O.~Pironneau.
\newblock Dynamic programming for mean-field type control.
\newblock {\em Journal of Optimization Theory and Applications}, 169:902--924,
  2016.

\bibitem[Mar55]{mar55}
J.~Marschak.
\newblock Elements for a theory of teams.
\newblock {\em Management Science}, 1:127--137, 1955.

\bibitem[MMRY12]{CDCTutorial}
A.~Mahajan, N.C. Martins, M.~Rotkowitz, and S.~Y\"uksel.
\newblock Information structures in optimal decentralized control.
\newblock In {\em IEEE Conference on Decision and Control}, Hawaii, USA, 2012.

\bibitem[MP22]{motte2022mean}
M.~Motte and H.~Pham.
\newblock Mean-field {M}arkov decision processes with common noise and
  open-loop controls.
\newblock {\em The Annals of Applied Probability}, 32(2):1421--1458, 2022.

\bibitem[MP23]{motte2023quantitative}
M.~Motte and H.~Pham.
\newblock Quantitative propagation of chaos for mean field {M}arkov decision
  process with common noise.
\newblock {\em Electronic Journal of Probability}, 28:1--24, 2023.

\bibitem[NEL15]{ni2015discrete}
Y.~Ni, R.~Elliott, and X.~Li.
\newblock Discrete-time mean-field stochastic linear--quadratic optimal control
  problems, ii: Infinite horizon case.
\newblock {\em Automatica}, 57:65--77, 2015.

\bibitem[PS24]{pradhan2024continuity}
S.~Pradhan and S.Y{\"u}ksel.
\newblock Continuity of cost in borkar control topology and implications on
  discrete space and time approximations for controlled diffusions under
  several criteria.
\newblock {\em Electronic Journal of Probability}, 29:1--32, 2024.

\bibitem[PW16]{pham2016discrete}
H.~Pham and X.~Wei.
\newblock Discrete time {McKean}--{Vlasov} control problem: a dynamic
  programming approach.
\newblock {\em Applied Mathematics \& Optimization}, 74:487--506, 2016.

\bibitem[PW17]{pham2017dynamic}
H.~Pham and X.~Wei.
\newblock Dynamic programming for optimal control of stochastic
  {McKean}--{Vlasov} dynamics.
\newblock {\em SIAM Journal on Control and Optimization}, 55(2):1069--1101,
  2017.

\bibitem[PY23]{pradhan2023controlled}
S.~Pradhan and S.~Y{\"u}ksel.
\newblock Controlled diffusions under full, partial and decentralized
  information: Existence of optimal policies and discrete-time approximations.
\newblock {\em arXiv preprint arXiv:2311.03254}, 2023.

\bibitem[Rad62]{Radner}
R.~Radner.
\newblock Team decision problems.
\newblock {\em Annals of Mathematical Statistics}, 33:857--881, 1962.

\bibitem[SBR18]{saldi2018markov}
N.~Saldi, T.~Ba\c{s}ar, and M.~Raginsky.
\newblock Markov--{N}ash equilibria in mean-field games with discounted cost.
\newblock {\em SIAM Journal on Control and Optimization}, 56(6):4256--4287,
  2018.

\bibitem[Sch12]{schervish2012theory}
M.~Schervish.
\newblock {\em Theory of statistics}.
\newblock Springer Science \& Business Media, 2012.

\bibitem[Ser82]{serfozo1982convergence}
R.~Serfozo.
\newblock Convergence of {L}ebesgue integrals with varying measures.
\newblock {\em Sankhy{\=a}: The Indian Journal of Statistics, Series A}, pages
  380--402, 1982.

\bibitem[SSY23]{SSYdefinetti2020}
S.~Sanjari, N.~Saldi, and S.~Y{\"u}ksel.
\newblock Optimality of independently randomized symmetric policies for
  exchangeable stochastic teams with infinitely many decision makers.
\newblock {\em Mathematics of Operations Research}, 48(3):1213--1809, 2023.

\bibitem[SSY24a]{sanjari2024nash}
S.~Sanjari, N.~Saldi, and S.~Y{\"u}ksel.
\newblock Nash equilibria for exchangeable team against team games, their mean
  field limit, and role of common randomness.
\newblock {\em to appear SIAM Journal on Control and Optimization (also
  arXiv:2210.07339)}, 2024.

\bibitem[SSY24b]{SSYMFsharing}
S.~Sanjari, N.~Saldi, and S.~Y{\"u}ksel.
\newblock Optimality of decentralized symmetric policies for stochastic teams
  with mean-field information sharing.
\newblock {\em arXiv preprint arXiv:2404.04957}, 2024.

\bibitem[SY20]{sun2020stochastic}
J.~Sun and J.~Yong.
\newblock {\em Stochastic linear-quadratic optimal control theory: Open-loop
  and closed-loop solutions}.
\newblock Springer Nature, 2020.

\bibitem[SY21a]{sanjari2019optimal}
S.~Sanjari and S.~Y{\"u}ksel.
\newblock Optimal policies for convex symmetric stochastic dynamic teams and
  their mean-field limit.
\newblock {\em SIAM Journal on Control and Optimization}, 59(2):777--804, 2021.

\bibitem[SY21b]{sanjari2018optimal}
S.~Sanjari and S.~Y{\"u}ksel.
\newblock Optimal solutions to infinite-player stochastic teams and mean-field
  teams.
\newblock {\em IEEE Transactions on Automatic Control}, 66(3):1071--1086, 2021.

\bibitem[SY22]{saldi2022geometry}
N.~Saldi and S.~Y{\"u}ksel.
\newblock Geometry of information structures, strategic measures and associated
  stochastic control topologies.
\newblock {\em Probability Surveys}, 19:450--532, 2022.

\bibitem[TMN24]{toumi2024mean}
N.~Toumi, R.~Malham{\'e}, and J.~Le Ny.
\newblock A mean field game approach for a class of linear quadratic discrete
  choice problems with congestion avoidance.
\newblock {\em Automatica}, 160:111420, 2024.

\bibitem[Tsi88]{tsitsiklis1988decentralized}
J.~Tsitsiklis.
\newblock Decentralized detection by a large number of sensors.
\newblock {\em Mathematics of Control, Signals and Systems}, 1(2):167--182,
  1988.

\bibitem[Wit68]{WitsenhausenCounter}
H.~S. Witsenhausen.
\newblock A counterexample in stochastic optimum control.
\newblock {\em SIAM J. on Control and Optimization}, 6:131--147, 1968.

\bibitem[Wit75]{wit75}
H.~S. Witsenhausen.
\newblock The intrinsic model for discrete stochastic control: Some open
  problems.
\newblock {\em Lecture Notes in Econ. and Math. Syst., Springer-Verlag},
  107:322--335, 1975.

\bibitem[YAY22]{yongacoglu2022independent}
B.~Yongacoglu, G.~Arslan, and S.~Y{\"u}ksel.
\newblock Independent learning in mean-field games: Satisficing paths and
  convergence to subjective equilibria.
\newblock {\em arXiv preprint arXiv:2209.05703}, 2022.

\bibitem[YB13]{YukselBasarBook}
S.~Y\"uksel and T.~Ba\c{s}ar.
\newblock {\em Stochastic Networked Control Systems: Stabilization and
  Optimization under Information Constraints}.
\newblock Springer, New York, 2013.

\bibitem[YS17]{YukselSaldiSICON17}
S.~Y\"uksel and N.~Saldi.
\newblock Convex analysis in decentralized stochastic control, strategic
  measures and optimal solutions.
\newblock {\em SIAM Journal on Control and Optimization}, 55:1--28, 2017.

\bibitem[Y{\"u}k24]{yuksel2023borkar}
S.~Y{\"u}ksel.
\newblock On {B}orkar and {Y}oung relaxed control topologies and continuous
  dependence of invariant measures on control policy.
\newblock {\em to appear SIAM Journal on Control and Optimization (also
  arXiv:2304.07685)}, 2024.

\end{thebibliography}
\end{document}